\documentclass[12pt]{article}
\title{Visceral theories without assumptions}
\author{Will Johnson}

\usepackage{amsmath, amssymb, amsthm}    	
\usepackage{fullpage} 	
\usepackage{amscd}
\usepackage{hyperref}
\usepackage[all]{xy}
\usepackage{centernot}
\usepackage{enumitem}

\DeclareMathOperator*{\forkindep}{\raise0.2ex\hbox{\ooalign{\hidewidth$\vert$\hidewidth\cr\raise-0.9ex\hbox{$\smile$}}}}

\newcommand{\rk}{\operatorname{rk}}

\newcommand{\RCVF}{\operatorname{RCVF}}
\newcommand{\RCF}{\operatorname{RCF}}
\newcommand{\ac}{\operatorname{ac}}

\newcommand{\res}{\operatorname{res}}
\newcommand{\Aut}{\operatorname{Aut}}

\newcommand{\acl}{\operatorname{acl}}
\newcommand{\im}{\operatorname{im}}
\newcommand{\dcl}{\operatorname{dcl}}
\newcommand{\tp}{\operatorname{tp}}

\newcommand{\dom}{\operatorname{dom}}

\newcommand{\trdeg}{\operatorname{tr.deg}}
\newcommand{\bd}{\operatorname{bd}}

\newcommand{\dpr}{\operatorname{dp-rk}}

\newcommand{\ter}{\operatorname{int}}
\newcommand{\slice}{\operatorname{slice}}

\newtheorem{theorem}{Theorem}[section] 
\newtheorem{nontheorem}[theorem]{Non-Theorem}
\newtheorem{lemma}[theorem]{Lemma}

\newtheorem{corollary}[theorem]{Corollary}
\newtheorem{fact}[theorem]{Fact}

\newtheorem{proposition}[theorem]{Proposition}
\newtheorem*{theorem-star}{Theorem}
\newtheorem*{lemma-star}{Theorem}
\newtheorem*{conjecture-star}{Conjecture}

\theoremstyle{definition}
\newtheorem{definition}[theorem]{Definition}
\newtheorem{example}[theorem]{Example}

\newtheorem{remark}[theorem]{Remark}

\theoremstyle{remark}
\newtheorem{claim}[theorem]{Claim}

\newtheorem*{acknowledgment}{Acknowledgments}

\newcommand{\Qq}{\mathbb{Q}}
\newcommand{\Cc}{\mathbb{C}}
\newcommand{\eq}{\textrm{eq}}

\newcommand{\Rr}{\mathbb{R}}

\newcommand{\Zz}{\mathbb{Z}}

\newcommand{\Nn}{\mathbb{N}}

\newcommand{\Mm}{\mathbb{M}}

\newcommand{\Oo}{\mathcal{O}}

\newcommand{\ba}{{\bar{a}}}
\newcommand{\bb}{{\bar{b}}}
\newcommand{\bc}{{\bar{c}}}
\newcommand{\be}{{\bar{e}}}
\newcommand{\barf}{{\bar{f}}}

\newcommand{\bz}{{\bar{z}}}
\newcommand{\bx}{{\bar{x}}}
\newcommand{\by}{{\bar{y}}}

\newenvironment{claimproof}[1][\proofname]
               {
                 \proof[#1]
                 
               }
               {
                 \endproof
               }

\begin{document}
\maketitle


\begin{abstract}
  Let $T$ be a theory with a definable topology.  $T$ is
  \emph{t-minimal} in the sense of Mathews if every definable set in
  one variable has finite boundary.  If $T$ is t-minimal, we show that
  there is a good dimension theory for definable sets, satisfying
  properties similar to dp-rank in dp-minimal theories, with one key
  exception: the dimension of $\dom(f)$ can be less than the dimension
  of $\im(f)$ for a definable function $f$.  Using the dimension
  theory, we show that any definable field in a t-minimal theory is
  perfect.

  We then specialize to the case where $T$ is \emph{visceral} in the
  sense of Dolich and Goodrick, meaning that $T$ is t-minimal and the
  definable topology comes from a definable uniformity (i.e., a
  definable uniform structure).  We show that almost all of Dolich and
  Goodrick's tame topology theorems for visceral theories hold without
  their additional assumptions of \emph{definable finite choice} (DFC)
  and \emph{no space-filling functions} (NSFF).  Lastly, we produce an
  example of a visceral theory with a space-filling curve, answering a
  question of Dolich and Goodrick.
\end{abstract}

\section{Introduction}
Let $T$ be a theory endowed with a Hausdorff definable topology.  Then
$T$ is \emph{t-minimal} in the sense of Mathews~\cite[Definition~2.4]{mathews} if the
following property holds:
\begin{quote}
  The interior $\ter(D)$ of a definable set $D \subseteq M^1$ is
  non-empty if and only if $D$ is infinite.
\end{quote}
Equivalently,
\begin{quote}
  There are no isolated points in $M$, and the boundary $\bd(D)$ of
  any definable set $D \subseteq M^1$ is finite.
\end{quote}
If the definable topology is induced by a definable uniformity (i.e.,
a definable uniform structure), then $T$ is \emph{visceral} in the
sense of Dolich and Goodrick \cite[Definition~3.3]{viscerality}.  We review
uniformities and definable topologies/uniformities in
Section~\ref{uniformity-review} below.

The class of visceral theories is very general.  For example, it
includes the following:
\begin{enumerate}
\item O-minimal expansions of DOAG and RCF.
\item C-minimal expansions of ACVF \cite{c-source,cminfields}.
\item Henselian valued fields of residue characteristic $0$, such as
  $\Cc((t))$ and $\Qq((t))$.
\item $p$-adically closed fields such as $\Qq_p$, and more generally,
  their $P$-minimal expansions \cite{p-min}.
\item dp-minimal expansions of fields (by \cite[Theorem~1.3]{WJ}).
\end{enumerate}
At the same time, viscerality is a strong enough assumption to imply
topological tameness results.  Building off earlier work of Simon and
Walsberg \cite{simonWalsberg}, Dolich and Goodrick \cite{viscerality}
proved generic continuity and cell decomposition theorems for visceral
theories, and developed a dimension theory.  However, Dolich and
Goodrick's cell decomposition requires an additional assumption of
\emph{definable finite choice} (DFC), and the dimension theory
requires a further assumption of \emph{no space-filling functions}
(NSFF).

The goal of the present paper is to drop these assumptions, as much as
possible.  We have four main results:
\begin{enumerate}
\item Cell decomposition works without the assumption of DFC (as expected).
\item There are visceral theories in which NSFF fails.  This answers a
  question in \cite[Section~1.1]{viscerality}.
\item Dimension theory can be generalized to arbitrary visceral
  theories, though it has some pathologies when NSFF fails.
\item In fact, the dimension theory works even in (non-visceral)
  t-minimal theories, a setting where generic continuity and cell
  decomposition can fail.
\end{enumerate}
To obtain these generalizations, it is necessary to slightly weaken
the definition of ``cell'' and modify the definition of ``dimension''.

\subsection{Statement of results}
Let $\Mm$ be a monster model of a complete theory $T$, with a
Hausdorff definable topology.
\begin{definition}[{\cite[Definition~3.22]{viscerality}}]
  $T$ has \emph{definable finite choice} (DFC) if every definable
  surjection $f : X \to Y$ with finite fibers has a definable section
  $Y \to X$.
\end{definition}
DFC holds in many natural examples of visceral theories, such as
ordered theories and P-minimal theories.  But it fails in other
natural examples, like ACVF.  One of the main goals of this paper is
to remove the assumption of DFC from all the results of \cite{viscerality}.
\begin{definition}[{\cite[Definition~3.29]{viscerality}}]
  $T$ has a \emph{space-filling function} if there is a definable
  surjection $f : X \to Y$ where $X \subseteq \Mm^n$, $Y \subseteq
  \Mm^m$, $Y$ has non-empty interior, and $n < m$.  Otherwise, $T$ has
  \emph{no space-filling functions} (NSFF).
\end{definition}
The assumption of NSFF is a bit ad hoc, and again we would like to
remove it.  There are two more well-known properties which are closely
related to NSFF:
\begin{definition}
  $T$ is \emph{dp-minimal} if the following holds: if $I_1, I_2$ are
  two infinite sequences whose concatenation $I_1 + I_2$ is
  indiscernible, and if $a \in \Mm^1$, then $I_1$ or $I_2$ is
  $a$-indiscernible.
\end{definition}
\begin{definition}
  $T$ has the \emph{exchange property} if $\acl(-)$ satisfies
  the Steinitz exchange property:
  \begin{equation*}
    a \in \acl(Cb) \setminus \acl(C) \implies b \in \acl(Ca) \qquad\qquad \text{for } a,b \in \Mm^1, ~ C \subseteq \Mm^\eq.
  \end{equation*}
\end{definition}
See \cite{dpExamples} or \cite[Chapter~4]{NIPguide} for more about
dp-minimality and dp-rank.  Dp-minimality implies NSFF by properties
of dp-rank.  Similarly, the exchange property implies NSFF by
properties of acl-rank.  See \cite[Proposition~3.31]{viscerality} for
details.
\begin{definition}
  If $X \subseteq \Mm^n$ is definable, the \emph{naive topological
  dimension} $d_t(X)$ is the maximum $k$ such that $\pi(X)$ has
  non-empty interior for some coordinate coordinate projection $\pi :
  \Mm^n \to \Mm^k$.
\end{definition}
Dolich and Goodrick showed that under the
assumptions of DFC and NSFF, there is a good dimension theory using
$d_t(-)$.  They asked whether every visceral theory satisfies NSFF.
In Section~\ref{sec:rcvf3} we give a negative answer:
\begin{theorem} \label{thm-main-ctx}
  There is a visceral theory with DFC with a definable space-filling
  function (NSFF fails).
\end{theorem}
The example is not hard to describe: it is 3-sorted RCVF with an angular component map.

Given that NSFF fails, one can ask what happens to dimension theory.
One can no longer use the topological dimension $d_t$ in this case.
Indeed, a failure of NSFF immediately gives an example of a definable
bijection $f : X \to Y$ for which $d_t(X) \ne d_t(Y)$.

In Sections~\ref{tmin-dim} and Section~\ref{sec:mdt}, we develop the
following dimension theory for t-minimal and visceral theories:
\begin{theorem} \label{long-theorem}
  Suppose $T$ is t-minimal.  To every definable set $X$, we can
  assign a dimension in $\{-\infty\} \cup \Nn$ in such a way that the
  following conditions hold:
  \begin{enumerate}
  \item \label{lt1} $\dim(X) = -\infty \iff X = \varnothing$.
  \item \label{lt2} $\dim(X) \le 0$ iff $X$ is finite.
  \item \label{lt3} If $X \subseteq \Mm^n$, then $\dim(X) \le n$, with
    equality iff the interior $\ter(X)$ is non-empty.
  \item \label{lt5} $\dim(X \times Y) = \dim(X) + \dim(Y)$.
  \item \label{lt7} $\dim(X \cup Y) = \max(\dim(X),\dim(Y))$.
  \item \label{lt8-e} If $X$ and $Y$ are in definable bijection, then
    $\dim(X) = \dim(Y)$.
  \item \label{lt8} If $f : X \to Y$ is a definable injection,
    then $\dim(X) \le \dim(Y)$.    
  \item \label{lt8.5} If $f : X \to Y$ is a definable
    surjection with finite fibers, then $\dim(X) = \dim(Y)$.
  \item \label{lt11} If $\phi(x,y)$ is a formula, then
    $\dim(\phi(\Mm,b))$ depends definably on $b$, in the sense that
    $\{b : \dim(\phi(\Mm,b)) = k\}$ is definable for each $k$.
  \item \label{lt10} If $f : X \to Y$ is a
    definable function and every fiber $f^{-1}(b)$ has
    dimension $\le k$, then $\dim(X) \le k + \dim(Y)$.
  \item \label{lt23} If $T$ has the exchange property, then $\dim(X)$
    agrees with the acl-rank of $X$.
  \item \label{lt24} If $T$ is dp-minimal, then $\dim(X)$ agrees with
    the dp-rank of $X$.
  \item \label{lt20} If $T$ is visceral and has NSFF, then $\dim(X)$
    agrees with the naive topological dimension $d_t(X)$.
  \end{enumerate}
\end{theorem}
In the visceral case, there is a simple characterization of $\dim(X)$
in terms of definable injections (see Proposition~\ref{inject-dim}), but
the definition in the t-minimal case is more involved.

In the case where $T$ is visceral and dp-minimal, almost all of
Theorem~\ref{long-theorem} is proven in \cite{simonWalsberg}, with
$\dim(-)$ interpreted as dp-rank.  In the case where $T$ is visceral
with NSFF and DFC, almost all of Theorem~\ref{long-theorem} appears in
\cite{viscerality}, with $\dim(-)$ interpreted as naive topological
dimension $d_t(-)$.  Our contribution is to deal with the case in
which NSFF (or DFC) fails, and to extend everything to non-visceral
t-minimal theories.

Here is an example application of the dimension theory:
\begin{corollary}[{= Corollary~\ref{perfect-field-cor}}]
  If $(K,+,\cdot)$ is a definable field in a t-minimal theory, then
  $K$ is perfect.
\end{corollary}
Several properties are conspicuously missing from
Theorem~\ref{long-theorem}, and require additional assumptions.  First
of all, dimension may behave badly in definable surjections:
\begin{theorem}[{see Sections~\ref{sec:mdt} and \ref{sec:ndep}}] \phantomsection \label{thm-aoa}
  \begin{enumerate}
  \item Let $\Mm$ be a visceral theory with NSFF.  If $f : X
    \to Y$ is definable surjection, then $\dim(X) \ge \dim(Y)$.
  \item For any $n$, there is a visceral theory with DFC and
    a definable surjection $f : X \to Y$ such that $\dim(X) = 1$ and
    $\dim(Y) \ge n$.
  \end{enumerate}
\end{theorem}
Second, the frontier dimension inequality may or may not hold:
\begin{theorem}[{see Sections~\ref{dof}, \ref{sec:rcvf2}--\ref{sec:ndep}}] \phantomsection \label{thm-frontiers}
  \begin{enumerate}
  \item \label{tf1} Let $\Mm$ be a visceral theory with the exchange
    property.  If $X \subseteq \Mm^n$ is definable, then
    $\dim(\partial X) <\dim(X)$.
  \item Let $\Mm$ be a visceral theory with NSFF.  If $X \subseteq
    \Mm^n$ is definable, then $\dim(\partial X) \le \dim(X)$.
  \item There is a visceral theory with NSFF and a
    definable set $X$ such that $\dim(X) = \dim(\partial X) = 1$.
  \item For any $n$, there is a visceral theory and a
    definable set $X$ such that $\dim(X) = 1$ and $\dim(\partial X)
    \ge n$.
  \end{enumerate}
\end{theorem}
Theorem~\ref{thm-frontiers}(\ref{tf1}) generalizes
\cite[Corollary~3.35]{viscerality} to the case without DFC.  Simon and Walsberg
also show that (\ref{tf1}) holds when $\Mm$ is dp-minimal
\cite[Proposition~4.3]{simonWalsberg}.

In order to state the cell decomposition and generic continuity
theorems, we need some definitions from \cite{simonWalsberg}.  Let $X,
Y$ be definable sets in $\Mm$.
\begin{definition}[Section 3.1 in \cite{simonWalsberg}] \label{corresp-def}
  Let $m$ be a positive integer.
  An \emph{$m$-correspondence} from $X$ to $Y$ is a function $f$
  assigning to each element $a \in X$ a subset $f(a) \subseteq Y$ of
  size $m$.  A \emph{correspondence} is an $m$-correspondence for some $m > 0$.

  The \emph{graph} $\Gamma(f)$ of an $m$-correspondence $f : X
  \rightrightarrows Y$ is the set of $(a,b) \in X \times Y$ such that
  $b \in f(a)$.  An $m$-correspondence is \emph{definable} if its
  graph is.
\end{definition}
\begin{definition}
  An $m$-correspondence $f : X\rightrightarrows Y$ is
  \emph{continuous} if the following equivalent conditions hold:
  \begin{enumerate}
  \item $f$ looks locally like the graph of $m$ distinct continuous
    functions.  More precisely, for any $p \in X$, there is a
    neighborhood $U \ni p$ and continuous functions $g_1, \ldots, g_m
    : U \to Y$ such that $f(a) = \{g_1(a),\ldots,g_m(a)\}$ for all $a
    \in U$.
  \item If $p \in X$ and $f(p) = \{q_1, \ldots, q_m\}$, then for any
    neighborhoods $E_1 \ni q_1, \ldots, E_m \ni q_m$, there is a
    neighborhood $U \ni p$ such that for any $p' \in U$, there are
    $q'_1 \in E_1, \ldots, q'_m \in E_m$ with $f(p') =
    \{q'_1,\ldots,q'_m\}$.
  \end{enumerate}
\end{definition}
\begin{definition}
  A \emph{$k$-cell} is a definable set of the form
  $\sigma(\Gamma(f))$, where $U \subseteq \Mm^k$ is a non-empty open
  set, $f : U \rightrightarrows \Mm^{n-k}$ is a continuous
  $m$-correspondence for some $m \ge 1$, and $\sigma$ is a coordinate
  permutation.  A \emph{cell} is a $k$-cell for some $k$.
\end{definition}
It turns out that a $k$-cell $C$ has $\dim(C) = k$
(Remark~\ref{pathetic-rem}(\ref{pr1})).  Here is our main cell
decomposition result:
\begin{theorem}[{= Theorems~\ref{cd1}, \ref{cd2}}] \label{cd-intro}
  Let $T$ be visceral.
  \begin{enumerate}
  \item Every definable set $X \subseteq \Mm^n$ can be written as a
    disjoint union of cells.
  \item If $f : X \to \Mm^m$ is a definable function, then we can
    write $X$ as a disjoint union of cells $X = \coprod_{i=1}^N C_i$
    in such a way that $f \restriction C_i$ is continuous for each
    $i$.
  \item More generally, if $f : X \rightrightarrows \Mm^m$ is a
    definable correspondence, then we can write $X$ as a disjoint
    union of cells $X = \coprod_{i=1}^N C_i$ in such a way that $f
    \restriction C_i$ is continuous for each $i$.
  \end{enumerate}
\end{theorem}
When $T$ is dp-minimal, most of Theorem~\ref{cd-intro} is proven in
\cite{simonWalsberg}.  When $T$ has DFC, most of
Theorem~\ref{cd-intro} is proven in \cite{viscerality}.

On a related note, we can show that functions and correspondences are
continuous at ``most'' points, and definable sets are locally
Euclidean at ``most'' points, though ``most'' must be understood with
respect to the filter of dense sets, rather than with respect to
dimension:
\begin{theorem}[{see Section~\ref{sec:afr}}]\label{weird-intro}
  Let $T$ be visceral.  Let $D \subseteq \Mm^n$ be definable.
  \begin{enumerate}
  \item The collection of dense definable subsets of $D$ is a filter:
    if $X, Y \subseteq D$ are definable and dense in $D$, then so is
    $X \cap Y$.
  \item If $X \subseteq D$ is definable and dense, then the relative
    interior $\ter_D(X) \subseteq X$ is dense.
  \item If $X \subseteq D$ is definable and dense, then the relative
    frontier $\partial_X D = X \cap \partial D$ is the complement of a
    dense set.
  \item If $f : D \to \Mm^m$ is definable, then $f$ is continuous on a
    dense, relatively open subset of $D$.
  \item More generally, if $f : D \rightrightarrows \Mm^m$ is a
    definable correspondence, then $f$ is continuous on a dense,
    relatively open subset of $D$.
  \item $D$ is locally Euclidean on a dense, relatively open definable
    subset $D_{\mathrm{Eu}}$, in the sense that for every $p \in
    D_{\mathrm{Eu}}$, there is a definable homeomorphism between a
    neighborhood of $p$ in $D$ and an open subset of $\Mm^k$ for some
    $k$ depending on $p$.
  \end{enumerate}
\end{theorem}
In future work, we will use cell decomposition and generic continuity
to study definable groups and fields in visceral theories.  Lord
willing, we will show that if $(K,+,\cdot)$ is a definable field in a
visceral theory, then $K$ is finite or large in the sense of Pop
\cite{Pop-little}.
\begin{remark}
  The generic continuity results do not generalize to t-minimal
  theories.  For example, consider RCF with the Sorgenfrey
  topology---the topology with half-open intervals $[a,b)$ as a basis
    of open sets.  This topology is definable and makes RCF into a
    t-minimal theory.  But the definable function $f(x) = -x$ is
    nowhere continuous.  The graph $\Gamma(f) \subseteq \Mm^2$ also
    shows that cell decomposition cannot work, at least in the form
    given above for visceral theories.
\end{remark}

\subsection{Notation and conventions} \label{uniformity-review}
We use standard model-theoretic notational conventions.  ``Definable''
means definable with parameters.  We do not consider interpretable
sets to be definable, and indeed, many of the theorems fail to
generalize to interpretable sets.  By a \emph{monster model}, we mean
a structure that is $\kappa$-saturated and $\kappa$-strongly
homogeneous for some cardinal $\kappa$ much larger than any cardinals
we care about.  \emph{Small} means ``$< \kappa$''.  We reserve the
blackboard bold $\Mm$ for monster models.

The letter $\pi$ generally denotes a coordinate projection $\Mm^n \to
\Mm^k$, not necessarily onto the first $k$ coordinates, and not
necessarily with $k=1$.  We write $\pi_i : \Mm^n \to \Mm$ for the
projection onto the $i$th coordinate, and $\pi^i : \Mm^n \to
\Mm^{n-1}$ for the projection onto the other coordinates.

We sometimes write tuples with symbols like $\ba, \bb,
\bc,\bx,\by,\bz$, and other times with unmarked letters $a,b,c,x,y,z$,
depending on whether we are emphasizing their tupleness.  In
particular, $\ba$ always means a tuple $(a_1,\ldots,a_n)$, but $a$ may
or may not mean a tuple.

The notation $a := \cdots$ or $\cdots =: a$ means ``Let $a =
\cdots$''.  We use this notation to give mathematical objects names,
in passing.

If $X$ is a set in a topological space, we write the closure and
interior as $\overline{X}$ and $\ter(X)$.  We write the boundary and
frontier as follows:
\begin{gather*}
  \bd(X) = \overline{X} \setminus \ter(X) \\
  \partial X = \overline{X} \setminus X.
\end{gather*}
There are several notions of dimension in this paper.  The naive
topological dimension of a set $X \subseteq \Mm^n$ will be written
$d_t(X)$.  If $\Mm$ has the exchange property, then we write the
acl-rank as $\rk(-)$.  We will occasionally mention dp-rank $\dpr(-)$
in passing; see \cite[Chapter~4]{NIPguide} for a definition.  The main
notion of dimension in this paper, defined in
Sections~\ref{dim-sec-a}--\ref{dim-sec-b}, will be written $\dim(-)$.
In Section~\ref{frontier-sec} there will also be a ``frontier rank''
writen $d(-)$.

See Definition~\ref{corresp-def} for the definition of correspondence.
We write a correspondence $f$ from $X$ to $Y$ as $f : X
\rightrightarrows Y$.  We write the graph and image of $f$ as
\begin{align*}
  \Gamma(f) &= \{(a,b) \in X \times Y : b \in f(a)\} \\
  \im(f) &= \bigcup_{a \in X} f(a).
\end{align*}
In the rest of this section, we review the notion of
\emph{uniformities} or (\emph{uniform structures}) from topology, and
fix the relevant notation.

Let $M$ be a set.  Let $E, D$ be binary relations on $M$, that is,
subsets of $M \times M$.  Let
\begin{gather*}
  E^{-1} = \{(y,x) \mid (x,y) \in E\} \\
  D \circ E = \{(x,z) \mid \exists y \in M ~ (x,y) \in E \text{ and } (y,z) \in D\}.
\end{gather*}
The following definition is standard in topology.
\begin{definition}
  A \emph{uniformity} or \emph{uniform structure}
  on $M$ is a collection $\Omega$ of binary relations on $M$ called
  \emph{entourages} satisfying the following axioms:
  \begin{enumerate}
  \item If $E, D \in \Omega$, then $E \cap D \in \Omega$.
  \item If $E \in \Omega$ and $E \subseteq D \subseteq M \times M$,
    then $D \in \Omega$.
  \item If $E \in \Omega$, then $E$ is reflexive.
  \item If $E \in \Omega$, then $E^{-1} \in \Omega$.
  \item If $E \in \Omega$, then there is $D \in \Omega$ such that $D
    \circ D \subseteq E$.
  \end{enumerate}
\end{definition}
We will prefer the term ``uniformity'' over ``uniform structure,'' to
avoid conflict with the model-theoretic sense of ``structure.''  A
space with a uniformity is called a \emph{uniform space}.

Let $\Omega$ be a uniformity on $M$.  For any entourage $E \in
\Omega$ and any $a \in M$, let
\begin{equation*}
  E[a] =\{b \in M : (a,b) \in E\}.
\end{equation*}
The \emph{uniform topology} on $M$ is characterized by the property
that for any $a \in M$, the family of sets $\{E[a] : E \in \Omega\}$
is a neighborhood basis of $a$.

A set $\mathcal{B} \subseteq \mathcal{P}(M \times M)$ is a \emph{basis
for a uniformity} if the upward closure
\begin{equation*}
  \Omega = \{E \subseteq M \times M \mid \exists D \in \mathcal{B} : D \subseteq E\}
\end{equation*}
is a uniformity.  Elements of $\mathcal{B}$ are called \emph{basic
entourages}.  Sets of the form $E[a]$ for $E \in \mathcal{B}$ and $a
\in M$ are called \emph{balls}.  The balls around $a$ again form a
neighborhood basis.
\begin{definition}
  Let $\Mm$ be a monster model of a complete theory $T$.  A topology
  $\tau$ on $\Mm$ is \emph{definable} if there is a definable basis of
  opens, or more precisely, there is a definable family $\{B_a\}_{a \in
    D}$ such that $\{B_a : a \in D\}$ is a basis of open sets for
  $\tau$.  Similarly, a uniformity $\Omega$ is
  \emph{definable} if there is a definable basis of entourages.
\end{definition}

\begin{remark} \label{group-uniformity-7}
  Suppose $(M,\cdot)$ is a group and $\tau$ is a group topology.  Let
  $\mathcal{B}_0$ be a basis of open neighborhoods of $1$.  If
  $\mathcal{B}$ is the collection of sets of the form
  \begin{equation*}
    \{(x,y) \in M^2 : xy^{-1} \in U\} \text{ for } U \in B_0,
  \end{equation*}
  then $\mathcal{B}$ is the basis for a uniformity $\Omega$ on
  $(M,\cdot)$.  If $\tau$ is definable, then $\Omega$ is definable.

  Consequently, if $T$ is an expansion of the theory of groups and $T$
  is t-minimal with respect to some definable group topology $\tau$,
  then $T$ is visceral, not just t-minimal.  Most examples of visceral
  theories arise this way.
\end{remark}
\begin{definition} \label{def-sep}
  A uniformity is \emph{separated} if the following equivalent conditions hold:
  \begin{itemize}
  \item The topology is $T_0$: distinct points are topologically distinguishable.
  \item The topology is $T_1$: any singleton is closed.
  \item The topology is Hausdorff.
  \end{itemize}
\end{definition}
\textbf{We will assume that all topologies are Haudsorff and all
  uniformities are separated.}
\begin{remark}
  Dolich and Goodrick \cite{viscerality} do \emph{not} assume
  Hausdorffness.  It may be possible to remove the assumption of
  Hausdorffness from the proofs in this paper, but it hardly seems
  worth the trouble, given that non-$T_0$ topologies rarely arise in
  practice.
\end{remark}

\subsection{Outline}
In Section~\ref{tmin-dim}, we develop the dimension theory for
t-minimal structures.  A certain pregeometry arises naturally, but
offers no insight into the dimension theory, so we postpone its
discussion to Appendix~\ref{strange}.  In Section~\ref{vt-section}, we
apply the t-minimal dimension theory to visceral theories, proving
generic continuity and cell decomposition, and adding final touches to
the dimension theory in the visceral case.  In
Section~\ref{frontier-sec}, we consider issues related to the
dimension of frontiers, and use this to prove
Theorem~\ref{weird-intro}.  Finally, in Section~\ref{cxsec}, we
construct the various counterexamples appearing in
Theorems~\ref{thm-main-ctx}, \ref{thm-aoa}, and \ref{thm-frontiers}.

At certain points in the text, we state Non-Theorems, to emphasize
what we are \emph{not} saying.  In Appendix~\ref{guide}, we briefly
explain the logical relation between the Non-Theorems, and show that
all of the Non-Theorems fail in the counterexamples of
Section~\ref{cxsec}.

\subsection{A comment on \cite{simonWalsberg}}
After the initial draft of this paper was written, I looked closely at
\cite{simonWalsberg} and realized two things:
\begin{itemize}
\item In the current paper, I have subconsciously repeated many of the
  arguments from \cite{simonWalsberg}.
\item There is a minor error in \cite[\S2]{simonWalsberg} concerning
  ``acl-dimension'' whose fix is complicated to describe.
\end{itemize}
Between Lemmas~2.1 and Lemmas~2.2 in \cite{simonWalsberg}, the authors
 introduce a general notion of acl-dimension for definable
sets.  They first define the dimension of a tuple $\dim(\ba/C)$ to be
the smallest length of a subtuple $\bb \subseteq \ba$ with $\ba \in \acl(C\bb)$.  This is certainly well-defined.  They
then define $\dim(A) = \max_{\ba \in A} \dim(\ba/C)$ for any
$C$-definable $A$.  Contrary to the claims in \cite{simonWalsberg},
this is \emph{not} well-defined---$\dim(A)$ depends on the choice of
the set $C$.  See Remark~\ref{bane} below for an example.

Fortunately, this does not affect the results of \cite{simonWalsberg}:
\begin{itemize}
\item First of all, we are going to prove that $\dim(A)$ \emph{is}
  well-defined in t-minimal theories such as those considered in
  \cite{simonWalsberg}.  See \S\ref{tmin-dim} below, especially
  Proposition~\ref{extend-2}.
\item More directly, one can repair the arguments in
  \cite[\S2]{simonWalsberg} as follows.  Replace every application of
  ``acl-dimension'' with ``dp-rank'' up to the proof of
  \cite[Proposition~2.4]{simonWalsberg}, showing that dp-rank and
  acl-dimension agree.  Unfortunately, there are many details to check.
\end{itemize}
In light of this situation, we will avoid citing \cite{simonWalsberg}
until we have the dimension theory in hand.  The arguments which
closely follow \cite{simonWalsberg} will be maintained, in order to
be self-contained.

\section{Dimension in t-minimal theories} \label{tmin-dim}
In this section, we develop the dimension theory for t-minimal
theories.  Section~\ref{sec:bs} reviews the machinery of \emph{broad}
and \emph{narrow} sets from \cite[Section~3.1]{prdf1a}.
Section~\ref{td:sec} applies this machinery to study the condition
``$X \subseteq \Mm^n$ has non-empty interior'', which will eventually
be equivalent to ``$\dim(X) = n$'', i.e., $X$ has maximum dimension.
Section~\ref{tech:sec} proves a key technical fact which ensures that
dimension is well-defined: if $\ba$ and $\bb$ are interalgebraic
$\acl$-independent tuples, then $\ba$ has the same length as $\bb$.
Using this, we develop dimension theory for complete types
$\dim(\ba/C)$ in Section~\ref{dim-sec-a}, and for more general
type-definable sets $\dim(X)$ in Section~\ref{dim-sec-b}.
\subsection{Broad sets} \label{sec:bs}
The following definition is from \cite[Definition~3.1]{prdf1a}, but
the idea was implicit in \cite{surprise}.
\begin{definition}
  Let $X_1, \ldots, X_n$ be infinite sets and let $R \subseteq
  \prod_{i = 1}^n X_i$ be a subset.  Then $R$ is \emph{broad} if the
  following holds: for any $k < \omega$, there are subsets $S_i
  \subseteq X_i$ with $|S_i| \ge k$ and $R \supseteq \prod_{i = 1}^n
  S_i$.  Otherwise, $R$ is \emph{narrow}.
\end{definition}
In other words, a set is broad if it contains an $n$-dimensional
$k$-by-$k$-by-$\cdots$-by-$k$ grid, for any $k$.  When $n = 1$,
``broad'' and ``narrow'' mean ``infinite'' and ``finite.''
\begin{remark}
  If we consider the multi-sorted structure $(X_1,\ldots,X_n;R)$ and
  if $(X_1',\ldots,X_n';R')$ is an elementary extension, then $R$ is
  broad in $\prod_{i = 1}^n X_i$ if and only if $R'$ is broad in
  $\prod_{i = 1}^n X'_i$.  (Broadness can be expressed as a
  conjunction of first-order sentences.)

  If $(X_1,\ldots,X_n;R)$ is $\aleph_0$-saturated, then $R$ is broad
  if and only if there are infinite subsets $S_i \subseteq X_i$ such
  that $R \supseteq \prod_{i = 1}^n S_i$.  In other words, $R$ is
  broad iff $R$ contains an
  $\omega$-by-$\omega$-by-$\cdots$-by-$\omega$ grid.
\end{remark}
\begin{fact} \label{broad-narrow}
  Narrow sets form an ideal on $\prod_{i = 1}^n X_i$.  In other words,
  \begin{enumerate}
  \item If $R' \subseteq R$ and $R$ is narrow, then $R'$ is narrow.
  \item If $R, R'$ are narrow, then $R \cup R'$ is narrow.
  \end{enumerate}
  Dually,
  \begin{enumerate}
  \item If $R \subseteq R' \subseteq \prod_{i = 1}^n X_i$ and $R$ is
    broad, then $R'$ is broad.
  \item If $R, R' \subseteq \prod_{i = 1}^n X_i$ and $R \cup R'$ is
    broad, then $R$ or $R'$ is broad.
  \end{enumerate}
\end{fact}
(1) is trivial, and we sketch the proof of (2) in order to be
self-contained:
\begin{proof}[Proof sketch]
  Suppose $R \cup R'$ is broad.  Passing to an elementary extension,
  we may assume that $(X_1,\ldots,X_n;R,R')$ is highly saturated.
  Then $R \cup R'$ contains an $\omega \times \cdots \times \omega$
  grid.  Take $n = 3$ for simplicity.  Then there exist distinct $a_1,
  a_2, \ldots \in X_1$, distinct $b_1, b_2, \ldots \in X_2$, and
  distinct $c_1, c_2, \ldots \in X_3$, such that
  \begin{equation*}
    \{a_1,a_2,\ldots\} \times \{b_1,b_2,\ldots\} \times
    \{c_1,c_2,\ldots\} \subseteq R \cup R'.
  \end{equation*}
  That is, $(a_i,b_j,c_k) \in R \cup R'$ for any $i, j, k$.  Let
  $\{\alpha_i\beta_i\gamma_i\}_{i < \omega}$ be an indiscernible
  sequence extracted from $\{a_ib_ic_i\}_{i < \omega}$.  Then the
  $\alpha_i$ are distinct, the $\beta_i$ are distinct, the $\gamma_i$
  are distinct, and $(\alpha_i,\beta_j,\gamma_k) \in R \cup R'$ for
  any $i < j < k$.  Swapping $R$ and $R'$, we may assume
  $(\alpha_0,\beta_1,\gamma_2) \in R$.  Then
  $(\alpha_i,\beta_j,\gamma_k) \in R$ for any $i < j < k$.  Therefore,
  for any $n$, we have
  \begin{equation*}
    \{\alpha_0,\ldots,\alpha_{n-1}\} \times \{\beta_n,\ldots,\beta_{2n-1}\} \times \{\gamma_{2n}, \ldots, \gamma_{3n-1}\} \subseteq R,
  \end{equation*}
  and $R$ is broad.
\end{proof}
Now fix a monster model $\Mm$ with a t-minimal topology $\tau$.  Say
that $D \subseteq \Mm^n$ is broad or narrow if it is broad or narrow
as a subset of $\underbrace{\Mm \times \cdots \times \Mm}_{\text{$n$
    times}}$.  Say that a formula or small type is broad or narrow if
its set of realizations is broad or narrow.
\begin{fact} \label{bn}
  Let $X \subseteq \Mm^n$ be type-definable.
  \begin{enumerate}
  \item \label{bn1} $X \subseteq \Mm^n$ is broad iff there are sets
    $S_1, \ldots, S_n \subseteq \Mm$ of size $\aleph_0$ such that $X
    \supseteq \prod_{i=1}^n S_i$.
  \item \label{bn2} If $X$ is a small filtered intersection
    $\bigcap_{i \in I} X_i$ of definable sets, then $X$ is broad iff
    every $X_i$ is broad.  Equivalently, a partial $n$-type
    $\Sigma(\bx)$ is broad iff every finite subtype $\Sigma_0(\bx)$ is
    broad.
  \item \label{bn3} If $X$ is type-definable over a small set $C$,
    then $X$ is broad iff $\tp(\ba/C)$ is broad for some $\ba \in X$.
    Equivalently, a partial $n$-type $\Sigma(\bx)$ over $C$ is broad
    iff some completion $p \in S_n(C)$ is broad.
  \item \label{bn4} If $X$ is broad and $Y \subseteq \Mm^m$ is broad,
    then $X \times Y \subseteq \Mm^{n+m}$ is broad.
  \item \label{bn5} Let $\pi^i : \Mm^n \to \Mm^{n-1}$ be the
    coordinate projection onto the coordinates other than the $i$th
    coordinate.  If $\pi^i : X \to \Mm^{n-1}$ has finite fibers, then
    $X$ is narrow.
  \item \label{bn6} If $\tp(\ba\bb/C)$ is broad, then $\tp(\ba/C\bb)$
    is broad.
  \item \label{bn7} If $a \in \Mm^1$, then $\tp(a/C)$ is broad iff $a
    \notin \acl(C)$.
  \end{enumerate}
\end{fact}
Again, we sketch the proof:
\begin{proof}[Proof sketch]
  Parts (\ref{bn1}) and (\ref{bn2}) hold by compactness or saturation.
  In Part (\ref{bn3}), if $\Sigma(\bx)$ is a broad partial type over
  $C$, then we can find a maximal broad partial type $p(\bx) \supseteq
  \Sigma(\bx)$ over $C$, by Zorn's lemma and Part (\ref{bn2}).
  Because narrow sets are an ideal, $p(\bx)$ must be a complete type
  over $C$.  Parts (\ref{bn4}) and (\ref{bn5}) follow easily from Part
  (\ref{bn1}).

  For Part (\ref{bn6}), the fact that $\tp(\ba\bb/C)$ is broad means
  that there are infinite sets $S_1,\ldots,S_n$ and $U_1,\ldots,U_m$ such
  that every $(\be,\barf) \in \prod_{i=1}^n S_i \times \prod_{i=1}^m
  U_m$ realizes $\tp(\ba\bb/C)$.  Moving by an automorphism, we may
  assume $\bb \in \prod_{i=1}^m U_m$.  For any $\be \in \prod_{i=1}^n
  S_i$, we have $(\ba,\bb) \in \prod_{i=1}^n S_i \times \prod_{i=1}^m
  U_m$, so $\be\bb \equiv_C \ba\bb$, or equivalently, $\be \models
  \tp(\ba/C\bb)$.  Then the product $\prod_{i=1}^n S_i$ shows
  $\tp(\ba/C\bb)$ is large.

  Finally, Part (\ref{bn7}) is clear.
\end{proof}

\subsection{Top dimension} \label{td:sec}
Fix a monster model $\Mm$ of a complete theory $T$, t-minimal with
respect to some fixed definable topology $\tau$.  By t-minimality,
every definable set $D \subseteq \Mm$ with empty interior is finite.
In particular, the boundary $\bd(D)$ of any definable $D \subseteq
\Mm$ is finite, because it's the union of the two definable sets
$\overline{D} \setminus D$ and $D \setminus \ter(D)$, both of which
lack interior.

\begin{definition}
  If $X \subseteq \Mm^n$ is a set and $1 \le i \le n$, then $\ter_i X$
  denotes the interior of $X$ along the $i$th coordinate axis: a tuple
  $(a_1,\ldots,a_n)$ is in $\ter_i X$ if and only if there is a
  neighborhood $U$ of $a_i$ such that
  \begin{equation*}
    \forall x \in U : (a_1,\ldots,a_{i-1},x,a_{i+1},\ldots,a_n) \in X.
  \end{equation*}
\end{definition}
\begin{lemma} \label{orangutan}
  If $D$ is broad, then $\ter_i D$ is broad.
\end{lemma}
\begin{proof}
  Let $\pi^i : \Mm^n \to \Mm^{n-1}$ be the projection which forgets
  the $i$th coordinate.  Then $\pi^i : D \setminus \ter_i D \to
  \Mm^{n-1}$ has finite fibers by t-minimality, and so $D \setminus
  \ter_i D$ is narrow.  Then $\ter_i D$ must be broad.
\end{proof}
\begin{proposition} \label{triad}
  If $D \subseteq \Mm^n$ is definable, then the following are
  equivalent:
  \begin{enumerate}
  \item $\ter D \ne \varnothing$.
  \item $D$ is broad.
  \item $\ter_1 \ter_2 \ter_3 \cdots \ter_n D \ne \varnothing$.
  \end{enumerate}
\end{proposition}
\begin{proof}
  Proceed by induction on $n$.  When $n = 1$, conditions (1) and (3)
  are identical, and t-minimality gives (1)$\iff$(2).  Suppose $n >
  1$.
  \begin{description}
  \item[$(1)\implies(2)$:] If $\ter D \ne \varnothing$, then $D$
    contains a product $B_1 \times \cdots \times B_n$ where each $B_i$
    is a non-empty basic open set in $\Mm$.  Then $D$ is broad.
  \item[$(2)\implies(3)$:] By $n$ repeated applications of
    Lemma~\ref{orangutan}, the set $\ter_1 \ter_2 \cdots \ter_n D$ is
    broad, hence non-empty.
  \item[$(3)\implies(1)$:] If $X \subseteq \Mm^n$ and $b \in \Mm$, let
    $\slice_b X$ denote $\{\ba \in \Mm^{n-1} : (\ba,b) \in X\}$.  Note
    that $\slice_b$ commutes with $\ter_i$ for $i < n$, in the sense
    that
    \begin{equation*}
      \slice_b \ter_i X = \ter_i \slice_b X \text{ for $1 \le i < n$
        and $b \in \Mm$ and $X \subseteq \Mm^n$.}
    \end{equation*}
    Take some $(a_1,\ldots,a_{n-1},b) \in \ter_1 \ter_2 \cdots \ter_n
    D$.  Then
    \begin{align*}
      \ba = (a_1,\ldots,a_{n-1}) &\in \slice_b \ter_1 \ter_2 \cdots \ter_n D \\
      &= \ter_1 \slice_b \ter_2 \cdots \ter_n D \\
      &= \cdots \\
      &= \ter_1 \ter_2 \cdots \ter_{n-1} \slice_b \ter_n D.
    \end{align*}
    Letting $D' = \slice_b \ter_n D \subseteq \Mm^{n-1}$, we see that
    \begin{equation*}
      \ter_1 \ter_2 \cdots \ter_{n-1} D' \ne \varnothing.
    \end{equation*}
    By induction (specifically $(3)\implies(2)$), $D'$ is broad.
    Therefore $D'$ contains a product $\prod_{i=1}^{n-1} S_i$ for some
    subsets $S_1,\ldots,S_{n-1} \subseteq \Mm$ of size $\aleph_0$.
    For every $\bc \in \prod_{i=1}^{n-1} S_i$, we have
    \begin{gather*}
      \bc \in D' = \slice_b \ter_n D \\
      (\bc,b) \in \ter_n D.
    \end{gather*}
    So there is a basic neighborhood $U_{\bc} \ni b$ such that
    $\{\bc\} \times U_{\bc} \subseteq D$.  By saturation, there is a
    basic neighborhood $U \ni b$ such that $U \subseteq U_{\bc}$ for
    every $\bc$ in the small set $\prod_{i=1}^{n-1} S_i$.  Then
    \begin{equation*}
      \{\bc\} \times U \subseteq \{\bc\} \times U_{\bc} \subseteq D, \qquad \qquad
      \text{for any } \bc \in \prod_{i=1}^{n-1} S_i.
    \end{equation*}
    Let $X \subseteq \Mm^{n-1}$ be the definable set of $\bc \in
    \Mm^{n-1}$ such that $\{\bc\} \times U \subseteq D$.  Then
    \begin{equation*}
      \prod_{i=1}^{n-1} S_i \subseteq X.
    \end{equation*}
    Therefore $X$ is broad.  By induction (specifically
    $(2)\implies(1)$), $X$ has non-empty interior, containing a
    non-empty open set $V$.  For any $\bc \in V$, we have $\bc \in X$
    and so $\{\bc\} \times U \subseteq D$.  In other words, $V \times
    U \subseteq D$, and so $\ter(D) \ne \varnothing$.  \qedhere
  \end{description}
\end{proof}

\begin{corollary} \label{cor-ideal}
  The collection of definable sets in $\Mm^n$ with empty interior is
  an ideal: if $D_1, D_2$ are definable sets with empty interior, then
  $D_1 \cup D_2$ has empty interior.
\end{corollary}
\begin{proof}
  Narrow sets are an ideal.
\end{proof}
\begin{corollary} \label{small-boundary-0}
  If $D \subseteq \Mm^n$ is definable, then $\bd(D)$ has empty
  interior.
\end{corollary}
\begin{proof}
  $\bd(D)$ is a union of the two definable sets $D \setminus \ter(D)$
  and $\overline{D} \setminus D$, which each have empty interior.
\end{proof}
\begin{corollary} \label{small-closure}
  If $D \subseteq \Mm^n$ is narrow (or equivalently, has empty
  interior), then $\overline{D}$ is narrow (or equivalently, has empty
  interior).
\end{corollary}
\begin{proof}
  $\overline{D} = D \cup \bd(D)$.
\end{proof}
Recall that the naive topological dimension $d_t(X)$ of a definable
set $X$ is the maximum $k$ such that there is a coordinate projection
$\pi : \Mm^n \to \Mm^k$ such that the image $\pi(X)$ has non-empty
interior.
\begin{corollary} \label{dt-union-max}
  $d_t(X \cup Y) = \max(d_t(X),d_t(Y))$.
\end{corollary}
\begin{proof}
  If $\pi : \Mm^n \to \Mm^k$ is a coordinate projection, then $\pi(X
  \cup Y)$ has non-empty interior iff $\pi(X)$ or $\pi(Y)$ has
  non-empty interior.
\end{proof}
\begin{corollary} \label{notop}
  The ideal of sets with empty interior does not depend on the
  topology $\tau$.  Neither does the naive topological dimension
  $d_t(X)$ of definable sets $X$.  More precisely, if $\tau'$ is
  another t-minimal definable topology on $\Mm$, and $X \subseteq
  \Mm^n$ is definable, then,
  \begin{enumerate}
  \item $X$ has non-empty interior with respect to $\tau$ iff $X$ has
    non-empty interior with respect to $\tau'$.
  \item The naive topological dimension $d_t(X)$ with respect to
    $\tau$ equals the naive topology dimension $d_t(X)$ with respect
    to $\tau'$.
  \end{enumerate}
\end{corollary}
\begin{proof}
  The ideal of narrow sets is defined without reference to the
  topology.  The naive dimension $d_t(X)$ is the largest $k$ such
  that $\pi(X)$ is broad for some coordinate projection $\pi : \Mm^n
  \to \Mm^k$.  Again, this doesn't depend on the topology.
\end{proof}
\begin{corollary} \label{worthless-cells}
  Let $C \subseteq \Mm^{\eq}$ be a set of parameters.  Let $D
  \subseteq \Mm^n$ be $C$-definable.
  \begin{enumerate}
  \item If $D$ is narrow, we can write $D$ as a disjoint union of
    $C$-definable sets \[ D = D_1 \sqcup D_2 \sqcup \cdots \sqcup
    D_n\] such that the projection $\pi^i : D_i \to \Mm^{n-1}$ has
    finite fibers.
  \item In general, we can write $D$ as a disjoint union of
    $C$-definable sets \[ D = D_0 \sqcup D_1 \sqcup \cdots \sqcup
    D_n\] such that $D_0$ is open, and the projection $\pi^i : D_i \to
    \Mm^{n-1}$ has finite fibers for $1 \le i \le n$.
  \end{enumerate}
\end{corollary}
\begin{proof}
  \begin{enumerate}
  \item Since $D$ is narrow, the set $\ter_1 \ter_2 \cdots \ter_n D$
    is empty by Proposition~\ref{triad}.  For $0 \le i \le n$ let
    \begin{equation*}
      X_i = \ter_{i+1} \ter_{i+2} \cdots \ter_n D.
    \end{equation*}
    Then $X_{n-1} = \ter_n X_n \subseteq X_n$ and we have a descending
    chain of $C$-definable sets
    \begin{equation*}
      D = X_n \supseteq X_{n-1} \supseteq \cdots \supseteq X_1
      \supseteq X_0 = \varnothing.
    \end{equation*}
    Letting $D_i = X_i \setminus X_{i-1} = X_i$, we see that $D$ is a
    disjoint union $D_1 \sqcup D_2 \sqcup \cdots \sqcup D_n$.  Since
    $D_i = X_i \setminus \ter_i X_i$, the projection $\pi^i : D_i \to
    \Mm^{n-1}$ has finite fibers by t-minimality.
  \item Let $D_0 = \ter D$.  Then $D \setminus D_0$ has empty
    interior, so it is narrow by Proposition~\ref{triad}.  Use the
    previous point to decompose $D \setminus D_0$ into a disjoint
    union of $C$-definable sets with finite-fiber projections.
    \qedhere
  \end{enumerate}
\end{proof}
\begin{definition}
  If $C \subseteq \Mm^{\eq}$ is a set of parameters and $\ba \in
  \Mm^n$ is a tuple, then $\ba$ is \emph{$\acl$-independent over $C$}
  if \[a_i \notin \acl(C \cup \{a_1,\ldots,\widehat{a_i},\ldots,a_n\})
  \text{ for each $i$,}\] where the $\widehat{\text{hat}}$ indicates
  omission.
\end{definition}
If the exchange property holds (for $\acl(-)$), then this agrees with
the usual notion of independence.
\begin{proposition} \label{broad-acl}
  If $C \subseteq \Mm^{\eq}$ is a set of parameters and $\ba \in
  \Mm^n$ is a tuple, then $\tp(\ba/C)$ is broad if and only if $\ba$
  is $\acl$-independent over $C$.
\end{proposition}
\begin{proof}
  First suppose $\tp(\ba/C)$ is broad.  By
  Fact~\ref{bn}(\ref{bn6},\ref{bn7}), $\tp(a_1/Ca_2,\ldots,a_n)$ is
  broad and $a_1 \notin \acl(Ca_2a_3 \cdots a_n)$.  By symmetry, $a_i
  \notin \acl (C a_1 a_2 \cdots \widehat{a_i} \cdots a_n)$ for any
  $i$, so $\ba$ is $\acl$-independent over $C$.


  Conversely, suppose that $\tp(\ba/C)$ is narrow.  Then $\ba \in D$
  for some narrow $C$-definable set $D \subseteq \Mm^n$.  By
  Corollary~\ref{worthless-cells}, we can write $D$ as a disjoint
  union
  \begin{equation*}
    D = D_1 \sqcup \cdots \sqcup D_n
  \end{equation*}
  where each set $D_i$ is $C$-definable, and $\pi^i : D_i \to
  \Mm^{n-1}$ has finite fibers.  Then $\ba \in D_i$ for some $i$.  The
  fact that $\pi^i : D_i \to \Mm^{n-1}$ has finite fibers implies that
  $a_i \in \acl(C \cup \{a_1,\ldots,\widehat{a_i},\ldots,a_n\})$.  The
  tuple $\ba$ fails to be $\acl$-independent over $C$.
\end{proof}
\begin{corollary} \label{extend-1}
  Let $C \subseteq C' \subseteq \Mm^{\eq}$ be small sets of
  parameters, and suppose $\ba \in \Mm^n$ is $\acl$-independent over
  $C$.  Then there is $\sigma \in \Aut(\Mm/C)$ such that $\sigma(\ba)$
  is $\acl$-independent over $C'$.
\end{corollary}
\begin{proof}
  The analogous fact holds for broad complete types.  Specifically, if
  $p \in S_n(C)$ is a broad complete type, then $p$ is a broad partial
  type over $C'$, so $p$ can be extended to a broad complete type $p'
  \in S_n(C)$.
\end{proof}
\begin{remark} \label{bane}
  Corollary~\ref{extend-1} really does use t-minimality---there are
  non-t-minimal theories in which it fails.  For example, let $(V,E)$
  be a countable infinitely branching tree.  That is, $(V,E)$ is a
  countable, connected, acyclic graph in which every vertex has degree
  $\aleph_0$.  ($E$ is the adjacency relation on $V$, regarded as a
  subset of $V^2$.)  Let $R(x,y,z)$ be the ternary relation expressing
  that $y$ is on the path from $x$ to $z$.

  Consider the structure $(V,E,R)$.  Fix two adjacent vertices $a,b
  \in V$.  Using automorphisms, one sees that $(a,b)$ is
  $\acl$-independent.  Let $f : V \setminus \{a\} \to V$ be the
  $\{a\}$-definable function moving each point $x$ one step closer to $a$:
  \begin{gather*}
    f(x) \mathrel{E} x \\
    R(x,f(x),a).
  \end{gather*}
  Then $f$ is $a$-definable.  Note that
  \begin{equation*}
    E = \{(x,f(x)) : x \ne a\} \cup \{(f(x),x) : x \ne a\}.
  \end{equation*}
  It follows that $\tp(a,b)$ has no $\acl$-independent extension to an
  $\acl$-independent type $p \in S_2(a)$, since $p(x,y)$ would
  necessarily include one of the formulas $x = f(y)$ or $y = f(x)$.

  This example shows that the ``acl-dimension'' of
  \cite[\S2]{simonWalsberg} is ill-defined.  The definable set $E =
  \{(x,y) \in V^2 : x \mathrel{E} y\}$ has acl-dimension 2 as a set
  defined over $\varnothing$, but acl-dimension 1 as a set defined
  over any larger set of parameters.
\end{remark}

We will also need the following variant of Corollary~\ref{extend-1}.
\begin{lemma} \label{indep-1}
  Let $\ba, \bb$ be two tuples, each of which is $\acl$-independent
  over a small set $C \subseteq \Mm^\eq$.  Then there is $\sigma \in
  \Aut(\Mm/C)$ such $(\sigma(\ba),\bb)$ is $\acl$-independent over
  $C$.
\end{lemma}
\begin{proof}
  By Proposition~\ref{broad-acl}, $\tp(\ba/C)$ and $\tp(\bb/C)$ are broad.
  Let $X$ and $Y$ be the set of realizations of $\tp(\ba/C)$ and
  $\tp(\bb/C)$.  Then $X$ and $Y$ are broad.  Their product $X \times
  Y$ is broad and type-definable over $C$, so there is $(\ba',\bb')
  \in X \times Y$ such that $\tp(\ba',\bb'/C)$ is broad.  Since $\bb'
  \equiv_C \bb$, we can move $(\ba',\bb')$ by an automorphism and
  arrange $\bb' = \bb$.  Then $(\ba',\bb)$ is $\acl$-independent over
  $C$.
\end{proof}

\subsection{Dimension is well-defined} \label{tech:sec}
The following proposition, whose proof is technical, is the key to
developing dimension theory.  Specifically, it will show that the
$\dim(\ba/C)$ in Section~\ref{dim-sec-a} is well-defined.
\begin{proposition} \label{tech-prop}
  For each $n \ge 1$, the following properties hold.
\begin{description}
\item[$(\mathbf{A}_n)$:] There is no definable relation $R \subseteq
  \Mm^n \times \Mm^{n-1}$ such that the projections of $R$ onto
  $\Mm^n$ and $\Mm^{n-1}$ have finite fibers, and the projection of
  $R$ onto $\Mm^n$ has broad image.
\item[$(\mathbf{B}_n)$:] If $C \subseteq \Mm^\eq$ is a set of parameters
  and $\ba \in \Mm^n$ and $\tp(\ba/C)$ is broad, then $\ba$ is not
  interalgebraic over $C$ with any tuple in $\Mm^{n-1}$.
\item[$(\mathbf{C}_n)$:] If $C \subseteq \Mm^\eq$ is a set of parameters
  and $\ba \in \Mm^n$ and $\tp(\ba/C)$ is broad and $\ba$ is
  interalgebraic over $C$ with a tuple $\bb \in \Mm^n$, then
  $\tp(\bb/C)$ is broad.
\end{description}
\end{proposition}
We will prove these statements jointly by induction on $n$, using a
series of lemmas.
\begin{remark}
  $(\mathbf{B}_1)$ holds, because a transcendental element $a \notin
  \acl(C)$ cannot be interalgebraic over $C$ with the empty tuple.
  Similarly, $(\mathbf{C}_1)$ holds, because a transcendental element $a
  \notin \acl(C)$ cannot be interalgebraic over $C$ with an algebraic
  element $b \in \acl(C)$.
\end{remark}
There are several other easy implications:
\begin{lemma}
  $(\mathbf{A}_n) \iff (\mathbf{B}_n)$ for any $n$.
\end{lemma}
\begin{proof}
  If $(\mathbf{A}_n)$ fails, take a counterexample $R \subseteq \Mm^n
  \times \Mm^{n-1}$.  Let $C$ be a small set defining $R$.  Let $X
  \subseteq \Mm^n$ be the image of $R$ in $\Mm^n$.  Then $X$ is
  $C$-definable and broad, so there is $\ba \in X$ such that
  $\tp(\ba/C)$ is broad.  Take $\bb \in \Mm^{n-1}$ such that
  $(\ba,\bb) \in R$.  The fact that the projections $R \to \Mm^n$ and
  $R \to \Mm^{n-1}$ have finite fibers imply that $\ba$ and $\bb$ are
  interalgebraic over $C$, contradicting statement $(\mathbf{B}_n)$.

  If $(\mathbf{B}_n)$ fails, take a counterexample $(\ba,\bb,C)$ so that
  $\ba \in \Mm^n$, $\bb \in \Mm^{n-1}$, $\tp(\ba/C)$ is broad, and
  $\acl(C \ba) = \acl(C \bb)$.  Then there is a $C$-definable relation
  $R \subseteq \Mm^n \times \Mm^{n-1}$ witnessing the
  interalgebraicity, in the sense that $(\ba,\bb) \in R$ and the
  projections of $R$ onto $\Mm^n$ and $\Mm^{n-1}$ have finite fibers.
  If $X$ denotes the projection of $R$ onto $\Mm^n$, then $X$ is broad
  because $X$ is $C$-definable and contains the point $\ba$ with
  $\tp(\ba/C)$ being broad.
\end{proof}
\begin{lemma}
  $(\mathbf{B}_n) \implies (\mathbf{C}_n)$ for any $n$.
\end{lemma}
\begin{proof}
  If $(\mathbf{C}_n)$ fails, then there are $\ba, \bb \in \Mm^n$ such
  that $\acl(C \ba) = \acl(C \bb)$, but $\tp(\ba/C)$ is broad and
  $\tp(\bb/C)$ is narrow.  By Proposition~\ref{broad-acl}, $\bb$ is not
  $\acl$-independent over $C$.  Without loss of generality, $b_n \in
  \acl(C b_1 \cdots b_{n-1})$.  Then $\ba$ is interalgebraic with
  $(b_1,\ldots,b_{n-1})$, contradicting $(\mathbf{B}_n)$.
\end{proof}
The last implication is much less trivial:
\begin{lemma}
  $(\mathbf{C}_{n-1}) \implies (\mathbf{A}_n)$ for any $n > 1$.
\end{lemma}
\begin{proof}
  Suppose $(\mathbf{A}_n)$ fails, witnessed by a definable relation $R
  \subseteq \Mm^n \times \Mm^{n-1}$.  Let $k < \infty$ be a bound on
  the size of the fibers of the projections $R \to \Mm^n$ and $R \to
  \Mm^{n-1}$.  Let $\dom(R) \subseteq \Mm^n$ be the projection of $R$
  onto $\Mm^n$, which we assume is broad.

  Say that a point $a \in \Mm^n$ is \emph{linked} to a point $b \in
  \Mm^{n-1}$ if $a \mathrel{R} b$, that is, $(a,b) \in R$.  Every
  point in $\Mm^n$ is linked to at most $k$ points in $\Mm^{n-1}$, and
  every point in $\Mm^{n-1}$ is linked to at most $k$ points in
  $\Mm^n$.

  If $a = (a_1,\ldots,a_n) \in \Mm^n$ and $1 \le i \le n$, then let
  $H(a,i)$ be the $i$th coordinate hyperplane through $a$:
  \begin{equation*}
    H(a,i) = \{(x_1,\ldots,x_n) \in \Mm^n : x_i = a_i\}.
  \end{equation*}
  Say that $b$ is \emph{linked to} $H(a,i)$ if $b$ is linked to some
  $a' \in H(a,i)$.

  A \emph{special triple} is a triple $(a,b,U)$ where $a \in \Mm^n$,
  $b \in \Mm^{n-1}$, $a$ is linked to $b$, $U$ is an open neighborhood
  of $b$, and every point $b' \in U$ is linked to all the coordinate
  hyperplanes $H(a,1),\ldots,H(a,n)$.  Note that if $(a,b,U)$ is
  special, and $U'$ is a smaller open neighborhood of $b$, then
  $(a,b,U')$ is special.  Say that $b \in \Mm^{n-1}$ is a
  \emph{special point} if it is part of a special triple $(a,b,U)$.
  The set of special points is definable.
  \begin{claim}
    The set of special points is broad in $\Mm^{n-1}$.
  \end{claim}
  \begin{claimproof}
    Take a small set of parameters $C \subseteq \Mm^\eq$ defining $R$.
    Then $\dom(R)$ and the set of special points are both
    $C$-definable.  It suffices to find one special point $b$ such
    that $\tp(b/C)$ is broad.  Since $\dom(R)$ is broad, we can find
    some $a = (a_1,\ldots,a_n) \in \dom(R)$ such that $\tp(a/C)$ is
    broad.  Take any $b = (b_1,\ldots,b_{n-1}) \in \Mm^{n-1}$ linked
    to $a$.  Note that $a$ and $b$ are interalgebraic over $C$.  The
    fact that $\tp(a/C)$ is broad implies by Fact~\ref{bn}(\ref{bn6})
    that
    \begin{equation*}
      \tp(\pi^i(a)/Ca_i) \text{ is broad,
        for any } 1 \le i \le n,
    \end{equation*}
    where $\pi^i(a) = (a_1,\ldots,\widehat{a_i},\ldots,a_n)$ as
    before.  Note that $\pi^i(a)$, $a$, and $b$ are interalgebraic
    over $Ca_i$.  By statement $(\mathbf{C}_{n-1})$, the fact that
    $\tp(\pi^i(a)/Ca_i)$ is broad implies that
    \begin{equation*}
      \tp(b/Ca_i) \text{ is broad, for any } 1 \le i \le n.  \tag{$\ast$}
    \end{equation*}
    A fortiori, $\tp(b/C)$ is broad.  It remains to show that $b$ is a
    special point.  For each $i$, let $G_i$ be the set of points $b'
    \in \Mm^{n-1}$ which are linked to the coordinate hyperplane
    $H(a,i)$.  Note the following:
    \begin{itemize}
    \item $b \in G_i$, since $b$ is linked to $a \in H(a,i)$.
    \item The set $G_i$ is $Ca_i$-definable, since $R$ is
      $C$-definable and $H(a,i)$ is $a_i$-definable.
    \item The set $G_i \setminus \ter G_i$ is also $C a_i$-definable.
    \item The set $G_i \setminus \ter G_i$ has empty interior, so it
      is narrow by Proposition~\ref{triad}.
    \item $b \notin G_i \setminus \ter G_i$ by the previous two points
      and ($\ast$).
    \end{itemize}
    By the first and last points, we see that $b \in \ter G_i$, for
    any $i$.  Let $U = \bigcap_{i=1}^n \ter G_i$.  Then $(a,b,U)$ is a
    special triple, so $b$ is a special point.
  \end{claimproof}
  By the Claim and Proposition~\ref{triad}, the set of special points has
  non-empty interior.  Take a non-empty open set $U_0$ contained in
  the set of special points.  Let $N = k^{n+1} + 1$.  Recursively build a
  sequence of special triples $(a_j,b_j,U_j)$ for $j = 1, \ldots, N$
  with
  \begin{gather*}
    U_0 \supseteq U_1 \supseteq \cdots \supseteq U_N \\ b_j \ne b_\ell
    \text{ for } j \ne \ell
  \end{gather*}
  as follows:
  \begin{itemize}
  \item Take $b_j$ to be any point in $U_{j-1} \setminus
    \{b_1,\ldots,b_{j-1}\} \subseteq U_0$.  Such a $b_j$ exists
    because there are no isolated points.  The point $b_j$ is a
    special point, by choice of $U_0$.
  \item Take a special triple $(a_j,b_j,V)$ witnessing that $b_j$ is a
    special point.  Let $U_j = V \cap U_{j-1}$.  Then $(a_j,b_j,U_j)$
    is a special triple and $U_j \subseteq U_{j-1}$.
  \end{itemize}
  Fix a point $b \in U_N$.  For each $i \le n$ and each $j \le N$, we
  have $b \in U_j$, and so $b$ is linked to $H(a_j,i)$.  Since $b$ is
  linked to at most $k$ points, the set of hyperplanes
  \begin{equation*}
    \{H(a_1,i),H(a_2,i),\ldots,H(a_N,i)\}
  \end{equation*}
  can have size at most $k$---there must be a lot of repetition.
  Equivalently, if $\pi_i : \Mm^n \to \Mm$ is the $i$th coordinate
  projection, then
  \begin{equation*}
    |\{\pi_i(a_1),\ldots,\pi_i(a_N)\}| \le k.
  \end{equation*}
  On the other hand, $\{a_1,\ldots,a_N\} \subseteq \prod_{i=1}^n
  \{\pi_i(a_1),\ldots,\pi_i(a_N)\}$, and so
  \begin{equation*}
    |\{a_1,\ldots,a_N\}| \le k^n.
  \end{equation*}
  Every $a_j$ is linked to at most $k$ points in $\Mm^{n-1}$.  By
  definition of special triples, $a_j$ is linked to $b_j$.  So then
  \begin{equation*}
    |\{b_1,\ldots,b_N\}| \le k \cdot |\{a_1,\ldots,a_N\}| \le k^{n+1}
    < N,
  \end{equation*}
  contradicting the fact that the $b_i$ are distinct.
\end{proof}
Combining all the lemmas, we see that the properties $(\mathbf{A}_n)$,
$(\mathbf{B}_n)$, and $(\mathbf{C}_n)$ hold for all $n$.  This completes the proof of Proposition~\ref{tech-prop}.

Mainly, we are interested in the following consequence:
\begin{proposition} \label{rank-equality}
  Let $C \subseteq \Mm^\eq$ be a set of parameters and let $\ba, \bb$
  be finite tuples in $\Mm$, which are interalgebraic over $C$.  If
  $\ba$ and $\bb$ are both $\acl$-independent over $C$, then they have
  the same length.
\end{proposition}
\begin{proof}
  Otherwise, say, $\ba \in \Mm^n$ and $\bb \in \Mm^m$ for $m < n$.
  Then $\ba$ is interalgebraic with the $n-1$-tuple
  \begin{equation*}
    (b_1,b_2,\ldots,b_{m-2},b_{m-1},
    \underbrace{b_m,b_m,b_m,\ldots,b_m}_{\text{$n-m$ times}}),
  \end{equation*}
  contradicting property $(\mathbf{B}_n)$.
\end{proof}

\subsection{Dimension of complete types} \label{dim-sec-a}
\begin{definition}
  Let $\ba \in \Mm^n$ be a finite tuple and $C \subseteq \Mm^\eq$ be a
  set of parameters.  An \emph{acl-basis} of $\ba$ over $C$ is a
  subtuple $\bb \subseteq \ba$ such that $\bb$ is $\acl$-independent
  and $\ba \in \acl(C\bb)$.
\end{definition}
\begin{remark} \phantomsection \label{acl-bases}
  \begin{enumerate}
  \item Any two $\acl$-bases have the same length, by
    Proposition~\ref{rank-equality}.
  \item The $\acl$-bases of $\ba$ are exactly the minimal subtuples
    $\bb \subseteq \ba$ such that $\ba \in \acl(C\bb)$.  Therefore, at
    least one $\acl$-basis exists.
  \item If $\bb$ is an $\acl$-basis of $\ba$ then $\bb$ is a maximal
    $\acl$-independent subtuple of $\ba$.
  \end{enumerate}
\end{remark}
\begin{nontheorem} \label{ddag}
  If $\bb$ is a maximal $\acl$-independent subtuple of $\ba$, then
  $\bb$ is an $\acl$-basis of $\ba$.
\end{nontheorem} 
\begin{definition} \label{dim-def-0}
  If $\ba$ is a finite tuple and $C \subseteq \Mm^\eq$, then the
  \emph{dimension} $\dim(\ba/C)$ is the length of any tuple $\bb$ such
  that $\acl(C \ba) = \acl(C \bb)$ and $\bb$ is $\acl$-independent
  over $C$.
\end{definition}
At least one such $\bb$ exists, namely, any $\acl$-basis of $\ba$ over
$C$.  The choice of $\bb$ is irrelevant by
Proposition~\ref{rank-equality}.  Note that we do not require $\bb$ to be
a subtuple of $\ba$ in Definition~\ref{dim-def-0}.
\begin{remark}
  As in Corollary~\ref{notop}, the definition of $\dim(\ba/C)$ is
  independent of the topology $\tau$.
\end{remark}
\begin{proposition} \label{dimprops}
  Let $\ba, \bb$ be finite tuples in $\Mm$, and let $C \subseteq
  \Mm^\eq$ be a set of parameters.
  \begin{enumerate}
  \item \label{dp1} If $\ba$ and $\bb$ are interalgebraic over $C$, then
    $\dim(\ba/C) = \dim(\bb/C)$.
  \item \label{dp2} $\dim(\ba/C) = 0 \iff \ba \in \acl(C)$.
  \item \label{dp3} If $\ba$ is an $n$-tuple, then $\dim(\ba/C) \le n$, and
    \begin{equation*}
      \dim(\ba/C) = n \iff (\ba \text{ is $\acl$-independent over } C)
      \iff (\tp(\ba/C) \text{ is broad.}).
    \end{equation*}
  \item \label{dp4} If $C \subseteq C'$, or more generally, if $\acl(C) \subseteq
    \acl(C')$, then $\dim(\ba/C) \ge \dim(\ba/C')$.
  \item \label{dp5} $\dim(\ba,\bb/C) \le \dim(\ba/C\bb) + \dim(\bb/C)$.
  \end{enumerate}
\end{proposition}
\begin{proof}
  \begin{enumerate}
  \item Clear---if $\be$ is interalgebraic with $\ba$ then it is
    interalgebraic with $\bb$, so both $\dim(\ba/C)$ and $\dim(\bb/C)$
    equal the length of $\be$.
  \item Clear---$\ba$ is interalgebraic with the empty tuple iff $\ba
    \in \acl(C)$.
  \item $\dim(\ba/C)$ is the length of an $\acl$-basis $\bb \subseteq
    \ba$, so $\dim(\ba/C) \le n$.  If $\ba$ is $\acl$-independent, we
    can take $\bb = \ba$, so $\dim(\ba/C) = n$.  If $\ba$ isn't
    $\acl$-independent, then $\bb$ must be a proper subtuple, so
    $\dim(\ba/C) < n$.
  \item Let $\bb$ be an $\acl$-basis of $\ba$ over $C$, and let $n$ be
    the length of $\bb$, so that $n = \dim(\ba/C)$.  Since $\bb$ and
    $\ba$ are interalgebraic over $C$, they are interalgebraic over
    $C'$ and then $\dim(\ba/C') = \dim(\bb/C') \le n$.
  \item Let $\be$ be an $\acl$-basis of $\ba$ over $C\bb$, and let
    $\barf$ be an $\acl$-basis of $\bb$ over $C$.  Then $(\ba,\bb) \in
    \acl(C\be\barf)$, so $(\ba,\bb)$ is interalgebraic over $C$ with
    $(\be,\barf)$ and
    \begin{equation*}
      \dim(\ba,\bb/C) = \dim(\be,\barf/C) \le |\be| + |\barf| =
      \dim(\ba/C\bb) + \dim(\bb/C).  \qedhere
    \end{equation*}
  \end{enumerate}
\end{proof}
Note that we are \emph{not} claiming the following:
\begin{nontheorem} \phantomsection \label{nt-kappa}
  \begin{enumerate}
  \item If $\acl(C\ba) \subseteq \acl(C\bb)$, then $\dim(\ba/C) \le
    \dim(\bb/C)$.
  \item If $\ba$ is a subtuple of $\bb$, then $\dim(\ba/C) \le \dim(\bb/C)$.
  \item $\dim(\ba/C) \le \dim(\ba/C\bb) + \dim(\bb/C)$.
  \end{enumerate}
\end{nontheorem}
\begin{proposition} \label{extend-2}
  Let $\ba$ be a finite tuple, and let $C \subseteq C' \subseteq
  \Mm^\eq$ be small sets of parameters.  Then there is an automorphism
  $\sigma \in \Aut(\Mm/C)$ such that $\dim(\sigma(\ba)/C') =
  \dim(\sigma(\ba)/C) = \dim(\ba/C)$.
\end{proposition}
\begin{proof}
  Let $\bb \subseteq \ba$ be an $\acl$-basis of $\ba$ over $C$.  By
  Corollary~\ref{extend-1}, there is $\sigma \in \Aut(\Mm/C)$ such
  that $\sigma(\bb)$ is $\acl$-independent over $C'$.  Then
  $\sigma(\ba)$ and $\sigma(\bb)$ are interalgebraic over $C$ and over
  the bigger set $C'$, so $\sigma(\bb)$ is an $\acl$-basis of
  $\sigma(\ba)$ over $C'$, and $\dim(\sigma(\ba)/C') = |\bb| =
  \dim(\ba/C)$.
\end{proof}
\begin{lemma} \label{indep-2}
  Let $\ba, \bb$ be two tuples and $C \subseteq \Mm^\eq$ be a small
  set of parameters.  Then there is $\sigma \in \Aut(\Mm/C)$ such that
  \begin{equation*}
    \dim(\sigma(\ba),\bb/C) = \dim(\ba/C) + \dim(\bb/C).
  \end{equation*}
\end{lemma}
\begin{proof}
  Let $\be$ and $\barf$ be $\acl$-bases of $\ba$ and $\bb$,
  respectively, over $C$.  By Lemma~\ref{indep-1}, we may move $\ba$
  and $\be$ by an automorphism over $C$ and arrange for $(\be,\barf)$
  to $\acl$-independent.  Then $(\be,\barf)$ is an $\acl$-basis for
  $(\ba,\bb)$ over $C$, so
  \begin{equation*}
    \dim(\ba,\bb/C) = |\be| + |\barf| = \dim(\ba/C) + \dim(\bb/C).
    \qedhere
  \end{equation*}
\end{proof}

\subsection{Dimension of definable sets} \label{dim-sec-b}
\begin{definition}
  Let $X \subseteq \Mm^n$ be type-definable over a small set $C
  \subseteq \Mm^\eq$.  The \emph{dimension} of $X$, written $\dim(X)$,
  is defined to be
  \begin{equation*}
    \dim(X) = \max \{\dim(\ba/C) : \ba \in X\},
  \end{equation*}
  or $-\infty$ if $X = \varnothing$.
\end{definition}
\begin{proposition}
  $\dim(X)$ is well-defined, independent of the choice of $C$.
\end{proposition}
\begin{proof}
  Write the dimension as $\dim_C(X)$ to make the dependence on $C$
  explicit.  We must show
  \begin{equation*}
    \dim_C(X) = \dim_{C'}(X) \tag{$\ast$}
  \end{equation*}
  if $C, C' \subseteq \Mm^\eq$ are two small sets over which $X$ is
  type-definable.  First suppose $C \subseteq C'$.  By
  Proposition~\ref{dimprops}(\ref{dp4}),
  \begin{gather*}
    \dim(\ba/C) \ge \dim(\ba/C') \text{ for every $\ba \in X$} \\
    \text{and so } \dim_C(X) \ge \dim_{C'}(X).
  \end{gather*}
  Take $\ba \in X$ maximizing $\dim(\ba/C)$, so that $\dim(\ba/C) =
  \dim_C(X)$.  By Proposition~\ref{extend-2}, we may move $\ba$ by an
  automorphism over $C$ and arrange for $\dim(\ba/C') = \dim(\ba/C)$.
  Then \[\dim_{C'}(X) \ge \dim(\ba/C') = \dim(\ba/C) = \dim_C(X).\]
  This proves ($\ast$) in the case where $C \subseteq C'$.  The
  general case then follows:
  \begin{equation*}
    \dim_C(X) = \dim_{C \cup C'}(X) = \dim_{C'}(X).  \qedhere
  \end{equation*}
\end{proof}
Again, the definition of dimension is independent of the topology
$\tau$ which witnesses t-minimality.
\begin{theorem} \label{dimension-theorem}
  Let $X$ and $Y$ be type-definable sets.
  \begin{enumerate}
  \item \label{dt1} If $X \subseteq Y$, then $\dim(X) \le \dim(Y)$.
  \item \label{dt2} $\dim(X \cup Y) = \max(\dim(X),\dim(Y))$.
  \item \label{dt3} $\dim(X) \le 0$ iff $X$ is finite.
  \item \label{dt4} If $X \subseteq \Mm^n$, then $\dim(X) \le n$, with equality if
    and only if $X$ is broad.
  \item \label{dt5} If $X \subseteq \Mm^n$ is definable, then $\dim(X) = n$ iff
    $X$ has non-empty interior.
  \item \label{dt6} $\dim(\Mm^n) = n$.  More generally, $\dim(U) = n$ for any
    non-empty open definable set $U \subseteq \Mm^n$.
  \item \label{dt7} If $f : X \to Y$ is a type-definable bijection, or a
    type-definable surjection with finite fibers, then $\dim(X) =
    \dim(Y)$.
  \item \label{dt8} Let $f : X \to Y$ be a type-definable function and $k$ be an
    integer such that $\dim(f^{-1}(b)) \le k$ for every $b \in Y$
    (every fiber has dimension at most $k$).  Then $\dim(X) \le k +
    \dim(Y)$.
  \item $\dim(X \times Y) = \dim(X) + \dim(Y)$.
  \end{enumerate}
\end{theorem}
\begin{proof}
  Take $C \subseteq \Mm^\eq$ a small set over which everything is
  type-definable.  The proofs of (1) and (2) are clear.
  \begin{enumerate}
    \setcounter{enumi}{2}
  \item $\dim(X) > 0$ iff there is $\ba \in X$ such that $\dim(\ba/C)
    > 0$, or equivalently, $\ba \notin \acl(C)$.  This holds iff $X$
    is infinite.
  \item Since $\dim(X) = \max \{\dim(\ba/C) : \ba \in X\}$,
    Proposition~\ref{dimprops}(\ref{dp3}) shows that $\dim(X) \le n$.
    Moreover, equality holds if and only if $\tp(\ba/C)$ is broad for
    some $\ba \in X$.  This holds if and only if $X$ itself is broad.
  \item This follows by the previous point and Proposition~\ref{triad}.
  \item This follows by the previous point.
  \item Every element of $X$ is interalgebraic over $C$ with an
    element of $Y$, and vice versa, so
    \begin{equation*}
      \max \{\dim(\ba/C) : \ba \in X\} = \max \{\dim(\bb/C) : \bb \in Y\}
    \end{equation*}
    by Proposition~\ref{dimprops}(\ref{dp1}).
  \item Take $\ba \in X$ maximizing $\dim(\ba/C)$, so that
    $\dim(\ba/C) = \dim(X)$.  Let $\bb = f(\ba)$.  Note that $\bb \in
    \dcl(C \ba) \subseteq \acl(C \ba)$, so that $\ba$ is
    interalgebraic with $(\ba,\bb)$.  By Proposition~\ref{dimprops} parts
    (\ref{dp1}) and (\ref{dp5}),
    \begin{equation*}
      \dim(X) = \dim(\ba/C) = \dim(\ba,\bb/C) \le \dim(\ba/C\bb) +
      \dim(\bb/C).
    \end{equation*}
    Now $\ba$ belongs to the fiber $f^{-1}(\bb)$ which is type-definable
    over $C \bb$, and so
    \begin{equation*}
      \dim(\ba/C\bb) \le \dim(f^{-1}(\bb)) \le k.
    \end{equation*}
    Similarly, $\bb$ belongs to the set $Y$, which is type-definable
    over $C$, so
    \begin{equation*}
      \dim(\bb/C) \le \dim(Y).
    \end{equation*}
    Putting everything together,
    \begin{equation*}
      \dim(X) \le \dim(\ba/C\bb) + \dim(\bb/C) \le k + \dim(Y).
    \end{equation*}
  \item Let $k = \dim(X)$.  Consider the projection $X \times Y \to
    Y$.  Each fiber is in type-definable bijection with $X$, so it has
    dimension $k$.  By the previous point,
    \begin{equation*}
      \dim(X \times Y) \le k + \dim(Y) = \dim(X) + \dim(Y).
    \end{equation*}
    We need the equality in the reverse direction.  Take some $\ba \in
    X$ and $\bb \in Y$ with $\dim(\ba/C) = \dim(X)$ and $\dim(\bb/C) =
    \dim(Y)$.  By Lemma~\ref{indep-2}, we may move $\ba$ by an
    automorphism over $C$ and arrange
    \begin{equation*}
      \dim(\ba,\bb/C) = \dim(\ba/C) + \dim(\bb/C) = \dim(X) + \dim(Y).
    \end{equation*}
    Since $(\ba,\bb)$ is in the set $X \times Y$, which is
    type-definable over $C$, we get
    \begin{equation*}
      \dim(X \times Y) \ge \dim(\ba,\bb/C) = \dim(X) + \dim(Y).
      \qedhere
    \end{equation*}
  \end{enumerate}
\end{proof}
We are not saying the following, however:
\begin{nontheorem} \label{surj-nt}
  Let $X, Y$ be type-definable sets.  If $f : X \to Y$ is a definable
  surjection, then $\dim(X) \ge \dim(Y)$.
\end{nontheorem}
\begin{example}
  If $f : G \to H$ is a definable homomorphism of definable groups,
  then
  \begin{equation*}
    \dim(\ker(f)) \le \dim(G) \le \dim(\ker(f)) + \dim(\im(f)) \le
    \dim(\ker(f)) + \dim(H),
  \end{equation*}
  but it can happen that $\dim(G) < \dim(\im(f))$.
\end{example}
\begin{definition} \label{nikp}
  If $k \in \{0,1,\ldots,n\}$, then a set $X \subseteq \Mm^n$ has a
  \emph{near-injective $k$-projection} if there is some coordinate
  projection $\pi : \Mm^n \to \Mm^k$ such that $X \to \Mm^k$ has
  finite fibers.
\end{definition}
If $X$ is type-definable and $X$ has a near-injective $k$-projection,
then $\dim(X) = \dim(\pi(X)) \le k$ by Theorem~\ref{dimension-theorem}
parts (\ref{dt4}) and (\ref{dt7}).
\begin{proposition} \label{explicit}
  Let $X \subseteq \Mm^n$ be type-definable over a small set $C
  \subseteq \Mm^\eq$ and let $k \in \{0,1,\ldots,n\}$.  The
  following are equivalent:
  \begin{enumerate}
  \item $\dim(X) \le k$.
  \item $X \subseteq \bigcup_{i=1}^N D_i$ for some $C$-definable sets
    $D_1,\ldots,D_N$ with near-injective $k$-projections.
  \item $X \subseteq \bigcup_{i=1}^N D_i$ for some definable sets
    $D_1,\ldots,D_N$ with near-injective $k$-projections.
  \end{enumerate}
\end{proposition}
This gives another topology-free characterization of dimension.
\begin{proof}
  The implication (2)$\implies$(3) is trivial.  The implication
  (3)$\implies$(1) holds by the remarks above, with
  Theorem~\ref{dimension-theorem}(\ref{dt1},\ref{dt2}) to see that
  $\dim(X) \le \max_{1 \le i \le n} \dim(D_i) \le k$.  It remains to
  prove (1)$\implies$(2).

  Suppose $\dim(X) \le k$.  Fix any $\ba \in X$.  Let $\bb$ be an
  $\acl$-basis for $\ba$ over $C$.  Then
  \begin{equation*}
    k' := |\bb| = \dim(\ba/C) \le \dim(X) \le k.
  \end{equation*}
  Permuting coordinates, we may assume that $\bb$ is the first $k'$
  coordinates of $\ba$.  There is a formula $\phi(x_1,\ldots,x_n) \in
  \tp(\ba/C)$ witnessing that $\ba \in \acl(C\bb)$, in the sense that
  for any $\bb' \in \Mm^{k'}$, the set $\phi(\bb',\Mm^{n-k'})$ is
  finite.  Take $D_\ba$ to be the set defined by $\phi$.  Then $\ba
  \in D_\ba$, and $D_\ba$ has is a $C$-definable set with a
  near-injective $k$-projection.

  Finally, letting $\ba$ vary, we get $X \subseteq \bigcup_{\ba \in X}
  D_\ba$.  The union is small (since all the $D_\ba$ are
  $C$-definable), so by saturation there is a finite subcover, proving
  (2).
\end{proof}
\begin{corollary} \label{filtered-isect}
  Let $X = \bigcap_{i \in I} X_i$ be a small filtered intersection of
  type-definable sets.  Then $\dim(X) = \min_{i \in I} \dim(X_i)$.
\end{corollary}
\begin{proof}
  If $k = \dim(X)$, then it is clear by
  Theorem~\ref{dimension-theorem}(\ref{dt1}) that $k \le \min_{i \in
    I} \dim(X_i)$.  It remains to find some $i \in I$ with $\dim(X_i)
  \le k$.  By Proposition~\ref{explicit} applied to $X$, there are
  definable sets $D_1,\ldots,D_N$ with near-injective $k$-projections
  such that $X \subseteq \bigcup_{j = 1}^N D_j$.  Then
  \begin{equation*}
    \bigcap_{i \in I} X_i \subseteq \bigcup_{j = 1}^N D_j.
  \end{equation*}
  Since the intersection and union are small, and the intersection on
  the left is filtered, saturation gives an $i \in I$ such that
  \begin{equation*}
    X_i \subseteq \bigcup_{j = 1}^N D_j
  \end{equation*}
  and then $\dim(X_i) \le k$ by Proposition~\ref{explicit}.
\end{proof}
\begin{corollary} \label{baire}
  Let $X = \bigcup_{a \in Y} X_a$ be a filtered union, where
  $\{X_a\}_{a \in Y}$ is a definable family.  Then $\dim(X) = \max_{a
    \in Y} \dim(X_a)$.
\end{corollary}
Unlike the previous corollary, the sets $X$ and $X_a$ are definable,
not just type-definable, and the family $\{X_a\}_{a \in Y}$ is large
rather than small.
\begin{proof}
  Let $k = \dim(X)$.  Similar to the previous proof, we only need to
  find some $a \in Y$ with $\dim(X_a) \ge k$.  By
  Proposition~\ref{explicit}, we have $X \subseteq \bigcup_{i = 1}^N D_i$
  where each set $D_i$ is a definable set with a near-injective
  $k$-projection.  Then $X = \bigcup_{i=1}^N (D_i \cap X)$, so there
  is some $i$ such that $\dim(D_i \cap X) = \dim(X) = k$.  Let $D =
  D_i \cap X$ and let $\pi : \Mm^n \to \Mm^k$ be the coordinate
  projection such that $\pi : D_i \to \Mm^k$ has finite fibers.  Then
  $\pi : D \to \Mm^k$ has finite fibers.  Since $\dim(D) = k$, the
  image $\pi(D)$ must have dimension $k$ by
  Theorem~\ref{dimension-theorem}(\ref{dt7}).  Then $\pi(D)$ is a
  broad subset of $\Mm^k$ by
  Theorem~\ref{dimension-theorem}(\ref{dt4}).  Therefore it contains a
  product $P = \prod_{i=1}^k S_i$ for some sets $S_i \subseteq \Mm$ of
  size $\aleph_0$.  For each $p \in P$, choose some $\tilde{p} \in D$ with
  $\pi(\tilde{p}) = p$, and let $\tilde{P}$ be the small set
  $\{\tilde{p} : p \in P\}$.  Then \[\tilde{P} \subseteq D \subseteq X
  = \bigcup_{a \in Y} X_a.\] Since the union is filtered, saturation
  gives some $a \in Y$ such that $\tilde{P} \subseteq X_a$.  Then
  $\tilde{P} \subseteq X_a \cap D$ and $P \subseteq \pi(X_a \cap D)$,
  showing that $\pi(X_a \cap D)$ is broad.  The projection $\pi : X_a
  \cap D \to \pi(X_a \cap D)$ has finite fibers, so $X_a \cap D$ has
  dimension $k$.  Thus $\dim(X_a) \ge k$.
\end{proof}
\begin{definition}
  The dimension of an $\mathcal{L}(\Mm)$-formula
  $\phi(x_1,\ldots,x_n)$ or a small partial type
  $\Phi(x_1,\ldots,x_n)$ is the dimension of the definable or
  type-definable set $X \subseteq \Mm^n$ defined by $\phi(\bx)$ or
  $\Phi(\bx)$.
\end{definition}
\begin{remark} \label{complete-dimension}
  The dimension of the complete type $\tp(\ba/C)$ is $\dim(\ba/C)$.
  Indeed, the corresponding type-definable set is $X = \{\bb \in \Mm^n
  : \bb \equiv_C \ba\}$.  If $\bb \in X$, then $\dim(\bb/C) =
  \dim(\ba/C)$ by automorphism invariance of $\dim(-/C)$.  Since $X$
  is type-definable over $C$,
  \begin{equation*}
    \dim \tp(\ba/C) = \dim X = \max \{\dim(\bb/C) : \bb \in X\} =
    \dim(\ba/C).
  \end{equation*}
\end{remark}
\begin{proposition}
  Let $\Phi$ be a partial type over a small set of parameters $C$.
  \begin{enumerate}
  \item $\dim(\Phi)$ is the maximum of $\dim(p)$ as $p$ ranges over
    completions of $\Phi$ (in $S_n(C)$).
  \item $\dim(\Phi)$ is the minimum of $\dim(\Phi_0)$ as $\Phi_0$
    ranges over finite subtypes $\Phi_0 \subseteq \Phi$.
  \end{enumerate}
\end{proposition}
\begin{proof}
  \begin{enumerate}
  \item This is clear in light of Remark~\ref{complete-dimension}.  In more detail, note that
    \begin{equation*}
      \{ p \in S_n(C) : p \supseteq \Phi\} = \{\tp(\ba/C) : \ba \text{
        satisfies } \Phi\}.
    \end{equation*}
    By Remark~\ref{complete-dimension},
    \begin{equation*}
      \max \{ \dim(p) : p \in S_n(C), ~ p \supseteq \Phi\} = \max
      \{\dim(\ba/C) : \ba \text{ satisfies } \Phi\} =: \dim(\Phi).
    \end{equation*}
  \item This is essentially a rephrasing of
    Corollary~\ref{filtered-isect}.  \qedhere
  \end{enumerate}
\end{proof}
\begin{definition} \label{pakc}
  A \emph{weak $k$-cell} is a definable set $D \subseteq \Mm^n$
  such that there is a coordinate projection $\pi : \Mm^n \to \Mm^k$
  with the following properties:
  \begin{itemize}
  \item The image $\pi(D) \subseteq \Mm^k$ is a non-empty open subset
    of $\Mm^k$.
  \item The projection $D \to \pi(D)$ has finite fibers.
  \end{itemize}
  A \emph{very weak $k$-cell} is defined in the same way,
  replacing ``non-empty open subset'' with ``broad set'' in the first
  bullet point.
\end{definition}
\begin{remark} \phantomsection \label{pathetic-rem}
  \begin{enumerate}
  \item \label{pr1} If $D$ is a weak $k$-cell, or a very weak $k$-cell,
    then $D$ has dimension $k$, by
    Theorem~\ref{dimension-theorem}(\ref{dt4},\ref{dt6},\ref{dt7}).
  \item \label{pr2} If $D$ has a near-injective $k$-projection $\pi : D \to
    \Mm^k$, then one of the following holds:
    \begin{itemize}
    \item The image $\pi(D)$ is broad, and $D$ is a very weak
      $k$-cell.
    \item The image $\pi(D)$ is narrow, and $\dim(D) = \dim(\pi(D)) <
      k$ by Theorem~\ref{dimension-theorem}(\ref{dt4},\ref{dt7}).
    \end{itemize}
  \item \label{pr3} If $D$ is a very weak $k$-cell with projection
    $\pi : D \to \Mm^k$, then we can write $D = D' \sqcup D''$, where
    $D'$ is a weak $k$-cell and $\dim(D'') < k$.  To see this,
    decompose the broad set $X := \pi(D)$ as $X' \cup X''$, where $X'
    = \ter(X)$ and $X'' = X \setminus X'$.  The set $X''$ has empty
    interior so it is narrow (Proposition~\ref{triad}).  Let $D'$ and
    $D''$ be the preimages of $X'$ and $X''$ in $D$.
  \end{enumerate}
\end{remark}
\begin{lemma} \label{pathetic-lem}
  Let $D \subseteq \Mm^n$ be a non-empty definable set with $\dim(D) =
  k$.
  \begin{enumerate}
  \item We can write $D$ as a finite union $\bigcup_{i = 1}^N D_i$
    where each $D_i$ has a near-injective $k$-projection.
  \item \label{pl2} We can write $D$ as a finite \emph{disjoint} union
    $\coprod_{i=1}^N D_i$ where each $D_i$ has a near-injective
    $k$-projection.
  \item We can write $D$ as a finite disjoint union $D_0 \sqcup
    \coprod_{i=1}^N D_i$, where $\dim(D_0) < k$ and $D_i$ is a
    very weak $k$-cell for $i > 0$.
  \item We can write $D$ as a finite disjoint union $D_0 \sqcup
    \coprod_{i=1}^N D_i$, where $\dim(D_0) < k$ and $D_i$ is a
    weak $k$-cell for $i > 0$.
  \end{enumerate}
\end{lemma}
\begin{proof}
  Proposition~\ref{explicit} gives $D \subseteq \bigcup_{i=1}^N D_i$,
  where each $D_i$ has a near-injective $k$-projection.  Replacing
  each $D_i$ with $D_i \cap D$ gives part (1).  (If $X$ has a
  near-injective $k$-projection, so does any subset.)  We can arrange
  for the $D_i$ to be pairwise disjoint by replacing $D_i$ with $D_i
  \setminus (D_1 \cup D_2 \cup \cdots \cup D_{i-1})$.  This gives (2).

  By Remark~\ref{pathetic-rem}(\ref{pr2}), some of the $D_i$'s are
  very weak $k$-cells, and the rest have dimension $< k$.
  Collecting the low-dimensional $D_i$'s into a new set $D_0$, we get
  part (3).

  By Remark~\ref{pathetic-rem}(\ref{pr3}), we can split each $D_i$
  into a weak $k$-cell $D'_i$ and a low-dimensional set $D''_i$.
  Replacing $D_i$ with $D'_i$ and throwing $D''_i$ into $D_0$
  gives part (4).
\end{proof}
\begin{proposition}[Weak cell decomposition]
  \label{pathetic-cells}
  Let $D$ be a definable set.  Then $D$ is a finite disjoint union of
  weak cells.
\end{proposition}
\begin{proof}
  This is immediate by Lemma~\ref{pathetic-lem}(4) and induction on
  $\dim(D)$.
\end{proof}
The weak cell decomposition is not very useful, but it at least
gives us the following:
\begin{theorem} \label{new-def-dim}
  Dimension is definable in families: if $\{X_a\}_{a \in Y}$ is a
  definable family of definable sets, then each of the sets
  \begin{equation*}
    Y_k = \{a \in Y : \dim(X_a) = k\}
  \end{equation*}
  is definable.
\end{theorem}
\begin{proof}
  For each $k$, the family $\mathcal{F}_k$ of weak $k$-cells is ind-definable,
  i.e., a small union of definable families.  Then the family
  \begin{equation*}
    \mathcal{F}_{\le k} = \{C : \text{$C$ is a weak cell with }
    \dim(C) \le k\} = \bigcup_{k' = 1}^k \mathcal{F}_{k'}
  \end{equation*}
  is also ind-definable.  By Proposition~\ref{pathetic-cells}, a definable
  set $D$ has dimension $k$ if and only if $D$ is a disjoint union
  $C_1 \sqcup \cdots \sqcup C_N$ with $C_1 \in \mathcal{F}_k$ and
  $C_2,\ldots,C_N \in \mathcal{F}_{\le k}$.  It follows that the
  family of $k$-dimensional definable sets is ind-definable.
  Therefore, each of the sets $Y_k$ is $\vee$-definable.  The
  complement of $Y_k$ is the $\vee$-definable set $Y_{-\infty} \cup
  Y_0 \cup Y_1 \cup \cdots \cup \widehat{Y_k} \cup \cdots$, so $Y_k$
  is both $\vee$-definable and type-definable, hence definable.
\end{proof}
Because of the definability of dimension, the dimension theory works
equally well for definable sets over models $M$ other than the
monster.  If $D \subseteq M^n$ is definable, define $\dim(D)$ to be
$\dim(D(\Mm))$ for any monster model $\Mm$ extending $M$.
\begin{theorem} \label{small-dimension-theorem}
  Let $M$ be a model of $T$.  Let $X, Y$ be definable sets.
  \begin{enumerate}
  \item \label{sdt1} If $X \subseteq Y$, then $\dim(X) \subseteq \dim(Y)$.
  \item \label{sdt2} $\dim(X \cup Y) = \max(\dim(X),\dim(Y))$.
  \item \label{sdt3} $\dim(X) \le 0$ iff $X$ is finite.
  \item \label{sdt4} If $X \subseteq M^n$, then $\dim(X) \le n$, with equality if
    and only if $X$ has non-empty interior.  In particular,
    $\dim(M^n) = n$.
  \item \label{sdt5} If $f : X \to Y$ is a definable bijection, or a
    definable surjection with finite fibers, then $\dim(X) =
    \dim(Y)$.
  \item \label{sdt6} Let $f : X \to Y$ be a definable function and $k$ be an
    integer such that $\dim(f^{-1}(b)) \le k$ for every $b \in Y$
    (every fiber has dimension at most $k$).  Then $\dim(X) \le k +
    \dim(Y)$.
  \item \label{sdt7} $\dim(X \times Y) = \dim(X) + \dim(Y)$.
  \item \label{sdt8} $\dim(X) \le k$ iff $X$ is covered by finitely many definable
    sets with near-injective $k$-projections.
  \item \label{sdt9} Let $X = \bigcup_{a \in Y} X_a$ be a filtered union, where
    $\{X_a\}_{a \in Y}$ is a definable family.  Then $\dim(X) =
    \max_{a \in Y} \dim(X_a)$.
  \item \label{sdt10} Dimension is definable in families: if
    $\{X_a\}_{a \in Y}$ is a definable family of definable sets, then
    each of the sets
    \begin{equation*}
      Y_k = \{a \in Y : \dim(X_a) = k\}
    \end{equation*}
    is definable.
  \end{enumerate}
\end{theorem}
Part (\ref{sdt8}) comes from Proposition~\ref{explicit}(2) with $C = M$.
The other parts come directly from the corresponding facts in $\Mm$ in
Theorems~\ref{dimension-theorem}, Corollary~\ref{baire}, and
Theorem~\ref{new-def-dim}.  Part (\ref{sdt5}) uses uniform finiteness of
t-minimal theories to ensure that the finiteness of the fibers is
preserved when passing from $M$ to $\Mm$.  Similarly, parts
(\ref{sdt6}) and (\ref{sdt9}) use definability of dimension to
preserve the bounds on dimensions in families, when passing from $M$
to $\Mm$.

Here is an example application:
\begin{corollary} \label{perfect-field-cor}
  Let $(K,+,\cdot)$ be a definable field in a t-minimal theory.  Then
  $K$ is perfect.
\end{corollary}
\begin{proof}
  Otherwise, $K$ is an infinite field with characteristic $p > 0$, and
  the set of $p$th powers $K^p = \{x^p : x \in K\}$ is a proper
  subfield.  Take $b \in K \setminus K^p$.  The definable injection
  \begin{gather*}
    K^p \times K^p \to K \\
    (x,y) \mapsto x + by
  \end{gather*}
  shows that $\dim(K) \ge \dim(K^p \times K^p) = 2 \dim(K^p)$.  But
  the definable bijection
  \begin{gather*}
    K \to K^p \\
    x \mapsto x^p
  \end{gather*}
  shows $\dim(K^p) = \dim(K)$.  Thus $\dim(K) \ge 2 \dim(K^p) = 2
  \dim(K)$, and $\dim(K) = 0$.  Then $K$ is finite, a contradiction.
\end{proof}

\subsection{Comparison to other dimension theories}
Continue to work in a monster model $\Mm$ of a t-minimal theory $T$.
\begin{theorem} \label{comparison}
  Let $X$ be a type-definable set in $\Mm^n$.
  \begin{enumerate}
  \item If $T$ has the exchange property, then $\dim(X)$ agrees with
    $\rk(X)$, the $\acl$-dimension of $X$.
  \item If $T$ is dp-minimal, then $\dim(X)$ agrees with $\dpr(X)$,
    the dp-rank of $X$.
  \item In general, $\dim(X) \le d_t(X)$, where $d_t(X)$ is the naive
    topological dimension of $X$.
  \end{enumerate}
\end{theorem}
\begin{proof}
  \begin{enumerate}
  \item If the exchange property holds, then an $\acl$-basis of $\ba$
    over $C$ is just a basis of $\ba$ in the pregeometry $\acl_C(-)$,
    and our $\dim(\ba/C)$ is the usual $\acl$-dimension.  Then
    $\dim(X)$ is the usual $\acl$-dimension.
  \item Fix a small set $C \subseteq \Mm^\eq$ over which $X$ is
    type-definable.  By well-known properties of dp-rank, $\dpr(X) = \max
    \{\dpr(\ba/C) : \ba \in X\}$.  Therefore, it suffices to show that
    $\dpr(\ba/C) = \dim(\ba/C)$.  Take an $\acl$-independent tuple
    $\bb$ interalgebraic over $C$ with $\ba$.  Then
    \begin{align*}
      \dim(\ba/C) &= \dim(\bb/C) = |\bb| \\
      \dpr(\ba/C) &= \dpr(\bb/C) = |\bb|.
    \end{align*}
    The final equality holds because broad type-definable sets $X
    \subseteq \Mm^n$ have dp-rank $n$.  Indeed, if $X \supseteq
    \prod_{i=1}^n \{a_{i,0},a_{i,1}, a_{i,2}, \ldots\}$, then the
    array of formulas $\{x_i = a_{i,j}\}_{1 \le i \le n, ~ j <
      \omega}$ is an ict-pattern of depth $n$ in $X$, showing $\dpr(X)
    = n$.  On the other hand, $\dpr(X) \le \dpr(\Mm^n)$ by the
    sub-additivity of dp-rank \cite{dp-add}.
  \item Both $\dim(-)$ and $d_t(-)$ respect unions
    (Corollary~\ref{dt-union-max} and
    Theorem~\ref{dimension-theorem}(\ref{dt2})), so by
    Proposition~\ref{pathetic-cells} we may assume $X$ is a weak
    $k$-cell.  Then $X$ has a finite-to-one coordinate projection onto
    a non-empty open set $U \subseteq \Mm^k$, so $k = \dim(X) \le
    d_t(X)$.  \qedhere
  \end{enumerate}
\end{proof}
Part (2) is related to Simon's work on dp-rank in dp-minimal theories
\cite{surprise}.  In fact, Proposition~3.4 in \cite{surprise}
essentially says that $\dpr(\ba/C) \ge r$ iff $\tp(\bb/C)$ is broad
for some subtuple $\bb \subseteq \ba$ of length $r$, without assuming
the presence of any definable topology.

\section{Visceral theories} \label{vt-section}
Fix a monster model $\Mm$ of a visceral theory $T$.  Fix a definable
basis of entourages $\mathcal{B}$.  We can assume that every $E \in
\mathcal{B}$ is open and symmetric, by replacing $\mathcal{B}$ with
the basis
\begin{equation*}
  \{\ter(E \cap E^{-1}) : E \in \mathcal{B}\}.
\end{equation*}
\subsection{The Hammer Lemma}
The following Lemma is a hammer, and everything is a nail:
\begin{lemma}[Hammer Lemma]
  Let $D \subseteq \Mm^n$ be a non-empty definable set.  Let $\bowtie$
  be a definable relation between $D$ and $\mathcal{B}$ with the
  following properties:
  \begin{itemize}
  \item For every $a \in D$, there is some $E \in \mathcal{B}$ such
    that $a \bowtie E$.
  \item If $a \bowtie E$ and $E' \subseteq E$, then $a \bowtie E'$.
  \end{itemize}
  Say that $a$ is \emph{compatible with} $E$ if $a \bowtie E$.  Then
  the following things hold:
  \begin{enumerate}
  \item If $D$ is infinite, there is an entourage $E \in \mathcal{B}$
    and an infinite definable subset $D' \subseteq D$ such that every $a
    \in D'$ is compatible with $E$.
  \item If $D$ is non-empty open, then there is an entourage $E
    \in \mathcal{B}$ and a non-empty open definable subset $D' \subseteq D$
    such that every $a \in D'$ is compatible with $E$.
  \item In general, there is an entourage $E \in \mathcal{B}$ and a
    definable subset $D' \subseteq D$ with $\dim(D') = \dim(D)$ such
    that that every $a \in D'$ is compatible with $E$.
  \end{enumerate}
\end{lemma}
\begin{proof}
  \begin{enumerate}
  \item Take $S \subseteq D$ with $|S| = \aleph_0$.  For each $a \in
    S$, let $E_a$ be a basic entourage such that $a \bowtie E_a$.  By
    saturation, there is a basic entourage $E \subseteq \bigcap_{a \in
      S} E_a$.  Then $a \bowtie E$ for every $a \in S$.  Let $D' = \{a
    \in D : a \bowtie E\}$.  Then $D'$ is infinite because $S
    \subseteq D'$.
  \item By Proposition~\ref{triad}, $D$ is broad, so it contains a product
    $P = \prod_{i=1}^n S_i$ where each $S_i$ is a set of size
    $\aleph_0$.  For each $a \in P$, let $E_a$ be a basic entourage
    such that $a \bowtie E_a$.  By saturation, there is a basic
    entourage $E \subseteq \bigcap_{a \in P} E_a$.  Then $a \bowtie E$
    for every $a \in P$.  Let $D'' = \{a \in D : a \bowtie E\}$.  Then
    $D''$ is broad because $D'' \supseteq P = \prod_{i=1}^n S_i$.
    Take $D' = \ter(D'')$.
  \item For each entourage $E$, let $D_E = \{a \in D : a \bowtie E\}$.
    The assumptions on $\bowtie$ say that $D$ is a filtered union
    $\bigcup_{E \in \mathcal{B}} D_E$.  By Corollary~\ref{baire},
    there is some $E \in \mathcal{B}$ such that $\dim(D_E) = \dim(D)$.
    Take $D' = D_E$. \qedhere
  \end{enumerate}
\end{proof}
Here is an example application:
\begin{lemma} \label{almost-perfect}
  If $D \subseteq \Mm^n$ is definable, then $D$ has only finitely many
  isolated points.
\end{lemma}
Lemma~\ref{almost-perfect} fails for t-minimal theories.  For example, in RCF
with the Sorgenfrey topology, the set $\{(x,y) : x+y = 0\} \subseteq
\Mm^2$ is infinite and discrete.
\begin{proof}
  Let $X$ be the set of isolated points in $D$.  Suppose for the sake
  of contradiction that $X$ is infinite.  If $\ba \in X$ and $E \in
  \mathcal{B}$, let $\ba \bowtie E$ mean that $E$ isolates $\ba$ in
  $X$, in the sense that $E[\ba] \cap X = \{\ba\}$, where $E[\ba]$
  denotes $\prod_{i=1}^n E[a_i]$.  By discreteness of $X$, every $\ba
  \in X$ is compatible with at least one $E \in \mathcal{B}$.  By the
  Hammer Lemma, there is an infinite definable subset $X_0 \subseteq
  X$ and an entourage $E$ compatible with every point in $X_0$.
  \underline{Fix $E$}.
  \begin{claim}
    Let $Y \subseteq \Mm^n$ be an infinite definable set and $\pi :
    \Mm^n \to \Mm$ be a coordinate projection.  Then there is an
    infinite definable subset $Y' \subseteq Y$ such that for any $a, b
    \in Y'$, we have $\pi(a) \mathrel{E} \pi(b)$.
  \end{claim}
  \begin{claimproof}
    If some fiber of $Y \to \pi(Y)$ is infinite, then we can take $Y'$
    to be this fiber.  Otherwise, the set $\pi(Y)$ is infinite.  Take
    a point $c$ in the interior of $\pi(Y)$.  Let $D$ be a basic
    entourage with $D \circ D \subseteq E$.  Then $D[c] \cap \pi(Y)$
    is a neighborhood of $c$, so it is infinite.  Let $Y' = \{a \in Y
    : \pi(a) \in D[c]\}$.  Then $Y'$ is infinite.  If $a, b \in Y'$,
    then $\pi(a), \pi(b) \in D[c]$, and so $\pi(a) \in E[\pi(b)]$ by
    choice of $D$.
  \end{claimproof}
  Starting with $X_0$ and applying the claim to each coordinate projection, we get a chain of infinite definable sets
  \begin{equation*}
    X_0 \supseteq X_1 \supseteq \cdots \supseteq X_n
  \end{equation*}
  such that $\pi_i(a) \mathrel{E} \pi_i(b)$ for $a,b \in X_i$.  Taking
  distinct $a,b \in X_n$, we get $a \in E[b]$, contradicting the
  choice of $E$.
\end{proof}

\subsection{Generic continuity of correspondences}
\begin{lemma} \label{continuity-1}
  Let $D \subseteq \Mm$ be infinite and definable and $f : D \rightrightarrows
  \Mm^m$ be a definable $k$-correspondence.
  \begin{enumerate}
  \item There is some infinite definable $D' \subseteq D$ such that
    $f$ is continuous on $D'$.
  \item In fact, $f$ is continuous on a cofinite subset of $D$.
  \end{enumerate}
\end{lemma}
\begin{proof}
  \begin{enumerate}
  \item By induction on $k$.  The base case $k=1$ is
    \cite[Proposition~3.12]{viscerality}.  Suppose $k > 1$.  For $a \in D$ and $E
    \in \mathcal{B}$, say $a \bowtie E$ if $E$ strongly separates the
    points of $f(a)$:
    \begin{equation*}
      E[\bx] \cap E[\by] = \varnothing \text{ for distinct } \bx, \by \in
      f(a).
    \end{equation*}
    Since the topology is Hausdorff, every $a \in D$ is compatible
    with at least one $E$.  By the Hammer Lemma, there is some $E$ and
    a ball $B \subseteq D$ such that every $a \in B$ is compatible
    with $E$.  That is, $E$ strongly separates the points of $f(a)$
    for any $a \in B$.  The set of isolated points in the graph
    $\Gamma(f)$ is finite by Lemma~\ref{almost-perfect}, so we can fix
    some $a \in B$ and $\bb \in f(a)$ such that $(a,\bb)$ is
    non-isolated in $\Gamma(f)$.  Then the set
    \begin{equation*}
      X = \Gamma(f) \cap (B \times E[\bb]) = \{(x,y) : x \in B, ~ y
      \in f(x) \cap E[\bb]\}
    \end{equation*}
    is infinite.  Let $D' \subseteq B$ the projection of $X$ onto the
    first coordinate.  The set $D'$ is infinite.  If $x \in D'$, then
    there is at least one $\by \in f(x) \cap E[\bb]$.  If there are
    two points $\by, \by'$ in $f(x) \cap E[\bb]$, then $\bb \in E[\by]
    \cap E[\by']$, contradicting the fact that $E$ strongly separates
    the points of $f(x)$ because $x \in D' \subseteq B$.

    Therefore, for every $x \in D'$ there is a unique $\by$ in $f(x)
    \cap E[\bb]$.  Let $g(x)$ be this unique $y$, and let $h(x) = f(x)
    \setminus \{g(x)\}$.  Then $g$ is a definable function and $h$ is
    a definable $(k-1)$-correspondence on $D'$.  By induction, $g$ and
    $h$ are continuous on an infinite definable subset $D'' \subseteq
    D'$.  Then $f$ is continuous on $D''$ too.
  \item Let $D_{\mathrm{bad}}$ be the set of points $a \in D$ such
    that $f$ is discontinuous at $a$.  If $D_{\mathrm{bad}}$ is
    infinite, then $f \restriction D_{\mathrm{bad}}$ is a
    counterexample to part (1). \qedhere
  \end{enumerate}
\end{proof}
\begin{proposition} \label{gencon}
  Let $D \subseteq \Mm^n$ be a non-empty open set and $f : D \rightrightarrows
  \Mm^m$ be a definable $k$-correspondence.
  \begin{enumerate}
  \item For each $i$, there is a smaller non-empty open set $D'
    \subseteq D$ (depending on $i$) such that $f(x_1,\ldots,x_n)$ is
    continuous in the variable $x_i$ on $D'$.
  \item There is a smaller non-empty open set $D' \subseteq D$ such
    that $f$ is continuous in each variable separately on $D'$.
  \item There is a point $\ba \in D$ such that $f$ is continuous at $\ba$.
  \item We can write $f$ as $D' \sqcup D''$ where $D'$ is open and
    dense in $D$, $\dim(D'') < n$, and $f$ is continuous on $D'$.
  \end{enumerate}
\end{proposition}
\begin{proof}
  \begin{enumerate}
  \item Without loss of generality, $i = 1$.  Let $D_{\mathrm{bad}}$ be the set
    of points $a \in D$ such that $f$ is not continuous in the
    variable $x_1$.  If $D_{\mathrm{bad}}$ is broad, it contains a box $B = B_1
    \times \cdots \times B_n$.  Fixing some points $(a_2,\ldots,a_n)
    \in B_2 \times \cdots \times B_n$, there is a definable
    $k$-correspondence
    \begin{gather*}
      g : B_1 \to \Mm^m \\
      g(x) = f(x,a_2,\ldots,a_n),
    \end{gather*}
    and $g$ is nowhere continuous on $B_1$, contradicting
    Lemma~\ref{continuity-1}.  Therefore $D_{\mathrm{bad}}$ is narrow, and the
    complement $D \setminus D_{\mathrm{bad}}$ is broad.  Take $D'$ to be a ball
    in $D \setminus D_{\mathrm{bad}}$.
  \item Apply part (1) to each variable in turn, building a descending
    sequence of non-empty open sets $D = D_0 \supseteq D_1 \supseteq
    \cdots \supseteq D_n$ such that $f \restriction D_i$ is continuous
    in the variable $i$.  Take $D' = D_n$.
  \item By part (2), we may first shrink $D$ and arrange that $f$ is
    continuous in each variable separately on $D$.  Suppose for the
    sake of contradiction that $f$ is nowhere continuous.  If $a \in
    D$ and $E \in \mathcal{B}$, let $a \bowtie E$ mean that $E$
    witnesses a failure of continuity:
    \begin{quote}
      For any $E_\delta$ there is $a' \in E_\delta[a]$ such that
      $f(a') \centernot{\mathrel{E}} f(a)$.
    \end{quote}
    (Here, the notation $f(x) \mathrel{E} f(x')$ means that we can write
    \begin{gather*}
      f(x) = \{y_1,\ldots,y_k\} \\
      f(x') = \{y_1',\ldots,y_k'\}
    \end{gather*}
    in such a way that $y_i \mathrel{E} y'_i$ for each $1 \le i \le
    k$.)
    
    Since $f$ is nowhere continuous, every $a \in D$ is compatible
    with at least one $E$.  Then the Hammer Lemma gives a ball $D'
    \subseteq D$ and an entourage $E_1$ such that every $a \in D'$ is
    $\bowtie$-compatible with $E_1$.  Let $E_2$ be an entourage that
    is so small that
    \begin{equation*}
      \underbrace{E_2 \circ E_2 \circ \cdots \circ E_2}_{\text{$n$
          times}} \subseteq E_1.  \tag{$\ast$}
    \end{equation*}
    If $a \in D'$ and $E \in \mathcal{B}$, let $a \propto E$ mean the
    following:
    \begin{quote}
      If $a' \in E[a]$ and $a'$ differs from $a$ in only one
      coordinate, then $f(a) \mathrel{E_2} f(a')$.
    \end{quote}
    Because $f$ is continuous in each variable separately, for every
    $a \in D'$ there is $E \in \mathcal{B}$ such that $a \propto E$.
    The Hammer Lemma gives a smaller ball $D'' \subseteq D'$ and an
    entourage $E_3$ such that $a \propto E_3$ for every $a \in D''$.
    Shrinking $D''$ further, we may assume that $D''$ is a box
    $\prod_{i=1}^n B_i$ where each $B_i$ is an $E_3$-small ball ($x,y
    \in B_i$ implies $x \mathrel{E_3} y$).

    Take any $\ba = (a_1,\ldots,a_n) \in D''$.  Because $\ba \in D''
    \subseteq D'$, we have $a \bowtie E_1$.  Therefore, there is $\bb
    = (b_1,\ldots,b_n) \in D''$ such that $f(\ba) \centernot
    {\mathrel{E_1}} f(\bb)$.  Because we arranged for $D''$ to be a
    box $D'' = \prod_{i=1}^n B_i$, each of the points
    \begin{equation*}
      c_i = (b_1,b_2,\ldots,b_i,a_{i+1},a_{i+2},\ldots,a_n) \text{ for
      } 0 \le i \le n
    \end{equation*}
    is in $D''$.  Thus $c_i \propto E_3$ holds.  Moreover, $a_i, b_i \in
    B_i$, and $B_i$ is $E_3$-small, so $a_i \mathrel{E_3} b_i$.  By
    definition of the relation $\propto$, it follows that $f(c_i)
    \mathrel{E_2} f(c_{i+1})$.  Looking at the chain of
    relations \[f(\ba) = f(c_0) \mathrel{E_2} f(c_1) \mathrel{E_2}
    f(c_2) \mathrel{E_2} \cdots \mathrel{E_2} f(c_n) = f(\bb),\] we
    conclude that $f(\ba) \mathrel{E_1} f(\bb)$, contradicting the
    choice of $\bb$.
  \item Let $D_{\mathrm{bad}}$ be the set of points $a \in D$ such that $f$ is
    discontinuous at $a$.  By the previous point, $D_{\mathrm{bad}}$ is narrow.
    By Corollary~\ref{small-closure}, the closure $\overline{D_{\mathrm{bad}}}$ is
    narrow.  Take $D'' = D \cap \overline{D_{\mathrm{bad}}}$ and $D' = D \setminus
    D'' = D \setminus \overline{D_{\mathrm{bad}}}$.  All the properties are clear,
    except perhaps the density of $D'$ in $D$.  Density can be proven
    as follows: given $b \in D$, take a small ball $B \ni b$ with $B
    \subseteq D$.  Then $\dim(B) = n > \dim(D'')$, so $B \not
    \subseteq D''$ and $B$ intersects $D'$.  \qedhere
  \end{enumerate}
\end{proof}

\subsection{Cell decomposition}
\begin{definition}
  A \emph{$k$-cell} is a definable set of the form
  $\sigma(\Gamma(f))$, where $U \subseteq \Mm^k$ is a non-empty open
  set, $f : U \rightrightarrows \Mm^{n-k}$ is a continuous $m$-correspondence for
  some $m \ge 1$, and $\sigma$ is a coordinate permutation.
\end{definition}
Any $k$-cell is a weak $k$-cell (Definition~\ref{pakc}).  In
particular, any $k$-cell has dimension $k$.
\begin{lemma} \label{just-now}
  Let $D \subseteq \Mm^n$ be a definable set with a near-injective
  $k$-projection (Definition~\ref{nikp}).  Then we can write $D$ as a
  disjoint union $D = D_0 \sqcup C_1 \sqcup \cdots \sqcup C_N$ where
  $\dim(D_0) < k$ and each $C_i$ is a $k$-cell.
\end{lemma}
\begin{proof}
  Let $\pi : D \to \Mm^k$ be the near-injective coordinate projection.
  Permuting coordinates, we may assume $\pi$ is the projection onto
  the first $k$ coordinates.  Let $X = \pi(D)$.  For each $m \ge 1$,
  let $X_m$ be the set of points $b \in X$ such that $\pi^{-1}(b) \cap
  D$ has size $m$.  By uniform finiteness, almost every $X_m$ is
  empty.  Let $D_m$ be the preimage of $X_m$ in $D$.  Then $D$ is the
  disjoint union of the sets $D_m$, and $\pi : D_m \to X_m$ is onto
  with every fiber of size $m$.  Thus $D_m$ is the graph of an
  $m$-correspondence $f_m : X_m \to \Mm^{n-k}$.  By generic continuity
  of correspondences (Proposition~\ref{gencon}(4)), we can write $X_m$ as
  $U_m \sqcup X'_m$ where $U_m$ is open (possibly empty), $X'_m$ is
  narrow, and $f_m$ is continuous on $U_m$.  Let
  \begin{equation*}
    \{C_1,\ldots,C_N\} = \{\Gamma(f_m \restriction U_m) : m \ge 1, ~
    U_m \ne \varnothing\}.
  \end{equation*}
  Then each $C_i$ is a $k$-cell.  Note that
  \begin{equation*}
    D \setminus \bigcup_{i = 1}^N C_i = \bigcup_m \Gamma(f_m \restriction X'_m), \tag{$\ast$}
  \end{equation*}
  because $D$ is the union of the graphs of the $f_m$'s, and
  $\dom(f_m)$ is the disjoint union of $U_m$ and $X'_m$.  The right
  hand side of ($\ast$) has dimension less than $k$ because
  \begin{equation*}
    \dim(\Gamma(f_m \restriction X'_m)) = \dim(X'_m) < k
  \end{equation*}
  by Theorem~\ref{dimension-theorem}(\ref{dt4},\ref{dt7}).  Therefore
  we can take $D_0 = D \setminus \bigcup_{i=1}^N C_i$.
\end{proof}
\begin{theorem}[Cell decomposition I] \label{cd1}
  If $D \subseteq \Mm^n$ is definable, then $D$ is a disjoint union of
  cells.
\end{theorem}
\begin{proof}
  Proceed by induction on $k = \dim(D)$.  By Lemma~\ref{pathetic-lem},
  we can write $D$ as a disjoint union $\coprod_{i=1}^N D_i$, where
  each $D_i$ has a near-injective $k$-projection.  By
  Lemma~\ref{just-now}, each $D_i$ can be written as a disjoint union
  of finitely many $k$-cells and a set of lower dimension.  Therefore,
  the same is true of $D$: we can write $D$ as $D_0 \sqcup C_1 \sqcup
  \cdots \sqcup C_N$, where each $C_i$ is a $k$-cell and $\dim(D_0) <
  \dim(D)$.  By induction, we can further decompose $D_0$ into cells.
\end{proof}
There is also a generic continuity version of cell decomposition,
Theorem~\ref{cd2} below.  Note that unlike the usual proof of cell
decomposition, we are not proving the two theorems together by a joint
induction.
\begin{definition}
  The \emph{local dimension} of a definable set $D$ at a point $p \in
  D$, written $\dim_p D$, is $\dim U \cap D$ for all sufficiently
  small neighborhoods $U \ni p$.
\end{definition}
Note that $\dim_p(D) \le \dim(D)$.
\begin{proposition} \label{local-dim}
  If $\dim(D) = k$, there is $p \in D$ such that $\dim_p(D) = k$.
\end{proposition}
\begin{proof}
  Write $D$ as a disjoint union of cells $D = \coprod_{i=1}^N C_i$.
  Some $C_i$ is a $k$-cell.  Take any $p \in C_i$.  Then $\dim_p(C_i)
  = k$ because $C_i$ looks locally like the graph of a continuous
  function.  Finally,
  \begin{equation*}
    k = \dim_p(C_i) \le \dim_p(D) \le \dim(D) = k.   \qedhere
  \end{equation*}
\end{proof}
\begin{lemma} \label{gorilla}
  Let $C$ be a $k$-cell, $f : C \rightrightarrows \Mm^m$ be a definable
  correspondence, and let $C_{\mathrm{bad}}$ be the set of $a \in C$ such that
  $f$ is discontinuous at $a$.  Then $\dim(C_{\mathrm{bad}}) < k = \dim(C)$.
\end{lemma}
\begin{proof}
  Otherwise, Proposition~\ref{local-dim} gives $a \in C_{\mathrm{bad}} \subseteq C$
  such that the local dimension $\dim_a(C_{\mathrm{bad}})$ equals $k$.  The
  cell $C$ is definably locally homeomorphic to $\Mm^k$.  Therefore,
  there is a definable homeomorphism between a neighborhood of $a$ in
  $C$ and an open set $U$ in $\Mm^k$.  Transferring $f$ along this
  homeomorphism, we get a definable correspondence $U \rightrightarrows \Mm^m$
  which is discontinuous on a broad set, contradicting generic
  continuity of correspondences on open sets (Proposition~\ref{gencon}).
\end{proof}
\begin{lemma} \label{chimp}
  Let $f : D \rightrightarrows \Mm^m$ be a definable correspondence and $k =
  \dim(D)$.  Then we can write $D$ as a disjoint union $D_0 \sqcup
  \coprod_{i=1}^N D_i$ where $\dim(D_0) < k$ and $f \restriction D_i$
  is continuous for $i > 0$.
\end{lemma}
\begin{proof}
  By cell decomposition (Theorem~\ref{cd1}), we can partition $D$ into
  cells, and then throw away the cells of dimension $< k$.  Then we
  reduce to the case where $D$ is a $k$-cell, which is handled by
  Lemma~\ref{gorilla}.
\end{proof}
\begin{lemma} \label{bonobo}
  Let $f : D \rightrightarrows \Mm^m$ be a definable correspondence.  Then we
  can write $D$ as a disjoint union $\coprod_{i=1}^N D_i$ where $f
  \restriction D_i$ is continuous for each $i$.
\end{lemma}
\begin{proof}
  By induction on $\dim(D)$ using Lemma~\ref{chimp}.
\end{proof}
\begin{theorem}[Cell decomposition II] \label{cd2}
  Let $D \subseteq \Mm^n$ be definable.
  \begin{enumerate}
  \item Let $f : D \to \Mm^m$ be a definable function.  Then we can
    write $D$ as a finite disjoint union of cells on which $f$ is
    continuous.
  \item More generally, let $f : D \rightrightarrows \Mm^m$ be a definable
    $k$-correspondence.  Then we can write $D$ as a finite disjoint union of
    cells on which $f$ is continuous.
  \end{enumerate}
\end{theorem}
\begin{proof}
  Functions are the same thing as 1-correspondences, so we only
  consider correspondences.  Use Lemma~\ref{bonobo} to partition $D$
  into finitely many pieces on which $f$ is continuous, and then use
  Cell Decomposition (Theorem~\ref{cd1}) to further partition these
  pieces into cells.
\end{proof}

\subsection{More dimension theory} \label{sec:mdt}
For visceral theories, we get a simpler description of dimension than
in t-minimal theories:
\begin{proposition} \label{inject-dim}
  If $D \subseteq \Mm^n$ is definable and $k \ge 0$, then the
  following are equivalent:
  \begin{enumerate}
  \item $\dim(D) \ge k$.
  \item There is a definable injection $f : B \hookrightarrow D$ where
    $B \subseteq \Mm^k$ is a ball.
  \end{enumerate}
\end{proposition}
\begin{proof}
  The implication (2)$\implies$(1) is clear by
  Theorem~\ref{dimension-theorem}.  For the implication
  (1)$\implies$(2), suppose $\dim(D) \ge k$ and take a cell
  decomposition $D = \coprod_{i=1}^N C_i$, where $C_i$ is a
  $k_i$-cell.  By Theorem~\ref{dimension-theorem}, there is some $i$
  such that $k_i \ge k$.  Then $C_i$ is definably locally homeomorphic
  to $\Mm^{k_i}$, so there is a ball $B' \subseteq \Mm^{k_i}$ with a
  definable injection into $C_i$.  It is easy to find a ball $B
  \subseteq \Mm^k$ with a definable injection $B \hookrightarrow B'$.
  The composition $B \hookrightarrow B' \hookrightarrow D$ gives (2).
\end{proof}
For the rest of the section, \textbf{assume $T$ has no space-filling
  functions}.  We first need to upgrade this to a statement about
``space-filling correspondences'':
\begin{lemma}[Assuming NSFF] \label{strong-nsff}
  Let $R \subseteq \Mm^n \times \Mm^m$ be a definable relation whose
  projection onto $\Mm^n$ has finite fibers and whose projection onto
  $\Mm^m$ has broad image.  Then $n \ge m$.
\end{lemma}
\begin{proof}
  Let $\dom(R)$ and $\im(R)$ be the projections of $R$ onto $\Mm^n$
  and $\Mm^m$.  For $a \in \dom(R)$, let $R(a)$ be the finite
  non-empty set $\{b \in \Mm^m : a \mathrel{R} b\}$.  Let $a \propto
  E$ mean that $E$ separates the points of $R(a) := \{b \in
  \Mm^m : a \mathrel{R} b\}$, in the sense that
  \begin{equation*}
    (b, b' \in R(a) \text{ and } b \mathrel{E} b') \implies b = b'.
  \end{equation*}
  For any $a \in \dom(R)$, we have $a \propto E$ for all sufficiently
  small $E$.

  If $b \in \im(R)$, let $b \bowtie E$ mean that there is $a \in
  \dom(R)$ such that $a \mathrel{R} b$ and $a \propto E$.  For any
  $b \in \im(R)$, we have $b \bowtie E$ for all sufficiently small
  $E$.  Since $\im(R)$ is broad, the Hammer Lemma gives a ball $B
  \subseteq \im(R)$ and an entourage $E$ such that $b \bowtie E$ for
  every $b \in B$.  Let
  \begin{equation*}
    R_1 = \{(a,b) \in R : b \in B \text{ and } a \propto E\}.
  \end{equation*}
  If $b \in B$ then $b \bowtie E$ so there is $a \in \dom(R)$ such
  that $a \mathrel{R} b$ and $a \propto E$.  Then $(a,b) \in R_1$ and
  $b \in \im(R_1)$.  This shows that $\im(R_1) = B$.
  
  Let $B' \subseteq B$ be an $E$-small ball, in the sense that $x,y
  \in B' \implies x \mathrel{E} y$.  Let
  \begin{align*}
    R_2 &= \{(a,b) \in R_1 : b \in B'\} \\ & = \{(a,b) \in R : b \in B'
    \text{ and } a \propto E\}.
  \end{align*}
  Then $\im(R_2)$ is the broad set $B'$.  If $(a,b)$ and $(a,b')$ are
  both in $R_2$ then $b,b' \in B'$, so $b \mathrel{E} b'$, and $a
  \propto E$, so $b = b'$.  Then $R_2$ is the graph of a function from
  $\dom(R_2) \subseteq \Mm^n$ onto the broad set $\im(R_2) = B'
  \subseteq \Mm^m$.  By NSFF, $n \ge m$.
\end{proof}
\begin{proposition}[Assuming NSFF]
  Let $\ba, \bb$ be finite tuples and $C$ be a small set of
  parameters.
  \begin{enumerate}
  \item If $\bb \in \acl(C\ba)$, then $\dim(\bb/C) \le \dim(\ba/C)$.
  \item In particular, if $\bb$ is a subtuple of $\ba$, then
    $\dim(\bb/C) \le \dim(\ba/C)$.
  \end{enumerate}
\end{proposition}
\begin{proof}
  \begin{enumerate}
  \item Replacing $\ba$ and $\bb$ with $\acl$-bases over $C$, we may
    assume that $\ba$ and $\bb$ are $\acl$-independent.  Then
    $\tp(\ba/C)$ and $\tp(\bb/C)$ are broad.  Let $n = |\ba| =
    \dim(\ba/C)$ and $m = |\bb| = \dim(\bb/C)$.  The fact that $\bb
    \in \acl(C\ba)$ means that $\bb = g(\ba)$ for some $C$-definable
    correspondence $g : D \rightrightarrows \Mm^m$, with $D \subseteq \Mm^n$.
    Then $\bb$ is in the $C$-definable set $\im(g)$, so $\im(g)$ must
    be broad.  By Lemma~\ref{strong-nsff} (applied to $R =
    \Gamma(g)$), we must have $m \le n$, i.e., $\dim(\bb/C) \le
    \dim(\ba/C)$.
  \item Clear, by part (1).  \qedhere
  \end{enumerate}
\end{proof}
\begin{theorem}[Assuming NSFF] \label{surjections}
  If $f : X \to Y$ is a type-definable surjection of type-definable
  sets, then $\dim(X) \ge \dim(Y)$.
\end{theorem}
\begin{proof}
  Take a small set $C \subseteq \Mm^\eq$ over which $f, X, Y$ are
  type-definable.  Take $\bb \in Y$ with $\dim(\bb/C) = \dim(Y)$.
  Take $\ba \in f^{-1}(\bb)$.  Then $\bb \in \dcl(C\ba)$, so
  \begin{equation*}
    \dim(X) \ge \dim(\ba/C) \ge \dim(\bb/C) = \dim(Y).  \qedhere
  \end{equation*}
\end{proof}
\begin{theorem}[Assuming NSFF] \label{naive}
  If $X \subseteq \Mm^n$ is definable, then $\dim(X) = d_t(X)$.
\end{theorem}
\begin{proof}
  If there is a definable surjection $X \to \Mm^k$ whose image has
  non-empty interior, then $\dim(X) \ge k$ by
  Theorem~\ref{surjections}.  This shows $\dim(X) \ge d_t(X)$.  The
  reverse inequality was Theorem~\ref{comparison}(3).
\end{proof}
Therefore, all the theorems on dimension hold for naive topological
dimension.
\begin{remark}
  The invariance of naive topological dimension under definable
  bijections \emph{cannot} hold unless NSFF holds.  Indeed, if $f : X
  \to \Mm^n$ is a space-filling function with $X \subseteq \Mm^m$ for
  some $m < n$, then
  \begin{itemize}
  \item The set $X$ has naive topological dimension at most $m$, as $X
    \subseteq \Mm^m$.
  \item The graph $\Gamma(f)$ has naive topological dimension at least
    $n$, because of the projection onto $\Mm^n$.
  \end{itemize}
  So $d_t(X) \le m < n \le d_t(\Gamma(f))$, in spite of the
  definable bijection between $X$ and $\Gamma(f)$.

  Similarly, Theorem~\ref{surjections} cannot hold unless NSFF holds
  (obviously).  So, for a visceral theory $T$, the following are
  equivalent:
  \begin{enumerate}
  \item There are no space-filling functions.
  \item $\dim(X) \ge \dim(Y)$ for any definable surjection $f : X \to
    Y$, where dimension is defined as in the current paper.
  \item Naive topological dimension is preserved under definable
    bijections.
  \end{enumerate}
\end{remark}

\section{Frontiers, generic continuity, and local Euclideanity} \label{frontier-sec}
As in Section~\ref{vt-section}, assume that $\Mm$ is a monster model
of a visceral theory $T$.
\subsection{Dimension of frontiers} \label{dof}
We now turn our attention to the dimension of the frontier $\partial X
= \overline{X} \setminus X$, for $X$ a definable set.  If $X \subseteq
\Mm^n$ is definable then $\partial X$ has empty interior, so
$\dim(\partial X) < n$ (by
Theorem~\ref{dimension-theorem}(\ref{dt5})).  Otherwise, not much can
be said, without further assumptions.
\begin{theorem} \label{another-NSFF}
  Suppose $T$ has NSFF.
  If $X \subseteq \Mm^n$ is definable, then $\dim(\partial X) \le
  \dim(X)$ and so $\dim(\overline{X}) = \dim(X)$.
\end{theorem}
\begin{proof}
  Otherwise $\dim(X) < \dim(\overline{X}) =: k$.  By
  Theorem~\ref{naive} there is a coordinate projection $\pi : \Mm^n
  \to \Mm^k$ such that $\pi(\overline{X})$ is broad, but $\pi(X)$ is
  narrow (since $d_t(X) < k$).  Let $C$ be the closure
  $\overline{\pi(X)}$; it is narrow by Corollary~\ref{small-closure}.
  Then $\pi^{-1}(C)$ is a closed set containing $X$, so $\pi^{-1}(C)
  \supseteq \overline{X}$ and $C \supseteq \pi(\overline{X})$.  But
  $\pi(\overline{X})$ is broad and $C$ is narrow, a contradiction.
\end{proof}
Recall that the exchange property implies NSFF.
\begin{theorem} \label{usual-frontier}
  If $T$ has the exchange property and $X \subseteq \Mm^n$ is
  definable, then $\dim(\partial X) < \dim(X)$.
\end{theorem}
This strengthens \cite[Corollary~3.35]{viscerality}, dropping the
assumption of definable finite choice.
\begin{proof}
  We follow the usual argument.
  By Theorem~\ref{comparison},
  $\dim(-)$ agrees with $\acl$-dimension.  Consequently, $\dim(-)$ is
  strictly additive: if $f : X \to Y$ is a definable function whose
  every fiber has dimension $k$, then $\dim(X) = k + \dim(Y)$.

  Let $k = \dim(\partial X) = d_t(\partial X)$.  Suppose for the sake
  of contradiction that $\dim(X) \le k$.  Take a coordinate projection
  $\pi : \Mm^n \to \Mm^k$ such that $\pi(\partial X)$ is broad,
  containing a ball $B$.  Permuting coordinates, we may assume $\pi$
  is the projection onto the first $k$ coordinates.  The set
  \begin{equation*}
    D = \{b \in  B : \pi^{-1}(b) \cap X \text{ is infinite}\}
  \end{equation*}
  must have dimension less than $k$, or else $X$ would have dimension
  at least $k+1$ by the additivity of dimension.  Then $D$ and
  $\overline{D}$ are narrow.  Shrinking $B$, we may assume $B \cap D =
  \varnothing$, and so $\pi^{-1}(b) \cap X$ is finite for any $b \in
  B$.  Let $B_i = \{b \in B : |\pi^{-1}(b) \cap X| = i\}$.  Then $B$
  is the union of the $B_i$'s, so one is broad.  Shrinking further, we
  may assume $B = B_i$ for some fixed $i$.  Then $\pi^{-1}(B) \cap X$
  is the graph of an $i$-correspondence $f : B \rightrightarrows
  \Mm^{n-k}$.  By generic continuity of correspondences
  (Proposition~\ref{gencon}(4)), we may shrink $B$ further and arrange
  that $f$ is continuous on $B$.  Then $\pi^{-1}(B) \cap X$ is the
  graph of a continuous $i$-correspondence on $B$, so $\pi^{-1}(B)
  \cap X$ is relatively closed in the open set $\pi^{-1}(B)$.  Then
  $\pi^{-1}(B) \cap \partial X = \varnothing$, contradicting the fact
  that $B \subseteq \pi(\partial X)$.
\end{proof}
Recall that dp-minimality implies NSFF
\cite[Proposition~3.31]{viscerality}.  Simon and Walsberg show the
following:
\begin{fact}[{\cite[Proposition~4.3]{simonWalsberg}}] \label{sw-fact}
  If $T$ is dp-minimal and $X \subseteq \Mm^n$ is definable, then
  $\dim(\partial X) < \dim(X)$.
\end{fact}
These facts can be applied to show that finite-to-finite
correspondences are local homeomorphisms, at most points:
\begin{proposition} \label{vincent-prop}
  Suppose $T$ is dp-minimal or has the exchange property.  Let $X_1,
  X_2$ be definable sets.  Let $R \subseteq X_1 \times X_2$ be a
  definable set projecting surjectively onto $X_1$ and $X_2$, such
  that both projections have finite fibers.  Then there is a definable
  set $R_0 \subseteq R$ such that $\dim(R \setminus R_0) < \dim(R)$,
  and for any $(a,b) \in R_0$, the set $R$ looks locally like the
  graph of a homeomorphism between a neighborhood of $a$ in $X_1$ and
  a neighborhood of $b$ in $X_2$.  The set $R_0$ can be chosen to be
  definable over the same parameters which define $R$.
\end{proposition}
\begin{proof}
  For $i = 1, 2$, let $R_i$ be the set of points $p \in R$ such that
  the $i$th projection $\pi_i : R \to X_i$ is a local homeomorphism at
  $p$.  It suffices to show that $\dim(R \setminus R_i) < \dim(R)$ for
  both $i$, because then we can take $R_0 = R_1 \cap R_2$.  By
  symmetry, we only need to consider $i=1$.

  For $a \in X$, let $g(a)$ be the number of $b \in Y$ such that
  $(a,b) \in R$.  By cell decomposition, we can partition $X$ into
  cells $X = \coprod_{i=1}^k C_i$ such that $g$ is constant on each
  cell.  Then $R$ is the graph of a correspondence, above each cell
  $C_i$.  Refining the cell decomposition, we can assume that $R$ is
  the graph of a continuous correspondence $h_i : C_i
  \rightrightarrows Y$, above each cell $C_i$.  Let $X_0 =
  \bigcup_{i=1}^k \partial C_i \subseteq X$, and let $X_1 = X
  \setminus X_0$.  Then \[\dim(X_0) = \max_{1 \le i \le k}
  \dim(\partial C_i) < \max_{1 \le i \le k} \dim(C_i) = \dim(X),
  \tag{$\ast$} \] where the strict inequality comes from
  Theorem~\ref{usual-frontier} or Fact~\ref{sw-fact}.  Then
  \begin{equation*}
    R_1 \supseteq \{(a,b) \in R : a \in X_1\}. \tag{$\dag$}
  \end{equation*}
  Indeed, if $(a,b) \in R$ and $a \in X_1$, then $a$ belongs to a
  unique cell $C_i$, and $a$ is not in the frontier of any cell, so
  $a$ is not in the closure of any cell other than $C_i$.  Locally
  around $(a,b)$, the set $R$ looks like the graph of the continuous
  correspondence $h_i : C_i \rightrightarrows Y$.  In particular, the
  projection $\pi_1 : R \to X$ is a local homeomorphism around
  $(a,b)$, so $(a,b) \in R_1$, proving ($\dag$).

  Finally, ($\dag$) implies that $\dim(R \setminus R_1) < \dim(R)$ by
  dimension theory.  Specifically, $R \setminus R_1$ has a
  finite-to-one surjection onto $X \setminus X_1 = X_0$, so $\dim(R
  \setminus R_1) = \dim(X_0) < \dim(X)$ by ($\ast$) and
  Theorem~\ref{dimension-theorem}(\ref{dt7}).
\end{proof}

\subsection{Frontier rank $d(-)$} \label{anti-dim}
From one point of view, the cell decomposition theorems
(Theorems~\ref{cd1}, \ref{cd2}) give a satisfactory description of
definable sets in visceral theories.  But on the other hand, these
results are unable to answer rather fundamental questions such as the
following:
\begin{enumerate}
\item Let $f : D \to \Mm^n$ be a definable function.  Is $f$ anywhere
  continuous?  That is, is there \emph{any} $a \in D$ such that $f$ is
  continuous at $a$?
\item Let $D \subseteq \Mm^n$ be definable.  Is $D$ locally Euclidean
  anywhere?  That is, is there a non-empty relatively open $U
  \subseteq D$ such that $U$ is definably homeomorphic to an open set
  in some $\Mm^k$?
\end{enumerate}
Cell decomposition lets us partition $D$ into cells $\coprod_{i=1}^N
C_i$.  Definable functions on cells are generically continuous
(Lemma~\ref{gorilla}), and cells are locally Euclidean.  But this
doesn't help answer the questions above unless we can find a point $p
\in D$ with the property that $p$ belongs to the closure
$\overline{C_i}$ of \emph{only one} cell.  The sort of configuration
we need to rule out is a cell decomposition $\coprod_{i=1}^N C_i$
where each $C_i$ is covered by the frontiers of the other cells.
\begin{example}
  For $n \ge 1$, we will see an example of a visceral theory
  containing a 1-dimensional set $D$ with $\dim(\partial D) = n$.  For
  $n=1$, we can even take the theory to have NSFF.  If $X = D \sqcup
  \partial D$ and $f : X \to \{0,1\}$ is the characteristic function
  \begin{equation*}
    f(x) = 
    \begin{cases}
      1 & x \in D \\
      0 & x \in \partial D
    \end{cases}
  \end{equation*}
  then $f$ is discontinuous on the set $\partial D$.  This shows that
  we cannot expect definable functions to be continuous off a set of
  low dimension.  In other words, Lemma~\ref{gorilla} does not
  generalize from cells to more general definable sets.
\end{example}
We will give positive answers to the two questions above, but it will
require more machinery.
\begin{definition}
  Let $X \subseteq \Mm^m$ be definable.  Then the \emph{frontier rank}
  $d(X)$ is the maximum $n$ such that there is a sequence of non-empty
  definable sets $Y_0, Y_1, \ldots, Y_n \subseteq \Mm^m$ with $Y_0
  \subseteq X$ and $Y_{i+1} \subseteq \partial Y_i$ for $0 \le i < n$.
  We set $d(X) = -\infty$ if no such $n$ exists, and $d(X) = +\infty$
  if there is no bound on $n$.
\end{definition}
Then $d(X)$ is characterized inductively as follows:
\begin{enumerate}
\item $d(X) \ge 0$ iff $X$ is non-empty.
\item $d(X) \ge n+1$ iff there is a definable subset $Y \subseteq X$
  with $d(\partial Y) \ge n$.
\end{enumerate}
We don't need to require $Y$ to be non-empty in (2) since $d(\partial
Y) \ge n$ implies $Y \ne \varnothing$.
\begin{remark}
  Taking $Y=X$ in (2) we see that if $d(\partial X) \ge n$ then $d(X)
  \ge n+1$.  In other words $d(\partial X) + 1 \le d(X)$.
\end{remark}
\begin{example} \label{d-example}
  If $X$ is finite and non-empty, then every definable subset $Y
  \subseteq X$ is closed ($\partial Y = \varnothing$), so $d(X) = 0$.
  If $Y \subseteq \Mm$ is definable, then $\partial Y$ is finite by
  t-minimality, so $d(\partial Y) \le 0$.  This shows that $d(\Mm) \le
  1$.  In fact, $d(\Mm) = 1$ because if $Y = \Mm \setminus \{a\}$ then $d(\partial Y) = d(\{a\}) = 0$.
\end{example}
\begin{lemma} \phantomsection \label{dlem0}
  \begin{enumerate}
  \item $d(\varnothing) = -\infty$.
  \item If $X_1 \subseteq X_2$ then $d(X_1) \le d(X_2)$.
  \item \label{unions} $d(X_1 \cup X_2) = \max(d(X_1),d(X_2))$.
  \end{enumerate}
\end{lemma}
\begin{proof}
  The first two points are clear, and imply $d(X_1 \cup X_2) \ge
  \max(d(X_1),d(X_2))$.  It remains to show that
  \begin{equation*}
    (d(X_1) < n \text{ and } d(X_2) < n) \implies d(X_1 \cup X_2) < n
    \tag{$\ast_n$}
  \end{equation*}
  for each $n \ge 0$.  We prove this by induction on $n$.  The base
  case $n = 0$ says that a union of two empty sets is empty, which is
  true.  Suppose $n > 0$.  Suppose $d(X_1) < n$ and $d(X_2) < n$.  Let $Y
  \subseteq X_1 \cup X_2$ be definable.  We must show $d(\partial Y) < n-1$.
  Note that $Y = (Y \cap X_1) \cup (Y \cap X_2)$, so
  \begin{gather*}
    \partial Y \subseteq \partial (Y \cap X_1) \cup \partial (Y \cap X_2).
  \end{gather*}
  Since $Y \cap X_i$ is a definable subset of $X_i$ and $d(X_i) <
  n$, we have $\partial(Y \cap X_i) < n-1$.  By induction
  ($\ast_{n-1}$), we have
  \begin{equation*}
    d(\partial Y) \le d(\partial (Y \cap X_1) \cup \partial (Y \cap X_2))
    < n-1. \qedhere
  \end{equation*}
\end{proof}
Say that a definable set $C \subseteq \Mm^m$ is \emph{locally
$n$-Euclidean} if around any point $a \in C$, there is a definable
local homeomorphism to an open set in $\Mm^m$.  For example, $n$-cells
are locally $n$-Euclidean.
\begin{lemma} \label{dlem}
  Suppose $d(\Mm^n) \le a$.  Let $C \subseteq \Mm^m$ be definable,
  locally closed, and locally $n$-Euclidean.
  \begin{enumerate}
  \item \label{dl1} If $C$ is closed, then $d(C) \le a$.
  \item \label{dl2} Otherwise, $d(C) \le a + d(\partial C) + 1$.
  \end{enumerate}
\end{lemma}
In the second part, one could probably remove the ``$+ 1$'' by being
more careful.
\begin{proof}
  \begin{enumerate}
  \item The idea is that a high value of $d(C)$ must be witnessed
    locally at a specific point $p \in C$, and then we can use local
    $n$-Euclideanity to get a similar high value for $d(\Mm^n)$.  In
    more detail, suppose $d(C) > a$.  Then there is a sequence of
    non-empty definable sets $Y_0, Y_1, \ldots, Y_{a+1} \subseteq
    \Mm^m$ with $Y_0 \subseteq X$ and $Y_{i+1} \subseteq \partial Y_i$
    for $0 \le i \le a$.  Note that
    \begin{equation*}
      Y_{i+1} \subseteq \partial Y_i \implies Y_{i+1} \subseteq
      \overline{Y_i} \implies \overline{Y_{i+1}} \subseteq
      \overline{Y_i},
    \end{equation*}
    so
    \begin{equation*}
      C = \overline{C} \supseteq \overline{Y_0} \supseteq
      \overline{Y_1} \supseteq \cdots \supseteq \overline{Y_{a+1}}.
    \end{equation*}
    In particular, each $Y_i$ is contained in $C$.  Take some $p \in
    Y_{a+1} \subseteq C$.  By local Euclideanity, there is an open
    neighborhood $U \ni p$ in $\Mm^m$ and a definable homeomorphism $f
    : C \cap U \to V$ for some open set $V \subseteq \Mm^n$.  Note
    that $Y_{i+1} \cap U \subseteq \partial(Y_i \cap U)$ for each $i$,
    and so
    \begin{equation*}
      f(Y_{i+1} \cap U) \subseteq \partial(f(Y_i \cap U)) \text{ for
        each } i.
    \end{equation*}
    Then the sets $Z_i := f(Y_i \cap U)$ show that $a+1 \le d(V) \le
    d(\Mm^n)$, contradicting the fact that $d(\Mm^n) \le a$.
  \item Let $W = \partial C$.  Then $W$ is closed, because $C$ is
    locally closed.  The disjoint union $C \cup W$ is $\overline{C}$,
    so it is also closed.  Let $b = d(W) = d(\partial C)$.  It
    suffices to show that $d(C \cup W) \le a + b + 1$, for then
    \begin{equation*}
      d(C) \le d(C \cup W) \le a + b + 1 = a + d(\partial C)+1.
    \end{equation*}
    So suppose for the sake of contradiction that $d(C \cup W) \ge a +
    b + 2$.  Then there is a sequence of definable sets $Y_0, Y_1,
    \ldots, Y_{a+b+2}$ with $Y_0 \subseteq C \cup W$ and $Y_{i+1}
    \subseteq \partial Y_i$.  Arguing as in the previous point, we
    have
    \begin{equation*}
      C \cup W = \overline{C \cup W} \supseteq \overline{Y_0}
      \supseteq \overline{Y_1} \supseteq \cdots \supseteq
      \overline{Y_{a+b+2}},
    \end{equation*}
    so each $Y_i$ is contained in $C \cup W$.  If $Y_{a+1} \subseteq
    W$, then the sequence of sets \[Y_{(a+1)+0}, Y_{(a+1)+1}, \ldots,
    Y_{(a+1)+(b+1)}\] shows that $d(W) \ge b+1$, contradicting the
    choice of $b$.  Thus $Y_{a+1} \not\subseteq W$, and there is some
    $p \in Y_{a+1} \cap C$.  Then arguing as in the previous point, we
    see that $d(\Mm^n) \ge a+1$, a contradiction.  \qedhere
  \end{enumerate}
\end{proof}
\begin{proposition}
  $d(\Mm^n) < \infty$ for every $n$.
\end{proposition}
\begin{proof}
  Define a function $f(n)$ by $f(1) = 1$ and $f(n) = n \cdot (f(n-1) +
  1)$ for $n > 1$.  Note that $f$ is increasing.  We will prove
  $d(\Mm^n) \le f(n)$ by induction on $n$.  The base case $n = 1$ says
  that $d(\Mm) \le 1$, which holds by t-minimality, as discussed in
  Example~\ref{d-example}.  Suppose $n > 1$.  By induction, we know
  that $d(\Mm^{n-1}) \le f(n-1)$.
  \begin{claim}
    If $1 \le k \le n$ and $X \subseteq \Mm^k$ is narrow, then \[d(X
    \times \Mm^{n-k}) \le f(n-1) + (k-1)(f(n-1)+1).\]
  \end{claim}
  \begin{claimproof}
    Proceed by induction on $k$.  If $k = 1$, then $X \subseteq \Mm$
    is finite, and $X \times \Mm^{n-1}$ is a union of finitely many
    parallel hyperplanes.  In particular, $X$ is locally
    $(n-1)$-Euclidean, so $d(X \times \Mm^{n-1}) \le f(n-1)$ by
    Lemma~\ref{dlem}(\ref{dl1}).

    Next suppose $k > 1$.  Write $X$ as
    a union of cells $\coprod_{i=1}^N C_i$.  If $d(C_i \times
    \Mm^{n-k}) \le f(n-1) + (k-1)(f(n-1)+1)$ for each $i$, then $d(X
    \times \Mm^{n-k}) \le f(n-1) + (k-1)(f(n-1)+1)$ because $d(-)$
    behaves correctly with respect to unions
    (Lemma~\ref{dlem0}(\ref{unions})).  Therefore, it suffices to
    handle the case where $X$ is a cell $C$ in $\Mm^k$.

    Let $j = \dim(C)$.  Then $j < k$, since $C$ is narrow in $\Mm^k$.
    Without loss of generality (permuting coordinates if necessary),
    we may assume that $C$ is the graph $\Gamma(g)$ of a continuous
    correspondence $g : U \rightrightarrows \Mm^{k-j}$ for some
    non-empty open $U \subseteq \Mm^j$.  Let $W$ be the set $\partial
    U \times \Mm^{k-j}$.  Note that $W \cap C = \varnothing$ and $W
    \cup C$ is closed, so $W \supseteq \partial C$.  Letting $W' = W
    \times \Mm^{n-k}$ and $C' = C \times \Mm^{n-k}$, then similarly
    $W' \cap C' = \varnothing$ and $W' \cup C'$ is closed, so that $W'
    \supseteq \partial C'$.  The cell $C$ is locally $j$-Euclidean, so
    the set $C'$ is locally $(j+n-k)$-Euclidean, where $j+n-k < n$.
    By the outer inductive hypothesis, $d(\Mm^{j+n-k}) \le f(j+n-k)
    \le f(n-1)$, and thus Lemma~\ref{dlem} shows
    \begin{equation*}
      d(C') \le f(n-1) + d(\partial C) + 1 \le f(n-1) + d(W') + 1.
    \end{equation*}
    Since $\partial U$ is
    narrow in $\Mm^j$ (Corollary~\ref{small-boundary-0}) and $j < k$,
    we know by induction that
    \begin{align*}
      d(W') = d(\partial U \times \Mm^{k-j} \times
      \Mm^{n-k}) &\le f(n-1) + (j-1)(f(n-1)+1) \\ &\le
      f(n-1)+(k-2)(f(n-1)+1).
    \end{align*}
    Putting everything together,
    \begin{align*}
      d(C') &\le f(n-1) + f(n-1) + (k-2)(f(n-1)+1) + 1 \\ &= f(n-1) +
      (k-1)(f(n-1) + 1)
    \end{align*}
    as desired.
  \end{claimproof}
  Taking $k = n$ in the Claim, we see that if $X \subseteq \Mm^n$ is a
  narrow definable set, then
  \begin{align*}
    d(X) = d(X \times \Mm^0) &\le f(n-1) + (n-1)(f(n-1)+1) \\
    &= n(f(n-1)+1) - 1 = f(n) - 1.
  \end{align*}
  Finally, if $Y \subseteq \Mm^n$ is any definable set, then $\partial
  Y$ is narrow (Corollary~\ref{small-boundary-0}), so $d(\partial Y)
  \le f(n) - 1$.  By definition of $d(-)$ this shows $d(\Mm^n) \le
  f(n)$.
\end{proof}
To summarize, we have assigned a number $d(X) \in \Nn \cup
\{-\infty\}$ to every definable set $X \subseteq \Mm^n$, with the
following properties:
\begin{gather*}
  d(X) = -\infty \iff X = \varnothing \\
  d(\partial X) + 1 \le d(X) \\
  X \subseteq Y \implies d(X) \le d(Y) \\
  d(X \cup Y) = \max(d(X),d(Y)).
\end{gather*}
\subsection{Applications of frontier rank} \label{sec:afr}
\begin{theorem} \label{gencon3}
  Let $f : X \rightrightarrows \Mm^n$ be a definable correspondence on
  a definable set $X$.  Then the set of points where $f$ is continuous
  is dense in $X$.
\end{theorem}
\begin{proof}
  Take a cell decomposition $X = \coprod_{i=1}^N C_i$ such that $f$ is
  continuous on each $C_i$ (using Theorem~\ref{cd2}).  Note that
  $X \not \subseteq \bigcup_{i=1}^N \partial C_i$ because
  \begin{equation*}
    d(X) = \max_{1 \le i \le N} d(C_i) > \max_{1 \le i \le N}
    d(\partial C_i) = d\left(\bigcup_{i=1}^N \partial C_i\right).
  \end{equation*}
  Take a point $p \in X \setminus \bigcup_{i=1}^N \partial C_i$.
  Without loss of generality, $p \in C_1$.  Then $p \notin
  \overline{C_i}$ for $2 \le i \le N$, so there is a neighborhood $U$
  of $p$ such that $U \cap X = U \cap C_1$.  Since $f$ is continuous
  on $C_1$, $f$ is continuous on a neighborhood of $p$.

  This gives \emph{one} point of continuity.  If $U$ is an open set which
  intersects $X$, we can repeat the argument on $f \restriction X \cap
  U$, to see that $f$ is continuous at some point in $X \cap U$.
  Therefore, the set of continuous points is dense.
\end{proof}
\begin{theorem} \label{dense-filter}
  Let $X \subseteq \Mm^n$ be a definable set.  If $Y, Z \subseteq X$
  are definable subsets, dense in $X$, then their intersection $Y \cap
  Z$ is also dense in $X$.  Consequently, the family of dense
  definable subsets of $X$ is a filter.
\end{theorem}
\begin{proof}
  Suppose $Y$ and $Z$ are dense in $X$ but $Y \cap Z$ is not.  Working
  within the subspace $X$, there is a non-empty open set $U$ which is
  disjoint from $Y \cap Z$.  But $Y$ is dense, so $Y \cap U$ must be
  non-empty and dense in $U$ .  Similarly, $Z \cap U$ is non-empty and
  dense in $U$.  Since $(Y \cap U) \cap (Z \cap U) = (Y \cap Z) \cap U
  = \varnothing$, the two sets $Y \cap U$ and $Z \cap U$ must be
  contained in each other's frontiers:
  \begin{gather*}
    Y \cap U \subseteq \partial(Z \cap U) \\
    Z \cap U \subseteq \partial(Y \cap U).
  \end{gather*}
  This remains true when we regard $Y \cap U$ and $Z \cap U$ as
  subsets of the ambient space $\Mm^n$.  Then
  \begin{gather*}
    d(Y \cap U) < d(Z \cap U) \\
    d(Z \cap U) < d(Y \cap U)
  \end{gather*}
  which is absurd.
\end{proof}
Say that a definable subset $Y \subseteq X$ is \emph{codense} in $X$
if the complement $X \setminus Y$ is dense.  Equivalently, $Y
\subseteq X$ is codense iff the relative interior $\ter_X Y$ is empty.
\begin{corollary} \label{codense-ideal}
  Let $X$ be definable.
  \begin{enumerate}
  \item Codense sets are an ideal: a union of two codense sets is
    codense.
  \item \label{ci2} If $Y \subseteq X$ is definable, then the relative frontier
    $\partial_X Y = (\partial Y) \cap X$ is codense.
  \item \label{ci3} If a definable set $Y \subseteq X$ is codense, then the relative closure
    $\overline{Y} \cap X$ is codense.
  \item \label{ci4} If a definable set $Y \subseteq X$ is dense, then the relative interior
    $\ter_X Y$ is dense.
  \item \label{ci5} A definable set $Y \subseteq X$ is codense in $X$
    iff it is nowhere dense in $X$.
  \end{enumerate}
\end{corollary}
\begin{proof}
  \begin{enumerate}
  \item Immediate from Theorem~\ref{dense-filter}.
  \item Recall that $Y$ is codense iff the relative interior $\ter_X Y$ is empty.  In any topological space, frontiers have empty interior.
  \item The relative closure is $Y \cup \partial_X Y$.
  \item This is dual to the previous point.
  \item If $Y$ is codense, then (\ref{ci3}) says the relative closure
    $\overline{Y} \cap X$ is codense, meaning that it has empty
    interior in $X$.  Thus $Y$ is nowhere dense in $X$.  Conversely, if $Y$ is
    nowhere dense in $X$, then it has empty interior in $X$.  \qedhere
  \end{enumerate}
\end{proof}
\begin{example} \label{vincent-example-1}
  Combining Theorem~\ref{gencon3} with
  Corollary~\ref{codense-ideal}(\ref{ci4}), we see that if $f : X
  \rightrightarrows \Mm^n$ is a definable correspondence, then there
  is a dense relatively open definable subset $X' \subseteq X$ such
  that $f \restriction X'$ is continuous.
\end{example}
\begin{theorem} \label{eu-thm}
  If $X \subseteq \Mm^n$ is definable, then there is a dense
  relatively open definable subset $X_{\mathrm{Eu}} \subseteq X$ such
  that $X_{\mathrm{Eu}}$ is locally Euclidean.
\end{theorem}
(However, the local dimension of $X_{\mathrm{Eu}}$ may vary---the set
$X_{\mathrm{Eu}}$ need not be locally $n$-Euclidean for some fixed
$n$.)
\begin{proof}
  Take a cell decomposition $X = \coprod_{i=1}^N C_i$.  Cells are
  locally closed (in the ambient space $\Mm^n$), so their frontiers
  $\partial C_i \subseteq \Mm^n$ are closed.  Let
  \begin{align*}
    X_{\mathrm{Eu}} &= X \setminus \bigcup_{i=1}^N \partial C_i \\
    &= X \setminus \bigcup_{i=1}^N \partial_X C_i.
  \end{align*}
  By the first line, $X_{\mathrm{Eu}}$ is a definable, relatively open
  subset of $X$.  By the second line and
  Corollary~\ref{codense-ideal}, it is dense.  If $p \in
  X_{\mathrm{Eu}}$, then as in the proof of Theorem~\ref{gencon3}
  there is a neighborhood $U \ni p$ such that $U \cap X = U \cap C_i$
  for some cell $C_i$.  Since cells are locally Euclidean, this shows
  that $X_{\mathrm{Eu}}$ is locally Euclidean.
\end{proof}
In certain cases, the ideal of codense (= nowhere dense) sets agrees with the earlier
ideal of low-dimensional sets:
\begin{proposition}
  Suppose $X$ is a definable set which is locally $n$-Euclidean.
  \begin{enumerate}
  \item $\dim(X) = n$.
  \item A subset $Y \subseteq X$ is codense iff $\dim(Y) < n$.
  \end{enumerate}
\end{proposition}
\begin{proof}
  When $X$ is $\Mm^n$ or an open subset of $\Mm^n$, this holds by
  Proposition~\ref{triad}, because ``codense'' is the same as ``empty
  interior''.

  Next we consider the general case.  The local dimension $\dim_p(X)$
  equals $n$ at every point, so the first point holds by
  Proposition~\ref{local-dim}.  For the second point, first suppose
  that $Y$ fails to be codense.  Then $Y$ contains a non-empty
  relatively open subset $U \subseteq X$, so $\dim(Y) \ge \dim(U) =
  n$.  Conversely, suppose that $\dim(Y) = n$.  By
  Proposition~\ref{local-dim}, there is a point $p \in Y \subseteq X$
  such that the local dimension $\dim_p(Y)$ equals $n$.  Take an open
  neighborhood $U \ni p$ such that $U \cap X$ is definably
  homeomorphic to an open set $V \subseteq \Mm^n$.  Then $\dim(U \cap
  Y) = n$.  The corresponding set in $V$ has dimension $n$, so it has
  non-empty interior.  This implies $U \cap Y$ has non-empty relative
  interior, and then so does $Y$.
\end{proof}
In other cases, we get a one-way implication between the two ideals:
\begin{proposition}
  Suppose the theory $T$ is dp-minimal or has the exchange property.
  Suppose that $X$ is an $n$-dimensional definable set.
  \begin{enumerate}
  \item If $Y \subseteq X$ is dense, then $\dim(X \setminus Y) < n$.
  \item If $Y \subseteq X$ is codense, then $\dim(Y) < n$.
  \end{enumerate}
\end{proposition}
\begin{proof}
  The two statements are dual, so it suffices to prove (1).  For part
  (1), if $Y$ is dense in $X$, then $X \setminus Y \subseteq \partial
  Y$, so \[\dim(X \setminus Y) \le \dim(\partial Y) < \dim(Y) \le
  \dim(X).\] The strict inequality holds by
  Theorem~\ref{usual-frontier} if $T$ has the exchange property, or
  Fact~\ref{sw-fact} in the dp-minimal case.
\end{proof}
\begin{example}
  Suppose $T$ has the exchange property or is dp-minimal.
  \begin{enumerate}
  \item If $f : X \rightrightarrows \Mm^n$ is a definable
    correspondence, then there is a dense relatively open definable
    subset $X' \subseteq X$ with $\dim(X \setminus X') < \dim(X)$,
    such that $f \restriction X'$ is continuous
    (Example~\ref{vincent-example-1}).  This example underlies the
    proof of Proposition~\ref{vincent-prop}.
  \item If $X \subseteq \Mm^n$ is definable, then there is a dense
    relatively open definable subset $X_{\mathrel{Eu}} \subseteq X$
    such that $\dim(X \setminus X_{\mathrel{Eu}}) < \dim(X)$ and
    $X_{\mathrel{Eu}}$ is locally Euclidean (Theorem~\ref{eu-thm}).
  \end{enumerate}
\end{example}

\subsection{Nice t-minimal theories}
\begin{definition}
  A t-minimal theory is \emph{nice} if correspondences are generically
  continuous, in the following sense: if $D \subseteq \Mm^n$ is a
  non-empty open definable set, and $f : D \rightrightarrows \Mm^k$ is
  a definable correspondence, then 
  \begin{equation*}
    \dim \{a \in D : f \text{ is not continuous at } a\} < \dim(D).
  \end{equation*}
\end{definition}
Visceral theories are nice, by Proposition~\ref{gencon}.  Other
examples of nice t-minimal theories include dense o-minimal and dense
C-minimal theories.
\begin{remark}
  By inspecting the proofs, one sees that \emph{all} the results of
  \S\ref{vt-section}--\ref{frontier-sec} generalize from visceral
  theories to nice theories, \textsc{EXCEPT} the following:
  \begin{itemize}
  \item Lemma~\ref{strong-nsff} through the end of \S\ref{vt-section},
    because the proof of Lemma~\ref{strong-nsff} uses entourages.
  \item Theorem~\ref{another-NSFF}, which depends on
    Lemma~\ref{strong-nsff}.
  \end{itemize}
  Perhaps the proofs of these results could be modified to work for
  nice t-minimal theories, but I couldn't see how to do it easily.
\end{remark}

\section{Counterexamples} \label{cxsec}

In this section, we give several examples of pathological visceral
theories.
\begin{itemize}
\item In Section~\ref{sec:rcvf3} we give a simple example of a
  visceral theory in which NSFF fails---there is a space filling
  function.
\item In Section~\ref{sec:rcvf2} we give an example of a visceral
  theory with NSFF, in which the frontier dimension inequality
  ($\dim(\partial X) < \dim(X)$) fails---there is a 1-dimensional set
  whose frontier is 1-dimensional.
\item In Section~\ref{sec:ndep} we give a class of visceral theories
  with space-filling functions, showing \emph{arbitrarily bad}
  failures of the frontier dimension inequality and the surjection
  dimension inequality ($\dim(\dom(f)) \ge \dim(\im(f))$).
\end{itemize}
Our examples will come from multi-sorted theories, so in
Section~\ref{multi-sort} we discuss how ``multi-sorted visceral
theories'' can be converted to 1-sorted visceral theories of the kind
we have been considering so far.

\begin{remark}
  The examples in this section are chosen to have DFC, in order to
  make the point that DFC does not prevent the pathologies.
\end{remark}

\subsection{Multi-sorted visceral structures} \label{multi-sort}
The definition of viscerality implicitly assumes a 1-sorted language.
However, the definition would make sense in a multi-sorted theory as
well.
\begin{definition}
  Let $T$ be a complete theory in a multi-sorted language.  Let $\Mm$
  be a monster model.
  \begin{enumerate}
  \item A \emph{visceral uniformity} on a definable set $X$ is a
    definable uniformity on $X$ such that for any definable set $D
    \subseteq X$, $D$ is infinite if and only if $\ter(D) \ne
    \varnothing$.
  \item $T$ is \emph{visceral} if there is a visceral uniformity on
    each sort.
  \end{enumerate}
\end{definition}
One can probably generalize the results of this paper (and \cite{viscerality}) to the setting of a definable set with a visceral
uniformity, though I haven't checked the details.\footnote{The key
point---the reason the arguments should generalize---is that we've
never used the fact that the basis $\mathcal{B}$ was parameterized by
tuples from the home sort.}  At any rate, multi-sorted
visceral theories can be converted to 1-sorted visceral theories, as
long as the number of sorts is finite:
\begin{proposition}\label{to-1-sort}
  Let $T$ be a complete visceral theory with finitely many sorts.  Let
  $\Mm$ be a monster model of $T$, with sorts $X_1, \ldots, X_n$.
  Assume the $X_i$ are pairwise disjoint.  Form a 1-sorted structure
  $\Mm' = X_1 \cup \cdots \cup X_n$ with the induced structure making
  $\Mm$ and $\Mm'$ bi-interpretable.  Then the theory of $\Mm'$ is
  visceral.
\end{proposition}
\begin{proof}
  Let $\Omega_i$ be a visceral uniformity on $X_i$ and let
  $\mathcal{B}_i$ be a uniformly definable basis.  Let
  \begin{gather*}
    \Omega = \left\{\bigcup_{i = 1}^n E_i ~ \middle| ~ (E_1,\ldots,E_n)
    \in \prod_{i = 1}^n \Omega_i\right\} \\
    \mathcal{B} = \left\{\bigcup_{i = 1}^n E_i ~ \middle| ~
    (E_1,\ldots,E_n) \in \prod_{i = 1}^n \mathcal{B}_i\right\} \\
  \end{gather*}
  Then $\Omega$ is a uniformity on $\bigcup_{i=1}^n X_i$ and
  $\mathcal{B}$ is a basis.  The basis $\mathcal{B}$ is uniformly
  definable in $\Mm'$.  The induced topology on $\bigcup_{i = 1}^n
  X_i$ is the disjoint union topology.  If $D \subseteq \Mm'$ is
  definable, then we can write $D$ as $\bigcup_{i = 1}^n D_i$ for
  definable sets $D_i \subseteq X_i$.  In the disjoint union topology,
  we have $\ter(D) = \bigcup_{i = 1}^n \ter(D_i)$, and so
  \begin{equation*}
    \ter(D) \ne \varnothing \iff \left(\exists i : \ter(D_i) \ne
    \varnothing\right) \iff \left(\exists i : |D_i| = \infty\right)
    \iff |D| = \infty.
  \end{equation*}
  Thus $\Omega$ is a visceral uniformity.
\end{proof}
In the following sections, we will give examples of multi-sorted
visceral theories with pathological behavior.
Proposition~\ref{to-1-sort} converts these examples into genuine
1-sorted visceral theories.

\subsection{3-sorted RCVF with an angular component} \label{sec:rcvf3}
For background on valued fields, see \cite{PE}.
Let $(K,v)$ be a valued field with value group $\Gamma$ and residue
field $k$.  Recall that an \emph{angular component} is a group
homomorphism $\ac : K^\times \to k^\times$ extending the residue map
on elements of valuation 0.

Let $\RCVF_{\ac,3}$ be the theory of henselian valued fields with an
angular component, such that the residue field is real closed and the
value group is divisible, in the three-sorted language with sorts $K,
\Gamma, k$, with the following functions and relations:
\begin{itemize}
\item The ring structure on $K$.
\item The ordered ring structure on $k$.
\item The ordered group structure on $\Gamma$.
\item The valuation map $v : K^\times \to \Gamma$ and the angular
  component map $\ac : K^\times \to k^\times$.
\end{itemize}
A representative model of $\RCVF_{\ac,3}$ is given by the field of Hahn
series $\Rr((t^{\Qq}))$, formal power series $\sum_{i \in I} c_it^i$
with $I$ a well-ordered subset of $\Qq$ and $c_i$ from $\Rr$.
Specifically, $K = \Rr((t^{\Qq}))$, $\Gamma = \Qq$, $k = \Rr$, and if
$x$ is a Hahn series with leading term $ct^i$, then $\ac(x) = c$ and
$v(x) = i$.

The following (well-known) fact is a consequence of Pas's relative
quantifier elimination for equicharacteristic 0 henselian valued
fields \cite[Theorem~4.1]{Pas}, combined with the well-known
quantifier elimination theorems for DOAG and RCF.
\begin{fact}
  $\RCVF_{\ac,3}$ has quantifier elimination.
\end{fact}
If $(K,\Gamma,k) \models \RCVF_{\ac,3}$, then $K \models \RCF$.  In
particular, $K$ has a natural order topology.  This order topology
agrees with the valuation topology from $v$.  The maps $\ac(-)$ and
$v(-)$ are locally constant on $K^\times$, with respect to this
topology.

Quantifier elimination tells us something about definable sets in one
variable:
\begin{lemma} \label{rcvf3-key}
  Let $M = (K,\Gamma,k)$ be a model of $\RCVF_{\ac,3}$.  Let $A =
  (K_0,\Gamma_0,k_0)$ be a substructure.
  \begin{enumerate}
  \item If $D \subseteq K$ is $A$-definable, then $\bd(D)$ is a finite
    subset of $K_0^{alg} \cap K$.
  \item If $D \subseteq \Gamma$ is $A$-definable, then $\bd(D)$ is a
    finite subset of $\Qq \cdot \Gamma_0$.
  \item If $D \subseteq k$ is $A$-definable, then $\bd(D)$ is a finite
    subset of $k_0^{alg} \cap k$.
  \end{enumerate}
\end{lemma}
\begin{proof}
  Only (1) requires much work.  Let $x$ be a variable in the sort $K$.
  Let $\mathcal{L}(A)$ be the base language expanded by constants for
  the elements of $A$.  By induction on the complexity of terms and
  formulas, we see the following:
  \begin{itemize}
  \item If $t(x)$ is an $\mathcal{L}(A)$-term in the sort $K$, then $t(x)$ is a
    polynomial in $K_0[x]$.
  \item If $t(x)$ is an $\mathcal{L}(A)$-term in the sort $\Gamma$,
    then $t : K \to \Gamma$ is locally constant, off a finite subset
    of $K_0^{alg} \cap K$.
  \item If $t(x)$ is an $\mathcal{L}(A)$-term in the sort $k$, then $t
    : K \to k$ is locally constant, off a finite subset of
    $K_0^{alg} \cap K$.
  \item If $\phi(x)$ is a quantifier-free $\mathcal{L}(A)$-formula,
    then the truth-value of $\phi(x)$ is locally constant, off a
    finite subset of $K_0^{alg} \cap K$.
  \end{itemize}
  By quantifier-elimination, the fourth point says exactly what we
  want.
\end{proof}
\begin{proposition}
  $\RCVF_{\ac,3}$ is a visceral theory with DFC and a space-filling
  function.
\end{proposition}
\begin{proof}
  Each of the three sorts $K, \Gamma, k$ is linearly ordered, so DFC
  holds.  Each sort expands an ordered abelian group, so the order
  topology lifts to a definable uniformity.  Definable sets have
  finite boundary by Lemma~\ref{rcvf3-key}.  Each sort is densely
  ordered, so there are no isolated points.  This proves viscerality.
  Finally, the map $K^\times \to \Gamma \times k^\times$ given by $x
  \mapsto (v(x),\ac(x))$ is a space-filling function.
\end{proof}
For use in the next section, we note the following characterization of
$\acl(-)$, which is an immediate consequence of Lemma~\ref{rcvf3-key}:
\begin{lemma}\label{acl-characterization}
  Let $(K,\Gamma,k)$ be a model of $\RCVF_{\ac,3}$.  Suppose $A
  \subseteq K$, $B \subseteq \Gamma$, and $C \subseteq k$.  Let $L$ be
  the subfield of $K$ generated by $A$, and let $v(L)$ and $\ac(L)$ be
  the images of $L$ under the valuation and angular component maps.
  \begin{enumerate}
  \item An element $x \in K$ is model-theoretically algebraic over
    $ABC$ if and only if $x$ is field-theoretically algebraic over
    $A$.
  \item An element $x \in \Gamma$ is model-theoretically algebraic
    over $ABC$ if and only if $x$ is in the $\Qq$-linear span of $v(L)
    \cup B$.
  \item An element $x \in k$ is model-theoretically algebraic over
    $ABC$ if and only if $x$ is field-theoretically algebraic over
    $\ac(L) \cup C$.
  \end{enumerate}
\end{lemma}

\subsection{2-sorted RCVF with an angular component} \label{sec:rcvf2}
Let $\RCVF_{\ac,2}$ be the reduct of $\RCVF_{\ac,3}$ with only the
sorts $K$ and $k$.  The value group and valuation map are
interpretable in $\RCVF_{\ac,2}$, so the two theories are
bi-interpretable.
\begin{proposition}
  $\RCVF_{\ac,2}$ is a visceral theory with DFC.
\end{proposition}
\begin{proof}
  Immediate from the fact that $\RCVF_{\ac,3}$ is visceral.
\end{proof}
Consequently, the dimension theory for visceral theories applies.
\begin{proposition}
  In $\RCVF_{\ac,2}$, there is a 1-dimensional set $C$ whose frontier
  $\partial C$ also has dimension 1.
\end{proposition}
\begin{proof}
  Let $C = \{(x,\ac(x)) : x \in K^\times\} \subseteq K \times k$.
  \begin{claim}
    $\partial C = \{0\} \times k$.
  \end{claim}
  \begin{claimproof}
    Let
    $D = \{0\} \times k$.  Then $C \cup D$ is closed, because $D$ is
    closed and $C$ is the graph of a continuous function, hence locally
    closed in $K^\times \times k$.  If $\alpha \in k^\times$, then there
    are $x \in K$ arbitrarily close to 0 with $\ac(x) = \alpha$;
    therefore $(0,\alpha) \in \overline{C}$.  It follows that
    $\overline{C}$ contains $\{0\} \times k^\times$, and therefore also
    contains $\overline{\{0\} \times k^\times} = D$.  So we see that $D
    \subseteq \overline{C} \subseteq C \cup D$, which implies that
    $\overline{C} = C \cup D$ and $\partial C = D$.
  \end{claimproof}
  Both $C$ and $\partial C$ have dimension 1, because they are in
  definable bijection with infinite subsets of $K^\times$ and
  $k^\times$, respectively.
\end{proof}
Lastly, we will prove that $\RCVF_{\ac,2}$ has NSFF.  The proof is a
bit complicated.  The strategy is to show that $\dim(\ba/C)$ satisfies
the anti-pathological property
\begin{equation}
  \dim(\bar{a}/C) \le \dim(\bar{a}\bar{b}/C), \label{target-equality}
\end{equation}
from which NSFF follows easily.  To prove (\ref{target-equality}), we
use Proposition~\ref{dimprops} to explicitly calculate $\dim(-/-)$ in
terms of transcendence degrees.

Fix a monser model $\Mm = (K,k) \models \RCVF_{\ac,2}$, and let
$\Gamma$ be the value group.  If $A$ is a subset of $K$ or $k$, let
$\langle A \rangle$ denote the subfield of $K$ or $k$ generated by
$A$.  If $L$ is a subfield of $K$, we let $v(L)$ and $\ac(L)$ denote
the image of $L$ under the valuation or angular component,
respectively.  We let $\res(L)$ be the residue field of $L$.  The set
$v(L)$ is always a subgroup of $\Gamma$, and $\res(L)$ is always a
subfield of $k$.  The set $\ac(L)$ is a multiplicative subgroup of
$k^\times$ containing $\res(L)$, but not necessarily a subfield.
\begin{lemma}\label{acl-char-2}
  Suppose $A \subseteq K$ and $B \subseteq k$.  Then $\acl(AB)$
  consists of $\langle A \rangle^{alg} \cap K$ in the field sort, and
  $\langle \ac(\langle A \rangle) \cup B\rangle^{alg} \cap k$ in the
  residue field sort.
\end{lemma}
\begin{proof}
  Lemma~\ref{acl-characterization} transfers from $\RCVF_{\ac,3}$ to
  $\RCVF_{\ac,2}$ thanks to the bi-interpretation.
\end{proof}
In particular, if $A \subseteq K$ and $B \subseteq k$, then $(A,B)$ is
(model-theoretically) algebraically closed if and only if the
following three conditions hold:
\begin{itemize}
\item $A$ is a relatively algebraically
closed subfield of $K$.
\item $\ac(A) \subseteq B$.
\item $B$ is a relatively algebraically closed subfield of $k$.
\end{itemize}

\begin{lemma} \label{abhyankaresque}
  Let $K_0 \subseteq K_1 \subseteq K$ be subfields.  Suppose
  $\trdeg(K_1/K_0)$ is finite.  Then
  \begin{equation*}
    \trdeg(\langle \ac(K_1) \rangle / \langle \ac(K_0) \rangle) \le
    \trdeg(K_1/K_0).
  \end{equation*}
  In particular, the left hand side is finite.
\end{lemma}
\begin{proof}
  Let $k_i = \res(K_i)$ for $i = 0, 1$.  Let $\Gamma_i = \Qq \cdot
  v(K_i)$ for $i = 0, 1$.  Then $k_0 \subseteq k_1 \subseteq k$ and
  $\Gamma_0 \subseteq \Gamma_1 \subseteq \Gamma$.  By Abhyankar's
  inequality \cite{Abhyankar}, 
  \begin{equation*}
    \trdeg(k_1/k_0) + \dim_{\Qq}(\Gamma_1/\Gamma_0) \le \trdeg(K_1/K_0).
  \end{equation*}
  In particular, the values on the left hand side are finite.  Let $n
  = \trdeg(k_1/k_0)$ and $m = \dim_{\Qq}(\Gamma_1/\Gamma_0)$.  Then
  $n+m \le \trdeg(K_1/K_0)$.  Take $\alpha_1, \ldots, \alpha_n \in
  k_1$ a transcendence basis over $k_0$.  Take $\gamma_1, \ldots,
  \gamma_m \in v(K_1)$ inducing a $\Qq$-linear basis of
  $\Gamma_1/\Gamma_0$.  Take $b_i \in K$ with $v(b_i) = \gamma_i$.
  Let $\beta_i = \ac(b_i)$.
  \begin{claim}
    The set $S = \{\alpha_1,\ldots,\alpha_n,\beta_1,\ldots,\beta_m\}$
    contains a transcendence basis of $\langle \ac(K_1) \rangle$ over
    $\langle \ac(K_0) \rangle$.
  \end{claim}
  \begin{claimproof}
    It suffices to show that every element of $\ac(K_1)$ is algebraic
    over $S$ and $\ac(K_0)$.  Take $\alpha \in \ac(K_1)$.  Then
    $\alpha = \ac(a)$ for some $a \in K_1$.  Then $v(a) \in v(K_1)$,
    so $v(a)$ is in the $\Qq$-linear span of $v(K_0) \cup
    \{\gamma_1,\ldots,\gamma_m\}$.  Replacing $\alpha$ and $a$ with
    $\alpha^n$ and $a^n$, we may assume that $v(a)$ is in the
    $\Zz$-linear span of $v(K_0) \cup \{\gamma_1,\ldots,\gamma_m\}$:
    \begin{gather*}
      v(a) = v(b) + k_1 \gamma_1 + \cdots + k_m \gamma_m \\
      b \in K_0, ~ k_1,\ldots,k_m \in \Zz.
    \end{gather*}
    Let $a' = a/(b b_1^{k_1} \cdots b_m^{k_m})$.  Then $v(a') = 0$ and
    \begin{equation*}
      \alpha' := \ac(a') = \frac{\ac(a)}{\ac(b)\prod_{i = 1}^m
        \ac(b_m)^{k_m}} = \alpha/(\ac(b)\beta_1^{k_1} \cdots \beta_m^{k_m}),
    \end{equation*}
    where $\ac(b) \in \ac(K_0)$.  Then it remains to show that
    $\alpha'$ is algebraic over $S \cup \ac(K_0)$.  Because $v(a') =
    0$, we have $\alpha' = \ac(a') = \res(a') \in \res(K_1)$.  Then
    $\alpha'$ is algebraic over $\{\alpha_1,\ldots,\alpha_n\} \cup
    \res(K_0) \subseteq S \cup \ac(K_0)$.
  \end{claimproof}
  By the claim, $\trdeg(\langle \ac(K_1) \rangle/\langle \ac(K_0)
  \rangle) \le |S| = n + m \le \trdeg(K_1/K_0)$.
\end{proof}
\begin{lemma} \label{calculation}
  Let $M_0 = (K_0,k_0)$ be a small model, that is, a small elementary
  substructure of $\Mm = (K,k)$.  Suppose $\bar{a}$ is a finite tuple
  from $K \cup k$, and $(K_1,k_1)$ is $\acl(M_0 \bar{a})$.  Then
  \begin{align*}
    \dim(\bar{a}/M_0) &= \trdeg(K_1/K_0) + \trdeg(k_1/k_0) -
    \trdeg(\langle \ac(K_1) \rangle/k_0),
  \end{align*}
  and the transcendence degrees on the right hand side are finite.
\end{lemma}
\begin{proof}
  First note that $\langle \ac(K_0) \rangle = \langle k_0^\times
  \rangle = k_0$, because $M_0$ is a model.  Therefore the
  transcendence degree $\trdeg(\langle \ac(K_1) \rangle / k_0)$ makes
  sense.


  Without loss of generality, $a_1, \ldots, a_n \in K$ and $a_{n+1},
  \ldots, a_{n+m} \in k$.  Let $L = K_0(a_1,\ldots,a_n)$.  By
  Lemma~\ref{acl-char-2},
  \begin{itemize}
  \item $K_1$ is the relative algebraic closure of $L$.
  \item $k_1$ is the relative algebraic closure of $\ac(L) \cup
    \{a_{n+1},\ldots,a_{n+m}\}$.
  \end{itemize}
  Then $\trdeg(K_1/K_0) = \trdeg(L/K_0) \le n < \infty$.  By
  Lemma~\ref{abhyankaresque}, 
  \begin{equation*}
    \trdeg(k_1/k_0) \le m + \trdeg(\langle \ac(L) \rangle / k_0) \le m
    + \trdeg(L/K_0) \le m + n.
  \end{equation*}
  So the transcendence degrees are finite.

  Let $b_1, \ldots, b_p$ be a transcendence basis of $K_1$ over $K_0$.
  Let $c_1, \ldots, c_q$ be a transcendence basis of $k_1$ over
  $\langle \ac(K_1) \rangle$.  Let $F = K_0(b_1,\ldots,b_p)$.  Then
  $K_1/F$ is algebraic.  By Lemma~\ref{abhyankaresque}, $\langle
  \ac(K_1) \rangle$ is algebraic over $\langle \ac(F) \rangle$.
  Therefore $c_1, \ldots, c_q$ is a transcendence basis of $k_1$ over
  $\langle \ac(F) \rangle$.  By Lemma~\ref{acl-char-2}, $(K_1,k_1) =
  \acl(M_0 \bar{b} \bar{c})$.  Therefore $\bar{a}$ is interalgebraic
  with $\bar{b}\bar{c}$ over $M_0$, so $\dim(\bar{a}/M_0) =
  \dim(\bar{b}\bar{c}/M_0)$.

  We claim that $\bar{b}\bar{c}$ is acl-independent over $M_0$:
  \begin{itemize}
  \item Suppose, say, $b_1 \in \acl(M_0, b_2, \ldots, b_p, \bar{c})$.
    By Lemma~\ref{acl-char-2}, $b_1$ is field-theoretically algebraic
    over $K_0 \cup \{b_2,\ldots,b_p\}$, contradicting the algebraic
    independence of $\{b_1,\ldots,b_p\}$ over $K_0$.
  \item Suppose, say, $c_1 \in \acl(M_0, \bar{b}, c_2, \ldots,c_q)$.
    As above, let $F = K_0(b_1,\ldots,b_p)$.  By
    Lemma~\ref{acl-char-2}, $c_1$ is field-theoretically algebraic
    over $\ac(F) \cup \{c_2,\ldots,c_q\}$.  This contradicts the fact
    that $\{c_1,\ldots,c_q\}$ are algebraically independent over
    $\langle \ac(F) \rangle$.
  \end{itemize}
  So $\bar{b}\bar{c}$ is acl-independent over $M_0$ as claimed and
  then $\dim(\bar{b}\bar{c}/M_0)$ is $p+q$.  Therefore
  \begin{align*}
    \dim(\bar{a}/M_0) &= \dim(\bar{b}\bar{c}/M_0) = p +q \\
    &= \trdeg(K_1/K_0) + \trdeg(k_1/\langle \ac(K_1) \rangle) \\
    &= \trdeg(K_1/K_0) + \trdeg(k_1/k_0) - \trdeg(\langle \ac(K_1) \rangle/k_0). \qedhere
  \end{align*}
\end{proof}

\begin{lemma}
  Let $\bar{a}, \bar{b}$ be tuples and $M_0$ be a small model. Then
  $\dim(\bar{a}/M_0) \le \dim(\bar{a}\bar{b}/M_0)$.
\end{lemma}
\begin{proof}
  Let $(K_0,k_0) = M_0$.  Let $(K_1,k_1) = \acl(M_0\bar{a})$ and
  $(K_2,k_2) = \acl(M_0 \bar{a} \bar{b})$.  By
  Lemma~\ref{calculation},
  \begin{align*}
    \dim(\bar{a}\bar{b}/M_0) - \dim(\bar{a}/M_0) = & \trdeg(K_2/K_0) +
    \trdeg(k_2/k_0) - \trdeg(\langle \ac(K_2) \rangle/k_0) \\
    & - \trdeg(K_1/K_0) - \trdeg(k_1/k_0) +\trdeg(\langle \ac(K_1)
    \rangle / k_0) \\
    = & \trdeg(K_2/K_1) + \trdeg(k_2/k_1) - \trdeg(\langle \ac(K_2)
    \rangle / \langle \ac(K_1) \rangle).
  \end{align*}
  The final value is nonnegative: $\trdeg(k_2/k_1) \ge 0$ is clear, and
  Lemma~\ref{abhyankaresque} gives
  \begin{equation*}
    \trdeg(K_2/K_1) - \trdeg(\langle \ac(K_2) \rangle / \langle
    \ac(K_1) \rangle) \ge 0. \qedhere
  \end{equation*}
\end{proof}

\begin{lemma} \label{final-monotone}
  Let $\bar{a}, \bar{b}$ be tuples and $M_0$ be a small model.  If
  $\bar{a} \in \dcl(M_0\bar{b})$, then $\dim(\bar{a}/M_0) \le
  \dim(\bar{b}/M_0)$.
\end{lemma}
\begin{proof}
  Clearly $\bar{b}$ is interalgebraic with $\bar{a}\bar{b}$ over $M_0$, so
  \begin{equation*}
    \dim(\bar{a}/M_0) \le \dim(\bar{a}\bar{b}/M_0) =
    \dim(\bar{b}/M_0). \qedhere
  \end{equation*}
\end{proof}

\begin{proposition}
  $\RCVF_{\ac,2}$ has no space-filling functions.
\end{proposition}
\begin{proof}
  Suppose $f : X \to Y$ is a definable surjection with $\dim(X) <
  \dim(Y)$.  Let $M_0$ be a small model over which everything is
  defined.  Take $\bar{b} \in Y$
  with $\dim(\bar{b}/M_0) = \dim(Y)$.  Take $\bar{a} \in X$ with
  $f(\bar{a}) = \bar{b}$.  Then $\dim(\bar{a}/M_0) \le \dim(X) <
  \dim(Y) = \dim(\bar{b}/M_0)$, contradicting
  Lemma~\ref{final-monotone}.
\end{proof}
The conclusion is that the frontier dimension inequality
(Theorem~\ref{usual-frontier}) can fail, even under the
assumptions of DFC and NSFF.

\subsection{Arbitrarily bad behavior} \label{sec:ndep}
If $X$ is a set and $\tau_1, \ldots, \tau_n$ are topologies on $X$,
say that $\{\tau_1,\ldots,\tau_n\}$ are \emph{jointly independent} if
the diagonal embedding
\begin{equation*}
  X \mapsto (X,\tau_1) \times (X,\tau_2) \times \cdots \times (X,\tau_n)
\end{equation*}
has dense image.  Equivalently, if $U_i$ is a non-empty $\tau_i$-open
set for $1 \le i \le n$, then $\bigcap_{i = 1}^n U_i$ is non-empty.
For example, the approximation theorem for V-topologies \cite[Theorem~4.1]{prestel-ziegler} says that if $K$ is a field
and $\tau_1,\ldots,\tau_n$ are distinct (non-trivial) V-topologies on
$K$, then $\{\tau_1,\ldots,\tau_n\}$ are jointly independent.
\begin{definition}
  $T_n$ is the theory of two-sorted structures $(V,\Gamma;v_1,\ldots,v_n)$, where
  \begin{itemize}
  \item $V$ is a divisible ordered abelian group, with symbols $<$, $+$,
    $0$, and a unary function symbol for $q \cdot x$, for each $q \in
    \Qq$.
  \item $\Gamma$ is a divisible ordered abelian group, with similar
    symbols to $V$.
  \item For each $1 \le i \le n$, $v_i$ is a surjection $V \to \Gamma
    \cup \{+\infty\}$ satisfying the following axioms:
    \begin{itemize}
    \item $v_i(x) = +\infty \iff x = 0$.
    \item $v_i(x+y) \ge \min(v_i(x),v_i(y))$.
    \item $v_i(q \cdot x) = v_i(x)$, for $q \in \Qq^\times$.
    \end{itemize}
  \item The topologies $\tau_0, \ldots, \tau_n$ are jointly
    independent, where
    \begin{itemize}
    \item $\tau_0$ is the order topology on $V$.
    \item $\tau_i$ is the topology on $V$ induced by $v_i : V \to
      \Gamma$.
    \end{itemize}
  \end{itemize}
\end{definition}
We first show that $T_n$ is consistent.
\begin{lemma} \label{valued-field-joke}
  There is an ordered field $(K,+,\cdot,<)$ with valuations $v_1,
  \ldots, v_n$ such that the residue characteristics are 0, the value
  groups are non-trivial and divisible, and the order topology and valuation
  topologies are all independent.
\end{lemma}
\begin{proof}
  By compactness, it suffices to show that for any finite set of
  primes $S = \{p_1,\ldots,p_m\}$, there is an ordered field
  $(K,+,\cdot,<)$ with valuations $v_1,\ldots,v_n$ such that the
  residue characteristics are not in $S$, the value groups are
  divisible, and the order and valuation topologies are independent.
  Take $K$ to be the field of real algebraic numbers.  Take distinct
  primes $q_1, \ldots, q_n$ outside of $S$.  Let $v_i$ be any
  valuation on $K$ extending the $q_i$-adic valuation on $\Qq$.  Then
  $v_i$ is non-trivial, and has residue characteristic $q_i \notin S$.
  The $v_i$ are pairwise distinct, because they have different residue
  characteristics.  Rank-1 valuations can be recovered from their
  topologies, so the $v_i$ induce distinct V-topologies on $K$.  The
  order topology is also a V-topology, and it is distinct from the
  $v_i$-topology for any $i$.  Indeed, the way to recover a rank 1
  valuation ring $\Oo$ from its topology is via the fact that
  \begin{equation*}
    \Oo = K \setminus \{x \in K : \lim_{i \to \infty} x^{-i} = 0\}.
  \end{equation*}
  Applying this to the order topology, we see that $\Oo = [-1,1]$,
  which is not a ring.

  Because all the V-topologies are distinct, they are independent by
  the approximation theorem for V-topologies.  Finally, we show that
  the value groups of the $v_i$ are divisible.  This follows because
  real-closed fields have $k$th roots of all positive elements, and
  \begin{equation*}
    \frac{1}{k} v_i(x) = v_i\left(\sqrt[k]{|x|}\right). \qedhere
  \end{equation*}
\end{proof}
\begin{proposition}
  $T_n$ is consistent.
\end{proposition}
\begin{proof}
  Take a multi-valued ordered field $(K,+,\cdot,<,v_1,\ldots,v_n)$ as
  in Lemma~\ref{valued-field-joke}.  Let $\Gamma_i$ be the value group
  of $K$.  As ordered abelian groups, $\Gamma_i \equiv \Gamma_j$.
  Therefore, passing to a resplendent elementary extension, we may
  assume $\Gamma_i \cong \Gamma_j$ for all $i, j$.  Moving things by
  an isomorphism, we may assume $\Gamma_i = \Gamma_j =: \Gamma$ for
  all $i, j$.  Then $(K,\Gamma,v_1,\ldots,v_n) \models T_n$.  (The
  requirement that $v_i(qx) = v_i(x)$ for $q \in \Qq^\times$ holds
  because the residue characteristic is 0.)
\end{proof}
Next, we verify quantifier elimination.
\begin{lemma} \label{infinite-branch}
  Let $M = (V,\Gamma)$ be a model of $T_n$.  Suppose that $b_1,
  \ldots, b_m \in V$ and $c \in \Gamma$.  Suppose that $v_1(b_i - b_j)
  \ge c$ for all $i \le j \le m$.  Then there is $a \in V$ such that
  $v_1(a - b_i) = c$ for all $1 \le i \le m$.
\end{lemma}
\begin{proof}
  Let $G^+$ be the set $\{x \in V : v_1(x) \ge c\}$, and $G^-$ be the
  set $\{x \in V : v_1(x) > c\}$.  Then $G^+$ and $G^-$ are
  $\Qq$-linear subspaces of $V$.  There is at least one element
  $\delta$ with $v_1(\delta) = c$, so $G^+ \supsetneq G^-$.  Then the
  index of $G^-$ in $G^+$ is infinite, because otherwise $G^+/G^-$
  would be a finite non-trivial rational vector space.  By assumption,
  the $b_i$ are all in the same coset $C$ of $G^+$.  The coset $C$ is
  a union of infinitely many cosets of $G^-$.  There are only finitely
  many $b_i$, so we can take some $a \in C$ which is not in any $b_i +
  G^-$.  Then $a - b_i \in G^+ \setminus G^-$ for all $i$.
\end{proof}
\begin{lemma} \label{real-qe}
  Let $M_1 = (V_1,\Gamma_1)$ and $M_2 = (V_2,\Gamma_2)$ be two models
  of $T_n$.  Let $A = (V_0,\Gamma_0)$ be a common substructure of
  $M_1$ and $M_2$.  Suppose that $A \subsetneq M_1$, and $M_2$ is
  $|M_1|^+$-saturated.  Then there is a substructure $A \subsetneq B
  \subseteq M_1$ and an embedding of $B$ into $M_2$ extending the
  inclusion $A \hookrightarrow M_2$.
\end{lemma}
\begin{proof}
  First suppose $\Gamma_0 \subsetneq \Gamma_1$.  By quantifier
  elimination in the theory of divisible ordered abelian groups, there
  is an embedding $j : \Gamma_1 \to \Gamma_2$ fixing $\Gamma_0$.  Let
  $B = (V_0,\Gamma_1)$.  Then we can embed $B$ into $M_2$ by using the
  inclusion $V_0 \hookrightarrow V_2$ and the embedding $j : \Gamma_1
  \to \Gamma_2$.  As $\Gamma_0 \subsetneq \Gamma_1$, we have $A
  \subsetneq B$, and we are done.

  So we may assume $\Gamma_0 = \Gamma_1$.  Then $V_0 \subsetneq V_1$.
  Take some $a \in V_1 \setminus V_0$.  Let $x$ be a variable in the
  sort $V$ and let $\Sigma(x)$ be the quantifier-free type over $A$
  asserting the following:
  \begin{enumerate}
  \item $x < b$, for any $b \in V_0$ with $a < b$.
  \item $x > b$, for any $b \in V_0$ with $a > b$.
  \item $v_i(x - b) = c$, where $c = v_i(a - b)$, for each $i$ and
    each $b \in V_0$.
  \end{enumerate}
  The third point is a formula over $A$ because $v_i(a - b) \in \Gamma_1 = \Gamma_0$.
  \begin{claim}
    The partial type $\Sigma(x)$ is finitely satisfiable in $M_2$.
  \end{claim}
  \begin{claimproof}
    A finite fragment of $\Sigma(x)$ looks like $\phi_0(x) \wedge
    \phi_1(x) \wedge \cdots \wedge \phi_n(x)$, where
    \begin{itemize}
    \item $\phi_0(x)$ is a conjunction of formulas of the form $x <
      b$ or $x > b$.
    \item For $i > 0$, $\phi_i(x)$ is a conjunction of formulas of the
      form $v_i(x - b) = c$.
    \end{itemize}
    On $V_2$, $\phi_0(x)$ defines an open set with respect to the order
    topology and $\phi_i(x)$ defines an open set with respect to the
    $v_i$ topology.  Because $M_2$ is a model, these topologies are
    jointly independent.  So it suffices to show that each $\phi_i$ is
    individually satisfiable in $M_2$.

    First consider $\phi_0$.  It has the form \[\bigwedge_{i = 1}^m (x
    < b_i) \wedge \bigwedge_{i = 1}^\ell (x > c_i)\] for some $b_i,
    c_i \in V_0$.  The fact that this is satisfiable in $V_1$ (by $a$)
    implies that it is satisfiable in $V_2$, because of quantifier
    elimination for DLO.

    Next consider, say, $\phi_1$.  It has the form \[\bigwedge_{i =
      1}^m v_1(x - b_i) = c_i\] for some $b_i \in V_0$ and $c_i \in
    \Gamma_0 = \Gamma_1$.  Without loss of generality,
    \begin{equation*}
      c_1 = c_2 = \cdots = c_\ell > c_{\ell +1} \ge \cdots \ge c_m.
    \end{equation*}
    Let $c = c_1$.  Applying the triangle inequality to $a - b_i$, $a
    - b_j$, and $b_i - b_j$, we see that
    \begin{gather*}
      v_1(b_i - b_j) \ge c \text{ for } i, j \le \ell \\
      v_1(b_i - b_j) = c_j \text{ for } i \le \ell < j.
    \end{gather*}
    These facts only involve elements of $A$, so they transfer from
    $M_1$ to $A$ to $M_2$.  Applying Lemma~\ref{infinite-branch} to
    $M_2$, we get $a' \in V_2$ such that $v_1(a' - b_i) = c = c_i$ for all
    $i \le \ell$.  For $i > \ell$, we have
    \begin{equation*}
      v_1(a' - b_i) = v_1((a' - b_1) + (b_1 - b_i)) =
      \min(v_1(a'-b_1),v_1(b_1-b_i)) = \min(c,c_i) = c_i,
    \end{equation*}
    because $v_1(a'-b_1) = c \ne v_1(b_1 - b_i) = c_i$.  In
    conclusion, $v_1(a' - b_i) = c_i$ for all $i$, and $a'$ satisfies
    $\phi_1$.
  \end{claimproof}
  By saturation of $M_2$ and the claim, there is some $a' \in V_2$
  realizing $\Sigma(x)$.  Let $B = (V_0 +\Qq a, \Gamma_0)$ and define
  an embedding $f$ from $B$ to $M_2$ by sending $b + qa$ to $b + qa'$,
  for $b \in V_0$ and $q \in \Qq$.  This is well-defined as a function
  because $a$ is $\Qq$-linearly independent from $V_0$ (simply because
  $a \notin V_0$ and $V_0$ is a $\Qq$-linear subspace).  It remains to
  show that $f$ is an embedding.  It is clear that $f$ preserves the
  abelian group structure on the $V$-sort, and all the structure on
  the $\Gamma$-sort.  For the valuation map $v_i$, note that
  \begin{gather*}
    v_i(b + q a) = v_i(q^{-1} b + a) \stackrel{*}{=} v_i(q^{-1} b + a') = v_i(b + q
    a'), \text{ for } q \in \Qq^\times \\
    v_i(b + 0 a) = v_i(b) = v_i(b + 0 a'),
  \end{gather*}
  where the starred equality holds because $a'$ satisfies $\Sigma(x)$.
  Therefore $v_i(f(x)) = v_i(x)$ for all $x$, and $f$ preserves the
  symbol $v_i$.  For the order, note that
  \begin{itemize}
  \item If $q < q'$, then
    \begin{gather*}
      b + qa < b' +q'a \iff b - b' < (q' - q)a \iff (q'-q)^{-1}(b -
      b') < a \\ \stackrel{*}{\iff} (q' - q)^{-1}(b - b') < a' \iff
      \cdots \iff b + qa' < b' + q'a',
    \end{gather*}
    where the starred equivalence holds because $a'$ satisfies
    $\Sigma(x)$.
  \item If $q > q'$, then
    \begin{equation*}
      b + qa < b' + q'a \iff b + qa' < b' + q'a'
    \end{equation*}
    by a similar argument.
  \item If $q = q'$, then $b + q a < b' + q' a \iff b < b' \iff b + q a' < b' + q' a'$.
  \end{itemize}
  Either way, we see that $x < y \iff f(x) < f(y)$, and so $f$
  preserves the ordering.  Therefore $f$ is an embedding and we are
  done.
\end{proof}
\begin{proposition}
  $T_n$ has quantifier elimination.
\end{proposition}
\begin{proof}
  This follows formally from Lemma~\ref{real-qe} by standard tests for
  quantifier elimination; see \cite[Corollary~B.11.10]{transseries}
  for example.
\end{proof}
\begin{corollary}
  $T_n$ is complete.
\end{corollary}
\begin{proof}
  Let $M_1, M_2$ be two models.  Then $(\{0\},\{0\})$ is a common
  substructure of both.  By quantifier elimination, $M_1 \equiv M_2$.
\end{proof}
\begin{definition}
  Let $(V,\Gamma)$ be a model of $T_n$.  The \emph{combined topology} on
  $V$ is the topology generated by the order topology and the
  $v_i$-topology for each $i$.  Equivalently, the combined topology is
  the subspace topology induced by the diagonal embedding
  \begin{equation*}
    V \hookrightarrow (V,\tau_0) \times (V,\tau_1) \times \cdots \times (V,\tau_n),
  \end{equation*}
  where $\tau_0$ is the order topology and $\tau_i$ is the
  $v_i$-topology.
\end{definition}
\begin{remark} \label{loc-const}
  Each $v_i(x)$ determines a function $V \to \Gamma$.  With respect to
  the combined topology on $V$, this function is \emph{locally
  constant} except at $x = 0$.

  Additionally, let $f : V \to \{\top,\bot\}$ send $x$ to the
  truth-value of $x > 0$.  Then $f$ is locally constant except at $x =
  0$.
\end{remark}
\begin{proposition}
  $T_n$ is visceral with respect to the combined topology on $V$ and
  the order topology on $\Gamma$.
\end{proposition}
\begin{proof}
  First we verify that the topologies in question are definable.  The
  order topology on $\Gamma$ obviously has a definable basis
  consisting of open intervals $(a,b)$.  For the combined topology on
  $V$, a definable basis is given by sets of the form
  \begin{equation*}
    (a,b) \cap B^1_{c_1}(b_1) \cap \cdots \cap B^n_{c_n}(b_n)
  \end{equation*}
  where $(a,b)$ is an interval in the order, and $B^i_{c_i}(b_i)$ is
  an open $v_i$-ball $\{x \in V : v_i(x - b_i) > c_i\}$.

  This shows that the combined topology on $V$ and order topology on
  $\Gamma$ are definable as \emph{topologies}, and then
  Remark~\ref{group-uniformity-7} lifts the definable topologies to
  definable uniformities.  (The group operations on $V$ are continuous
  with respect to the combined topology.)

  It remains to show that the topologies are t-minimal.  For the order
  topology on $\Gamma$, it is easy to see by quantifier elimination
  that $\Gamma$ is o-minimal: every definable subset is a finite union
  of points and intervals.  The order is dense, so $\Gamma$ is
  t-minimal as desired.

  Now consider the combined topology on $V$.  The independence of the
  order topology and valuation topologies ensures that there are no
  isolated points.  Finally suppose we have a definable set $D
  \subseteq V$.  Let $\mathcal{L}(M)$ be the base language expanded by
  a constant symbol for every element of the model.  Then $D$ is
  defined by a quantifier-free $\mathcal{L}(M)$-formula $\phi(x)$,
  where $x$ is a single variable in the $V$-sort.  By induction on
  complexity, one can see the following:
  \begin{itemize}
  \item If $t(x)$ is an $\mathcal{L}(M)$-term in the $V$ sort, then
    $t$ is of the form $qx + b$, for some $q \in \Qq$ and $b \in V$.
  \item If $s(x)$ is an $\mathcal{L}(M)$-term in the $\Gamma$ sort,
    then the function $s : V \to \Gamma$ is locally constant off a
    finite set.  (This uses the first half of Remark~\ref{loc-const}.)
  \item If $\phi(x)$ is a quantifier-free $\mathcal{L}(M)$-formula,
    then the truth-value of $\phi$ is locally constant off a finite
    set.  (This uses the previous point, as well as the second half of
    Remark~\ref{loc-const}.)
  \end{itemize}
  Applying the third point to the quantifier-free formula $\phi(x)$,
  we see that the characteristic function of $D$ is locally constant
  off a finite set.  This exactly means that $\bd(D)$ is finite.
\end{proof}
\begin{remark}
  $T_n$ has definable finite choice, because of the linear orders on
  $V$ and $\Gamma$.
\end{remark}
\begin{proposition}
  Let $(V,\Gamma)$ be a monster model of $T_n$.
  \begin{enumerate}
  \item There is a definable surjection $f : X \to Y$ where $\dim(X) =
    1$ and $\dim(Y) = n$.
  \item There is a definable set $C \subseteq V \times \Gamma^{n-1}$
    such that $\dim(C) = 1$ and $\dim(\partial C) \ge n-1$.
  \end{enumerate}
\end{proposition}
\begin{proof}
  ~
  \begin{enumerate}
  \item Take $X = V \setminus \{0\}$, $Y = \Gamma^n$, and $f(x) =
    (v_1(x),\ldots,v_n(x))$.  The independence of the topologies
    ensures that for any $(\gamma_1,\ldots,\gamma_n) \in \Gamma^n$,
    there is some $x \in V$ with $f(x) = \bar{\gamma}$.
  \item Let $C$ be the graph of the function $g : V \setminus \{0\} \to \Gamma^{n-1}$
    given by
    \begin{equation*}
      f(x) = (v_2(x) - v_1(x), v_3(x) - v_1(x), \ldots, v_n(x) - v_1(x)).
    \end{equation*}
    Then $\dim(C) = \dim(V \setminus \{0\}) = 1$.
    \begin{claim}
      If $c_2, \ldots, c_n \in \Gamma$, then $(0,c_2,\ldots,c_n) \in
      \overline{C}$.
    \end{claim}
    \begin{claimproof}
      For any $0 < \epsilon \in V$ and any $\gamma \in \Gamma$,
      we can find
      \begin{itemize}
      \item $\gamma_1 \in \Gamma$ such that $\gamma_1 > \gamma$ and
        $\gamma_1 + c_i > \gamma$ for $i = 2, \ldots, n$.  (Take
        $\gamma_1 \gg 0$.)
      \item Non-zero $x \in V$ such that $0 < x < \epsilon$, $v_1(x) = \gamma_1$, and
        $v_i(x) = \gamma_1 + c_i$ for $i = 2, \ldots, n$.
      \end{itemize}
      Then $C$ contains the point $(x,f(x)) = (x,v_2(x) - v_1(x),
      \ldots) = (x, c_2, \ldots,c_n)$, where $0 < x < \epsilon$ and
      $v_i(x) > \gamma$ for all $i$.  Because $\epsilon$ and $\gamma$
      were arbitrary, we can make $x$ be arbitrarily close to 0 in the
      combined topology.  Therefore $(0,c_2,\ldots,c_n)$ is in the
      closure $\overline{C}$ as claimed.
    \end{claimproof}
    Then $\partial C$ contains $\{0\} \times \Gamma^{n-1}$, which has
    dimension $n - 1$. \qedhere
  \end{enumerate}
\end{proof}

\appendix

\section{Appendix: a strange matroid}\label{strange}
Let $\Mm$ be a monster model of a t-minimal theory.
\begin{proposition}\label{joke}
  Given a finite tuple $\ba \in \Mm^n$ and a set of parameters $C
  \subseteq \Mm^\eq$, the set of $\acl$-bases of $\ba$ over $C$ are
  the bases of a matroid.
\end{proposition}
\begin{proof}
  Say that a subtuple $\bb \subseteq \ba$ is \emph{spanning} if $\ba
  \in \acl(C\bb)$.  By Remark~\ref{acl-bases}(2), the minimal spanning
  sets are exactly the $\acl$-bases, so the spanning sets are exactly
  the sets which contain $\acl$-bases.  Say that a subtuple $\bb
  \subseteq \ba$ is \emph{co-independent} if the complementary tuple
  $\ba \setminus \bb$ is spanning.  The maximal co-independent sets
  are the complements of the $\acl$-bases.  By matroid duality, it
  suffices to show that the co-independent sets constitute (the
  independent sets of) a matroid.  We must prove the following:
  \begin{enumerate}
  \item The empty set is co-independent.
  \item Any subset of a co-independent set is co-independent.
  \item If $\bb$ is a subtuple of $\ba$, then any two maximal
    co-independent subtuples of $\bb$ have the same length.
  \end{enumerate}
  Taking complements, we must prove the following about spanning
  tuples:
  \begin{enumerate}
  \item $\ba$ is a spanning tuple.
  \item If $\bc \subseteq \bb \subseteq \ba$ and $\bc$ is spanning,
    then $\bb$ is spanning.
  \item If $\bc \subseteq \ba$, then any two minimal elements of \[
    \mathcal{F} = \{\bb \subseteq \ba : \bb \supseteq \bc \text{ and
      $\bb$ is spanning}\}\] have the same length.
  \end{enumerate}
  The first two points are clear.  The third point holds because
  \begin{equation*}
    \mathcal{F} = \{\bc \cup \be : \be \in \mathcal{F}'\},
  \end{equation*}
  where
  \begin{align*}
    \mathcal{F'} &= \{\be \subseteq \ba \setminus \bc : \ba \in \acl(C\bc\be)\} \\ &= \{\be \subseteq \ba \setminus \bc : \ba \setminus \bc \in \acl(C\bc\be)\}.
  \end{align*}
  The minimal elements of $\mathcal{F'}$ are the $\acl$-bases of $\ba
  \setminus \bc$ over $C\bc$, which have the same length by
  Remark~\ref{acl-bases}(1).
\end{proof}
The matroid in Proposition~\ref{joke} is counter-intuitive.  For example,
we are not claiming the following:
\begin{nontheorem} \label{joke-nt}
  In the matroid of Proposition~\ref{joke},
  \begin{enumerate}
  \item The closure operation is $\acl_C(-)$.
  \item The independent sets are the $\acl$-independent sets.
  \item The rank of a subtuple $\bb \subseteq \ba$ is $\dim(\bb/C)$.
  \end{enumerate}
\end{nontheorem}

\section{Appendix: guide to the Non-Theorems} \label{guide}
The reader can check the following implications between the exchange
property, NSFF, and the Non-Theorems (\ref{ddag}, \ref{nt-kappa},
\ref{surj-nt}, and \ref{joke-nt}):
\begin{equation*}
  \xymatrix{
    \ref{ddag} \ar@{=>}[r] & \ref{nt-kappa}(2) \ar@{=>}[dr] & \ref{nt-kappa}(1) \ar@{=>}[l] \ar@{=>}[r] & \ref{surj-nt} \ar@{=>}[r] & \mathrm{NSFF} \\
    \ref{joke-nt}(2) \ar@{=>}[u] & \ref{joke-nt}(3) \ar@{=>}[u] & \ref{nt-kappa}(3) \ar@{=>}[u] & \ref{joke-nt}(1) \ar@{=>}[r] & \text{Exchange Property} \ar@{=>}[u]}
\end{equation*}
In Non-Theorem~\ref{nt-kappa}, part (2) implies part (3) via
Proposition~\ref{dimprops}(\ref{dp5}), and part (3) implies (1) by
Proposition~\ref{dimprops}(\ref{dp2}).  Non-Theorem~\ref{surj-nt} implies
NSFF by Theorem~\ref{dimension-theorem}(\ref{dt4}).  The other
implications are straightforward or well-known.

As a consequence, if $T$ has a space-filling function, then every one
of the Non-Theorems must fail.

\begin{acknowledgment}
  The author was supported by the National Natural Science Foundation
  of China (Grant No.\@ 12101131), not to be confused with the NSFF,
  as well as the Ministry of Education of China (Grant No.\@
  22JJD110002).  The author benefitted from helpful communication with
  Alf Dolich and John Goodrick.  Qihang Jing found many typos in
  an earlier draft.
\end{acknowledgment}

\bibliographystyle{plain} \bibliography{mybib}{}
\end{document}